\pdfoutput=1
\documentclass[hidelinks, reqno]{amsart}
\usepackage[english]{babel}
\usepackage{graphicx}
\usepackage{subcaption}
\usepackage{tabularx}
\usepackage{booktabs}
\usepackage{array}
\usepackage{amsmath}
\usepackage{amsfonts}
\usepackage{amssymb}
\usepackage{amsthm}
\usepackage{appendix}
\usepackage{mathtools}
\usepackage{dsfont}
\usepackage[scr]{rsfso}
\usepackage{enumitem}
\usepackage{permute}
\usepackage{slashed}
\usepackage[dvipsnames]{xcolor}
\usepackage[pagebackref = true, colorlinks, linkcolor = Red, citecolor = Green, bookmarksdepth=2, linktocpage=true]{hyperref}
\usepackage[capitalize]{cleveref}
\usepackage{comment}
\usepackage{changepage}
\usepackage{nameref}
\usepackage{xspace}

\usepackage[T1]{fontenc}
\usepackage{makecell}

\usepackage{tikz}
\usetikzlibrary{calc}

\allowdisplaybreaks

\DeclareMathOperator{\interior}{int}

\DeclareMathOperator{\Id}{Id}
\DeclareMathOperator{\tr}{tr}

\DeclareMathOperator{\vol}{vol}
\DeclareMathOperator{\supp}{supp}

\DeclareMathOperator{\SL}{SL}

\DeclareMathOperator{\SO}{SO}

\DeclareMathOperator{\PSL}{PSL}

\DeclareMathOperator{\RP}{\mathbb{R}P}
\DeclareMathOperator{\CP}{\mathbb{C}P}

\DeclareMathOperator{\T}{T}

\DeclareMathOperator{\diag}{diag}

\DeclareMathOperator{\ad}{ad}
\DeclareMathOperator{\Ad}{Ad}

\DeclareMathOperator{\inj}{inj}

\DeclareMathOperator{\Fix}{Fix}

\newcommand{\N}{\mathbb{N}}
\newcommand{\Z}{\mathbb{Z}}

\newcommand{\R}{\mathbb{R}}
\newcommand{\C}{\mathbb{C}}

\newcommand{\LieG}{\mathfrak{g}}
\newcommand{\LieA}{\mathfrak{a}}

\newcommand{\LieU}{\mathfrak{u}}
\newcommand{\LieK}{\mathfrak{k}}
\newcommand{\LieM}{\mathfrak{m}}
\newcommand{\LieP}{\mathfrak{p}}
\newcommand{\LieQ}{\mathfrak{q}}

\newcommand{\LieSL}{\mathfrak{sl}}

\newcommand{\LieH}{\mathfrak{h}}

\newcommand{\rankG}{\mathsf{r}}
\newcommand{\rankH}{{\mathsf{r}_H}}

\newcommand{\Fboundary}{\mathcal{F}}

\newcommand{\involution}{\mathsf{i}}

\makeatletter
\newcommand{\mytag}[2]{%
\text{#1}%
\@bsphack
\begingroup
\@onelevel@sanitize\@currentlabelname
\edef\@currentlabelname{%
\expandafter\strip@period\@currentlabelname\relax.\relax\@@@%
}%
\protected@write\@auxout{}{%
\string\newlabel{#2}{%
{#1}%
{\thepage}%
{\@currentlabelname}%
{\@currentHref}{}%
}%
}%
\endgroup
\@esphack
}
\makeatother

\DeclareFontFamily{U}{mathb}{\hyphenchar\font45}
\DeclareFontShape{U}{mathb}{m}{n}{
<5> <6> <7> <8> <9> <10> gen * mathb
<10.95> mathb10 <12> <14.4> <17.28> <20.74> <24.88> mathb12
}{}
\DeclareSymbolFont{mathb}{U}{mathb}{m}{n}
\DeclareMathSymbol{\bigast}{2}{mathb}{"06}

\def\XXint#1#2#3{{\setbox0=\hbox{$#1{#2#3}{\int}$}
\vcenter{\hbox{$#2#3$}}\kern-.5\wd0}}

\theoremstyle{plain}
\newtheorem{theorem}{Theorem}[section]
\newtheorem{proposition}[theorem]{Proposition}
\newtheorem{lemma}[theorem]{Lemma}
\newtheorem{corollary}[theorem]{Corollary}

\theoremstyle{definition}
\newtheorem{definition}[theorem]{Definition}

\theoremstyle{remark}
\newtheorem{remark}[theorem]{Remark}
\newtheorem{notation}[theorem]{Notation}

\Crefname{enumi}{Property}{Properties}
\Crefname{alternativei}{Alternative}{Alternatives}
\Crefname{subsection}{Subsection}{Subsections}

\begin{document}
\selectlanguage{english}

\title[Finitary estimates for the distribution of lattice orbits]{Finitary estimates for the distribution of lattice orbits in homogeneous spaces I: Riemannian metric}

\author{Zuo Lin}
\address{Department of Mathematics, UC San Diego, 9500 Gilman Drive, La Jolla, CA 92093, USA}
\email{zul003@ucsd.edu}

\author{Pratyush Sarkar}
\address{Department of Mathematics, UC San Diego, 9500 Gilman Drive, La Jolla, CA 92093, USA}
\email{psarkar@ucsd.edu}

\date{\today}

\begin{abstract}
Let $H < G$ both be noncompact connected semisimple real algebraic groups where the former is maximal proper and $\Gamma < G$ be a lattice. Building on the work of Gorodnik--Weiss, we refine their techniques and obtain effective results. More precisely, we prove effective convergence of the distribution of dense $\Gamma$-orbits in $G/H$ to some limiting density on $G/H$ assuming effective equidistribution of regions of maximal horospherical orbits under one-parameter diagonal flows inside a dense $H$-orbit in $\Gamma \backslash G$. The significance of the effectivized argument is due to the recent effective equidistribution results of Lindenstrauss--Mohammadi--Wang for $\Delta(\SL_2(\R)) < \SL_2(\R) \times \SL_2(\R)$ and $\SL_2(\R) < \SL_2(\C)$ and arithmetic lattices $\Gamma$, and future generalizations in that direction.
\end{abstract}

\maketitle

\setcounter{tocdepth}{1}
\tableofcontents

\section{Introduction}
Let $G$ be a noncompact connected semisimple real algebraic group of rank $\rankG$ and $K < G$ be a maximal compact subgroup. Let $G$ be endowed with the left $G$-invariant and right $K$-invariant Riemannian metric induced by the Killing form. Let $\Gamma < G$ be a lattice. Let $H < G$ be a closed Lie subgroup. A central topic in homogeneous dynamics is the orbit counting problem. This involves investigating the asymptotic behavior, effectively if possible, of counting functions associated to discrete $\Gamma$-orbits in $G/H$. There has been extensive research on this problem especially on the whole group $G$ and the symmetric space $G/K$ where all $\Gamma$-orbits are discrete (see \cite{BO12} and references therein). However, there has been research to a lesser extent on non-discrete $\Gamma$-orbits.

Let us restrict our attention and let $H < G$ be a noncompact semisimple proper Lie subgroup and consider dense $\Gamma$-orbits. One has to now measure the size of balls in $\Gamma$ rather than in $G/H$ in order to make sense of the orbit counting problem. In another perspective giving more information, one can study the limiting distribution of dense $\Gamma$-orbits according to some natural limiting process. A major work in this direction was done by Gorodnik--Weiss \cite{GW07}. Let us recount their main theorem. We denote the measure induced by the Riemannian metric on $G$ on any space $X$ by $\mu_X$. We normalize $\mu_{\Gamma \backslash G}$ to a probability measure $\hat{\mu}_{\Gamma \backslash G}$. We denote by $B_r^X(x) \subset X$ the open ball of radius $r > 0$ centered at $x \in X$ with respect to the metric induced by the Riemannian metric on $G$. We write $H_T := H \cap KB_T^G(e)K$ and $\Gamma_T := \Gamma \cap KB_T^G(e)K$ for all $T > 0$. See \cref{sec:Preliminaries} for more details. For all $y_0 \in G/H$, there is a canonical measure $\nu_{y_0} \ll \mu_{G/H}$ as defined in \cite[Eq. (12) and Proposition 5.1]{GW07} (see \cref{sec:LimitingDensity}) which we normalize as $\hat{\nu}_{y_0} := \mu_{\Gamma \backslash G}(\Gamma \backslash G)^{-1}\nu_{y_0}$. In \cite{GW07}, they proved that $\hat{\nu}_{y_0}$ is the limiting density of the orbit $\Gamma y_0 \subset G/H$ whenever it is dense in the following sense.

\begin{theorem}[{\cite[Theorem 1.1]{GW07}}]
\label{thm:GWTheorem}
Let $y_0 \in G/H$ such that $\Gamma y_0 \subset G/H$ is dense. Then, for all $\psi \in C_{\mathrm{c}}(G/H)$, we have
\begin{align*}
\lim_{T \to +\infty}\frac{1}{\mu_H(H_T)} \sum_{\gamma \in \Gamma_T} \psi(\gamma y_0) = \int_{G/H} \psi \, d\hat{\nu}_{y_0}.
\end{align*}
\end{theorem}

\begin{remark}
Their original theorem is actually for connected simple $H$ and their more general theorems are for connected semisimple $H$ satisfying the so-called \emph{balanced} property. However, it was overlooked that when using the Riemannian metric for skew balls, any $H$ as above is balanced.
\end{remark}

The above theorem can be viewed as a kind of ergodic theorem for a nonamenable group $\Gamma$ which is in general difficult to obtain. We mention here some other related works \cite{Oh05,GO07,GN14}.

Let us restrict our attention further and let $H < G$ be a noncompact semisimple maximal proper Lie subgroup. Note that as a consequence of Ratner's theorem (conjectured by Raghunathan and Dani), we have the dichotomy that $\Gamma$-orbits in $G/H$ are either discrete or dense. The main objective of this paper is to study the dense $\Gamma$-orbits in $G/H$ in an \emph{effective} fashion.

We first present our results in two concrete settings where the theorems are unconditional. We then turn to the general setting where theorems are (at the moment) conditional on a hypothesis. In the introduction we prefer to write the simplified versions of the theorems. The full versions are stated in \cref{sec:GeneralTheorems}.

\subsection{\texorpdfstring{Special theorems: counting oriented circles in $\mathbb{S}^2 \cong \CP^1$ and $\mathbb T^2 \cong \RP^1 \times \RP^1$}{Special theorems: counting oriented circles in \unichar{"1D54A}² ≅ ℂP¹ and \unichar{"1D54B}² ≅ ℝP¹×ℝP¹}}
\label{sec:Example}
In this subsection, let
\begin{align*}
(G, H) &= (\SL_2(\C), \SL_2(\R)) \qquad \text{or} \qquad (G, H) &= (\SL_2(\R) \times \SL_2(\R), \Delta(\SL_2(\R))),
\end{align*}
and $\Gamma < G$ be an arithmetic lattice. Considering the standard actions $\SL_2(\C) \curvearrowright \CP^1$ or $\SL_2(\R) \times \SL_2(\R) \curvearrowright \RP^1 \times \RP^1$, the $G$-space $G/H$ can be naturally identified with the $G$-space $\mathcal{C}$ of oriented circles in $\mathbb{S}^2$ or $\mathbb T^2 \cong \RP^1 \times \RP^1$, respectively. Fix any Riemannian metric on $\mathcal{C}$ and denote by $C^{0, \chi}(\mathcal{C})$ the corresponding space of $\chi$-H\"{o}lder continuous real-valued functions on $\mathcal{C}$. Our first theorem is an orbit counting result with a power saving error term for $(G, H) = (\SL_2(\C), \SL_2(\R))$. It is a special case of \cref{thm:MainTheoremCountingCircles} below, greatly simplifying the statement.

\begin{theorem}
\label{thm:MainTheoremNoClosedOrbitCountingCircles}
Let $\Gamma < \SL_2(\mathbb{C})$ be an arithmetic lattice such that $\Gamma \backslash \SL_2(\mathbb{C})$ contains no periodic $\SL_2(\R)$-orbits. Then, there exists $\kappa \in \bigl(0, \frac{1}{2}\bigr)$ such that for all $\chi \in (0, 1]$, $\psi \in C_{\mathrm{c}}^{0, \chi}(\mathcal{C})$, oriented circles $\mathsf{c} \in \mathcal{C}$, and $T > 0$, we have
\begin{align*}
\sum_{\gamma \in \Gamma_T} \psi(\gamma \mathsf{c})
= 16\pi^2\hat{\nu}_{\mathsf{c}}(\psi) e^{\frac{1}{2}T} + O_{\psi, \mathsf{c}}\bigl(e^{(\frac{1}{2} - \chi\kappa)T}\bigr).
\end{align*}
\end{theorem}

\begin{remark}
The hypothesis that $\Gamma \backslash \SL_2(\mathbb{C})$ contains no periodic $\SL_2(\R)$-orbits is justified due to theorems of Maclachlan--Reid \cite[Corollary 6]{MR87} and Reid \cite{Rei91} which provide infinitely many such arithmetic lattices $\Gamma < \SL_2(\mathbb{C})$. We refer the reader to \cite[Chapter 5, \S 5.3 and Chapter 9, \S 9.5]{MR03} for details.
\end{remark}

More generally, we have \cref{thm:MainTheoremCountingCircles} where we need to take into account the periodic $H$-orbits. It follows by combining \cref{thm:LMW_For_SL2C,thm:MainTheoremSimplified}. For all $\psi \in C_{\mathrm{c}}^{0, \chi}(\mathcal{C})$, we define a corresponding constant
\begin{align*}
D_\psi := \inf\bigl\{r > 0: \supp(\psi) \subset B_r^G(e) \cdot H \subset G/H \cong \mathcal{C}\bigr\} > 0.
\end{align*}

\begin{theorem}
\label{thm:MainTheoremCountingCircles}
Suppose that
\begin{align*}
(G, H) &=
\begin{cases}
(\SL_2(\C), \SL_2(\R)) & \text{or} \\
(\SL_2(\R) \times \SL_2(\R), \Delta(\SL_2(\R))) & \text{resp.}
\end{cases}
\end{align*}
and $\Gamma < G$ is an arithmetic lattice. There exists $\kappa > 0$ such that the following holds. Let $\chi \in (0, 1]$, $\psi \in C_{\mathrm{c}}^{0, \chi}(\mathcal{C})$, $g_0 \in G$, $x_0 = \Gamma g_0 \in \Gamma \backslash G$, and $\mathsf{c} = g_0 H \in G/H \cong \mathcal{C}$. There exists $M_{\psi, g_0} > 0$ such that for all $R \gg_{\Gamma, g_0, \chi} 1$ and $T \gg_\Gamma \log(R) + M_{\psi, g_0}$, at least one of the following holds.
\begin{enumerate}
\item We have
\begin{align*}
\sum_{\gamma \in \Gamma_T} \psi(\gamma \mathsf{c}) = 16\pi^2\hat{\nu}_{\mathsf{c}}(\psi)e^{\frac{1}{2}T} +
\begin{cases}
O_{D_\psi}\bigl(\|\psi\|_{C^{0, \chi}} R^{-\chi\kappa}e^{\frac{1}{2}T}\bigr) & \text{or} \\
O_{D_\psi}\bigl(\|\psi\|_{C^{0, \chi}} \bigl(T^{-\frac{1}{17}}e^{\frac{1}{2}T} + R^{-\chi\kappa}e^{\frac{1}{2}T}\bigr)\bigr) & \text{resp.}
\end{cases}
\end{align*}
\item There exists $x \in \Gamma \backslash G$ with
\begin{align*}
d(x_0, x) \leq e^{-\frac{1}{8}T}
\end{align*}
such that $xH$ is periodic with $\vol(xH) \leq R$.
\end{enumerate}
\end{theorem}

\begin{remark}
In \cref{thm:MainTheoremCountingCircles}, we have also used the volume formula from \cref{thm:SkewBallVolumeAsymptotic} to obtain a more explicit asymptotic orbit counting formula in both cases of $(G, H)$. In the first case, $H = \SL_2(\R)$ and we recall that $K_H = \SO(2) \cong \mathbb{S}^1$ and $M_H = \{I, -I\} < K_H$. Taking the generator $J =
\bigl(
\begin{smallmatrix}
0 & 1 \\
-1 & 0
\end{smallmatrix}
\bigr)
\in
\LieK
$, we parametrize $K_H = \{e^{tJ}: 0 \leq t \leq 2\pi\}$. We calculate using the \emph{Killing form of $\LieG$} that $\|J\| = 4$. Thus, $\mu_{K_H}(K_H) = 2\pi\|J\| = 8\pi$ and $\mu_{M_H}(M_H) = 2$ which gives the coefficient $\frac{\mu_{K_H}(K_H)^2}{2\mu_{M_H}(M_H)} = 16\pi^2$. Similarly, we calculate using the \emph{Killing form of $\LieG$} that $\delta_{2\rho_H} = \frac{1}{2}$. Calculations for the second case is similar and yield identical numbers.
\end{remark}

\subsection{General theorem: simplified version}
Let us return to the general setting. Recall that $H < G$ is a noncompact semisimple maximal proper Lie subgroup. Let $A_H < H$ be a maximal subgroup consisting of semisimple elements and $\LieA_H$ be its Lie algebra endowed with the inner product and norm induced by the Killing form. Let $\Phi_H^+$ be a corresponding choice of a set of positive roots and $\LieA_H^+ \subset \LieA_H$ be the corresponding closed positive Weyl chamber. We write $a_v := \exp(v)$ for all $v \in \LieA_H$. We will consider the flow on $\Gamma \backslash G$ given by the right translation action of $\{a_{tv}\}_{t \in \R}$ for various unit vectors $v \in \interior(\LieA_H^+)$. Let $U_H < H$ be a maximal expanding horospherical subgroup. We normalize $\hat{\mu}_{U_H} := \mu_{U_H}\bigl(B_1^{U_H}(e)\bigr)^{-1}\mu_{U_H}$. Denote $\inj_{\Gamma \backslash G}(x_0) > 0$ for the injectivity radius at $x_0 \in \Gamma \backslash G$. We denote by $\mathcal{S}^\ell(\cdot)$ the $L^2$ Sobolev norm of order $\ell \in \N$. See \cref{sec:Preliminaries} for more details. We make the following fundamental hypothesis, which can be regarded as an effective version of Shah's theorem \cite[Theorem 1.4]{Sha96}, throughout the paper. We call it \emph{Hypothesis Shah Effective Equidistribution}, or \emph{Hypothesis Shah-EE} for short.

\newtheoremstyle{named}{}{}{\itshape}{}{\bfseries}{.}{ }{#1 \thmnote{#3}}
\theoremstyle{named}
\newtheorem{namedhypothesis}{Hypothesis}
\newcommand{\ShahEE}{Hypothesis~\nameref{hyp:EffectiveEquidistribution}\xspace}
\begin{namedhypothesis}[Shah-EE]
\label{hyp:EffectiveEquidistribution}
There exist $\kappa_0 > 0$, $\varrho_0 > 0$, a decreasing family $\{c_\varsigma\}_{\varsigma > 0} \subset \R_{>0}$, and $\ell \in \mathbb N$ such that for all $x_0 \in \Gamma \backslash G$, $R \gg_{G, \Gamma} \inj_{\Gamma \backslash G}(x_0)^{-\varrho_0}$, $v \in \interior(\LieA_H^+)$ with $\|v\| = 1$ and $\varsigma := \min_{\alpha \in \Phi_H^+} \alpha(v)$, and $t \geq c_\varsigma\log(R)$, at least one of the following holds.
\begin{enumerate}
\item For all $\phi \in C_{\mathrm{c}}^\infty(\Gamma \backslash G)$, we have
\begin{align*}
\left|\int_{B_1^{U_H}(e)} \phi(x_0 u a_{tv}) \, d\hat{\mu}_{U_H}(u) - \int_{\Gamma \backslash G} \phi \, d\hat{\mu}_{\Gamma \backslash G} \right| \leq \mathcal{S}^\ell(\phi) R^{-\varsigma\kappa_0}.
\end{align*}
\item There exists $x \in \Gamma \backslash G$ with
\begin{align*}
d(x_0, x) \leq R^{c_\varsigma} t^{c_\varsigma} e^{-\varsigma t}
\end{align*}
such that $xH$ is periodic with $\vol(xH) \leq R$.
\end{enumerate}
The constants $\kappa_0$, $\varrho_0$, and $\{c_\varsigma\}_{\varsigma > 0}$ depend only on $(G, H, \Gamma)$, and $\ell$ depends only on $\dim(G)$.
\end{namedhypothesis}

\begin{remark}
The explicit dependence on the injectivity radius in the condition $R \gg_{G, \Gamma} \inj_{\Gamma \backslash G}(x_0)^{-\varrho_0}$ in the above hypothesis is as in \cite[Eq. (14.2)]{LMW22} (though it is not explicitly included in \cite[Theorem 1.1]{LMW22}). This explicit formula is required in the proof of \cref{thm:K_EffectiveEquidistribution}. See also \cref{rem:DependenceOfInjectivityRadiusInProof}.
\end{remark}

\begin{remark}
There are only finitely many periodic $H$-orbits $xH$ with $\vol(xH) \leq R$; see \cite[Theorem 5.1]{DM93}. For a quantitative results, see \cite[Theorem 5]{SS22} and \cite{EMV09}. See also \cite[Corollary 10.7 and Remark 10.11]{MO23} for its generalization to geometrically finite $3$-dimensional hyperbolic manifolds. 
\end{remark}

\begin{notation}
Throughout the paper, whenever we assume that \ShahEE holds, we keep the same notation for the various constants without further comments. Also, we take $\ell \in \N$ provided by the Sobolev embedding theorem and hence depends only on $\dim(G)$.
\end{notation}

The main objective of this paper is to study the dense $\Gamma$-orbits in $G/H$ in an \emph{effective} fashion, provided we have the effective equidistribution in $\Gamma \backslash G$ of regions of maximal horospherical orbits under one-parameter diagonal flows inside a dense $H$-orbit in $\Gamma \backslash G$. More precisely, the main theorem we prove is the following effective version of \cref{thm:GWTheorem} whenever \ShahEE holds. Here, we fix any Riemannian metric on $G/H$ and denote by $C^{0, \chi}(G/H)$ the corresponding space of $\chi$-H\"{o}lder continuous real-valued functions on $G/H$. For all $\psi \in C_{\mathrm{c}}^{0, \chi}(G/H)$, we define a corresponding constant
\begin{align*}
D_\psi := \inf\bigl\{r > 0: \supp(\psi) \subset B_r^G(e) \cdot H \subset G/H\bigr\} > 0.
\end{align*}
We also define $\varsigma_H := \frac{1}{2}\min_{\alpha \in \Phi_H^+} \alpha(v_{2\rho_H}) > 0$ where $v_{2\rho_H} = \nabla\rho_H/\|\nabla\rho_H\|$ and $\rho_H$ is half the sum of positive roots in $\Phi_H^+$ with multiplicity (see \cref{sec:Preliminaries}).

\begin{theorem}
\label{thm:MainTheoremSimplified}
Suppose \ShahEE holds. There exists $\kappa > 0$ such that the following holds. Let $\chi \in (0, 1]$, $\psi \in C_{\mathrm{c}}^{0, \chi}(G/H)$, $g_0 \in G$, $x_0 = \Gamma g_0 \in \Gamma \backslash G$, and $y_0 = g_0 H \in G/H$. There exists $M_{\psi, g_0} > 0$ such that for all $R \gg_{G, \Gamma, g_0, \chi} 1$ and $T \gg_{G, H, \Gamma} \log(R) + M_{\psi, g_0}$, at least one of the following holds.
\begin{enumerate}
\item We have:
\begin{enumerate}
\item if $\rankG = 1$, then
\begin{align*}
\left|\frac{1}{\mu_H(H_T)} \sum_{\gamma \in \Gamma_T} \psi(\gamma y_0) - \int_{G/H} \psi \, d\hat{\nu}_{y_0} \right| \ll_{D_\psi} \|\psi\|_{C^{0, \chi}} R^{-\chi\kappa};
\end{align*}
\item if $\rankG \geq 2$, then
\begin{align*}
\left|\frac{1}{\mu_H(H_T)} \sum_{\gamma \in \Gamma_T} \psi(\gamma y_0) - \int_{G/H} \psi \, d\hat{\nu}_{y_0} \right| \ll_{D_\psi} \|\psi\|_{C^{0, \chi}} \bigl(T^{-\frac{1}{2\dim(G) + 5}} + R^{-\chi\kappa}\bigr).
\end{align*}
\end{enumerate}
\item There exists $x \in \Gamma \backslash G$ with
\begin{align*}
d(x_0, x) \leq e^{-\frac{\varsigma_H}{2}T}
\end{align*}
such that $xH$ is periodic with $\vol(xH) \leq R$.
\end{enumerate}
\end{theorem}

The above theorem is a simplified version of \cref{thm:MainTheoremEpsilon}. The full versions of the main theorem can be found in \cref{sec:GeneralTheorems} with many more details on the error terms, and various constants and their estimates. In fact, a lot of work was done to obtain such precise information.

\subsection{\texorpdfstring{On \ShahEE}{On Hypothesis~\nameref{hyp:EffectiveEquidistribution}}}
Thanks to breakthroughs involving many authors, \ShahEE is known to hold in some cases. Of course, this is really the motivation for introducing \ShahEE which conjecturally holds in general.

The first known instances of $(G, H, \Gamma)$ for which \ShahEE holds is due to the recent work of Lindenstrauss--Mohammadi--Wang \cite{LMW22} where they proved the following. We denote by $\Delta: \SL_2(\R) \to \SL_2(\R) \times \SL_2(\R)$ the diagonal embedding.

\begin{theorem}[{\cite[Theorem 1.1]{LMW22}}]
\label{thm:LMW_For_SL2C}
Suppose that either
\begin{align*}
(G, H) &= (\SL_2(\C), \SL_2(\R)) & &\text{or} & (G, H) &= (\SL_2(\R) \times \SL_2(\R), \Delta(\SL_2(\R))),
\end{align*}
and $\Gamma < G$ is an arithmetic lattice. Then, \ShahEE holds.
\end{theorem}

\begin{remark}
More generally, the results of Lindenstrauss--Mohammadi--Wang hold when $\Gamma$ is a lattice with algebraic entries.
\end{remark}

Even more recently, Lindenstrauss--Mohammadi--Wang--Yang also prove the following.

\begin{theorem}[{\cite[Theorem 1.3]{LMWY23}}]
Suppose that
\begin{align*}
(G, H) = (\SL_3(\R), \SO_{Q}(\R)^\circ)
\end{align*}
where we take the quadratic form $Q(x_1, x_2, x_3) = x_2^2 - 2x_1x_3$, and $\Gamma < \SL_3(\R)$ is any lattice. Then, \ShahEE holds.
\end{theorem}

\begin{remark}
Recall that a lattice $\Gamma < \SL_3(\R)$ is automatically arithmetic by Margulis’ arithmeticity theorem. Also, observe that $\SO_Q(\R)^\circ \cong \SO(2, 1)^\circ \cong \PSL_2(\R)$.
\end{remark}

\subsection{\texorpdfstring{Outline of the proof of \cref{thm:MainTheorem}}{Outline of the proof of Theorem~\ref{thm:MainTheorem}}}\label{subsec:Outline}
The proof of the main theorem is composed of five major parts and one minor part. The five major parts in totality amount to showing that the ``$\Gamma$-average'' is asymptotic (in an effective fashion) to the ``$G$-average'': for all $y_0 \in G/H$, we have
\begin{align}
\label{eqn:GammaAverageGAverage}
\biggl|\frac{1}{\mu_H(H_T)} \sum_{\gamma \in \Gamma_T} \phi(\gamma y_0) - \frac{1}{\mu_H(H_T)\mu_{\Gamma \backslash G}(\Gamma \backslash G)} \int_{G_T} \phi(g y_0) \, d\mu_G(g)\biggr| \to 0
\end{align}
with an explicit error term, provided a necessary technical assumption on avoidance of periodic $H$-orbits is satisfied. The minor part relates the ``$G$-average'' to a limiting density: there exists a canonical $\nu_{y_0} \ll \mu_{G/H}$ such that
\begin{align*}
\biggl|\frac{1}{\mu_H(H_T)\mu_{\Gamma \backslash G}(\Gamma \backslash G)} \int_{G_T} \phi(g y_0) \, d\mu_G(g) - \int_{G/H} \phi \, d\hat{\nu}_{y_0}\biggr| \to 0
\end{align*}
again with an explicit error term. The minor part is not difficult once we have precise asymptotic formulas for the volume of the so-called Riemannian skew balls which is also used in the other parts.

The five major parts correspond to \cref{sec:EstimatesForTheBusemannFunction,sec:VolumeCalculations,sec:K-Equidistribution,sec:EquidistributionOfRiemannianSkewBalls,sec:EffectiveDuality} which take up the bulk of the paper. Let us outline these parts below, not necessarily in a linear order. We often compare with the noneffective arguments of Gorodnik--Weiss \cite{GW07} and Shah \cite{Sha96}, and for simplicity, we defer mentioning the technical assumption on avoidance of periodic $H$-orbits and how we carry it through.

\begin{enumerate}[label=Part \arabic*.]
\item Let us recall a part of the argument in \cite{GW07} to prove (the noneffective) \cref{eqn:GammaAverageGAverage}. This requires using the duality between $\Gamma \backslash G$ and $G/H$. One first shows that the sum in \cref{eqn:GammaAverageGAverage} is asymptotic to an average of an associated function $\phi$ on $\Gamma \backslash G$ over a Riemannian skew ball of $H$ in $G$. Next, one uses the equidistribution of Riemannian skew balls to show that this average is asymptotic to a (normalized) integral of $\phi$ over $\Gamma \backslash G$. Finally, one shows that the integral in \cref{eqn:GammaAverageGAverage} is also asymptotic to the same (normalized) integral of $\phi$ over $\Gamma \backslash G$.

In the effectivized argument, we need to use the effectivized version of equidistribution of Riemannian skew balls not only in the second step of the above argument but also for the \emph{error term} in the first step of the argument. We also need to use the precise asymptotic formulas for the volume of Riemannian skew balls to deal with the error terms.

\item We found in Part~1 that we need effective equidistribution of Riemannian skew balls. For the sake of simplicity, let us ignore the complication due to the ``skewness'' (though it is the main difficulty) and assume that we are dealing with the usual Riemannian balls $H_T$. In this part, we then need to prove:
\begin{align*}
\left| \frac{1}{\mu_{H}(H_T)} \int_{H_T} \phi(x_0 h) \, d\mu_H(h) - \int_{\Gamma \backslash G} \phi \, d\hat{\mu}_{\Gamma \backslash G} \right| \to 0
\end{align*}
with an explicit error term. One can use an integral formula associated to the Cartan decomposition $H = K_H A_H^+ K_H$ to write
\begin{align}
\label{eqn:H_TVolumeFormula}
\begin{aligned}
&\int_{H_T} \phi(x_0 h) \, d\mu_H(h) \\
={}&\frac{1}{\mu_{M_H}(M_H)}\int_{K_H} \int_{(\LieA_H^+)_T} \int_{K_H} \phi(x_0 k_1 a_v k_2) \xi_H(v) \, dv \, d\mu_{K_H}(k_1) \, d\mu_{K_H}(k_2).
\end{aligned}
\end{align}
The $\LieA_H^+$-coordinate gives the radial component and the first $K_H$-coordinate gives the angular component, both measured in the locally symmetric space $\Gamma \backslash G/K$. For example, if $(G,H) = (\SL_2(\C), \SL_2(\R))$, the locally symmetric space is a $3$-dimensional hyperbolic manifold and the $H_T$-orbit is an immersed $2$-dimensional hyperbolic ball. It turns out, as in hyperbolic manifolds, that the volume of an $H_T$-orbit is concentrated near its boundary. Moreover, in the $\LieA_H^+$-coordinate, the volume of an $H_T$-orbit is also concentrated near $\R_{> 0} v_{2\rho_H}$ where $v_{2\rho_H} \in \interior(\LieA_H^+)$ is the maximal growth direction of the sum of positive roots $2\rho_H$ of $\LieA_H$. Thus, instead of the equidistribution of $H_T$ on the left hand side of \cref{eqn:H_TVolumeFormula}, one can focus on that of ``a sector of a Riemannian annulus''. Transferring to the right hand side of \cref{eqn:H_TVolumeFormula}, we need the equidistribution of $K_H$-orbits under one-parameter diagonal flows along directions near $v_{2\rho_H}$. In the noneffective argument in \cite{GW07}, the latter is provided by \cite[Corollary 1.2]{Sha96}).

The full effective argument is very technical and heavily relies on the precise asymptotic formulas for the volume of Riemannian skew balls.

\item In this part, we prove effective equidistribution of $K_H$-orbits, which is needed in Part~2, \emph{assuming} \ShahEE regarding effective equidistribution of $U_H$-orbits (the effective version of \cite[Theorem 1.4]{Sha96}). A little more precisely, we prove
\begin{align*}
\left|\int_{K_H} \phi(x_0 k a_{tv}) \varphi(k) \, d\mu_{K_H}(k) - \int_{\Gamma \backslash G} \phi \, d\hat{\mu}_{\Gamma \backslash G} \cdot \int_{K_H} \varphi \, d\mu_{K_H} \right| \to 0
\end{align*}
with an explicit error term, provided that we have
\begin{align*}
\left|\int_{B_1^{U_H}(e)} \phi(x_0 u a_{tv}) \, d\hat{\mu}_{U_H}(u) - \int_{\Gamma \backslash G} \phi \, d\hat{\mu}_{\Gamma \backslash G} \right| \to 0
\end{align*}
with an explicit error term. The proof of this part relies on the fact that $K_H$-orbits can be approximated by small pieces of $U_H$-orbits. The geometric picture of this for $(G,H) = (\SL_{2}(\C), \SL_{2}(\R))$ is that large hyperbolic circles can be approximated by a collection of small pieces of horocycles. The noneffective version of the argument appears in \cite{Sha96} which is the passage from \cite[Theorem 1.4]{Sha96} to \cite[Corollary 1.2]{Sha96}. The effectivized argument for the special case $(G, H) = (\SL_2(\R) \times \SL_2(\R), \Delta(\SL_2(\R)))$ appears in \cite[Theorem 1.4]{LMW23}.

A key ingredient in the \emph{general} effectivized argument is an orthogonal decomposition $\LieK_H = \LieK_H^\star \oplus \LieM_H$ where $\LieK_H$ is the Lie algebra of $K_H$ and $\LieM_H$ is the Lie algbera of $M_H = Z_{K_H}(A_H)$. We use the fact that the orthogonal complement $\LieK_H^\star$ is naturally isomorphic to $\LieU_H$, the Lie algebra of $U_H$, as vector spaces. We also use related estimates for maps coming from the local product structure $K_H \subset H \approx U_H^+U_H^-A_HM_H$.

\item In this part, we seek to develop the necessary asymptotic formulas for the volume of Riemannian skew balls. This is absolutely necessary for Parts~1 and 2 and the minor part mentioned above. First, we stick to the usual Riemannian balls for simplicity. As in \cite{GW07}, we begin the proof by using the volume formula
\begin{align*}
\mu_H(H_T) &= \frac{1}{\mu_{M_H}(M_H)}\int_{K_H} \int_{(\LieA_H^+)_T} \int_{K_H} \xi_H(v) \, d\mu_{K_H}(k_1) \, dv \, d\mu_{K_H}(k_2) \\
&= \frac{\mu_{K_H}(K_H)^2}{\mu_{M_H}(M_H)} \int_{(\LieA_H^+)_T} \xi_H(v) \, dv
\end{align*}
associated to the Cartan decomposition $H = K_H A_H^+ K_H$. Here, approximately $\xi_H(v) \approx e^{2\rho_H(v)}$ and hence we are lead to investigate the precise asymptotic formulas for general integrals of the form $\int_{V_T} e^{\lambda(v)} \, dv$ for a normed vector space $(V, \|\cdot \|)$ and a nonzero linear form $\lambda \in V^*$, strengthening \cite[Theorem 9.3]{GW07}.

To deal with nontrivial Riemannian skew balls, the above tools are not enough due to ``varying radius'' and the fact that the above integral is also concentrated near the boundary. In this case $(\LieA_H^+)_T$ is replaced with a ``skew ball in $\LieA_H^+$'' denoted by $(\LieA_H^+)_T[g_1, g_2]$ for which precise asymptotic formulas are intractable to calculate directly. Instead, we do the following. We first restrict the integral to a cone $\mathcal{C}_\tau^{\LieA_H} \subset \interior(\LieA_H^+)$ of size $\tau > 0$ containing $\R_{> 0} v_{2\rho_H}$ due to the fact that the above integral is concentrated near $\R_{> 0} v_{2\rho_H}$. Our region is then $\mathcal{C}_\tau^{\LieA_H} \cap (\LieA_H^+)_T[g_1, g_2]$ which can then be approximated by sandwiching it between the intersection of the cone with two balls of similar radius:
\begin{align*}
\mathcal{C}_\tau^{\LieA_H} \cap (\LieA_H^+)_{T + T_{g_1, g_2} - \delta} \subset \mathcal{C}_\tau^{\LieA_H} \cap (\LieA_H^+)_T[g_1, g_2] \subset \mathcal{C}_\tau^{\LieA_H} \cap (\LieA_H^+)_{T + T_{g_1, g_2} + \delta}
\end{align*}
where $T_{g_1, g_2}$ is some constant depending on $g_1$ and $g_2$. Applying the techniques from above to the sandwiching balls give asymptotic formulas for the integral over the intersection of the balls with the cone while the intersection of the balls with the complement of the cone give error terms---the larger the cone, the smaller the error terms. However, this procedure is very delicate since there is also \emph{multiplicative} error coming from the approximate balls which is a function of $\delta$---the larger the cone, the \emph{larger} the multiplicative error. It turns out that $T_{g_1, g_2} \pm \delta$ is closely related to the Busemann function and to precisely approximate $\delta$ as a function of $\tau$ requires precise estimates for the Busemann function which is the purpose of Part~5. Due to the opposing forces for the aformentioned errors, we need to very carefully shrink $\tau$ as a function of $T$ at an appropriate rate and keep track of the error terms.

\item In this part, we develop the necessary Lie theoretic tools and find precise estimates for the Busemann function for a general semisimple Lie group. The overall idea is to use definitions and then use asymptotics for expressions such as $d(o, a_{tv}uo)$ for unit vectors in the Weyl chamber $v \in \LieA^+$ and $u \in U$. For $v \in \interior(\LieA^+)$, simply by triangle inequality, $d(o, a_{tv}uo)$ is asymptotic to $tv$ with an exponential error term $d(o, a_{tv}ua_{-tv} o) = O(e^{-\eta t})$ for some $\eta > 0$, for small $u$. However, in higher rank, it is possible to have $v \in \partial\LieA^+$ in the walls of the Weyl chamber in which case, the behavior can be different. For example, if $\log(u)$ is in an appropriate root space, then $a_{tv}$ and $u$ commute. In this case the behavior is Euclidean and $d(o, a_{tv}uo)$ is asymptotic to $tv$ with an error term $O(t^{-1})$. Deriving these types of asymptotics requires studying the Cartan projection of such commuting elements. Also, we actually need to deal with a \emph{neighborhood} of $\partial\LieA^+$ \emph{uniformly} so that certain constant coefficients are uniform---the issue is that the exponential rate $\eta > 0$ goes to $0$ as $v$ goes to the walls of the Weyl chamber $\partial\LieA^+$.
\end{enumerate}

Let us say a few words about carrying through the condition on avoidance of periodic $H$-orbits. This is also a delicate business and is essential for each of the arguments in Parts~1-3 above. For Part~3, avoidance of periodic $H$-orbits for a point $x_0 \in \Gamma \backslash G$ implies the same for \emph{all} points in the $K_H$-orbit $x_0K_H$ by right $K$-invariance of the metric. Miraculously, there is an appropriate scale for $Z$ so that this can be upgraded to all points in the translate $x_0K_Ha_{Zv}$ while \emph{simultaneously} ensuring that the estimates in the argument for Part~3 hold. For Part~2, we have a similar situation where all the arguments work once we choose an appropriate scale for the Riemannian skew annulus. The argument for Part~1 is similar.

\subsection{\texorpdfstring{Big $O$, $\Omega$, and Vinogradov notations}{Big O, Ω, and Vinogradov notations}}
Throughout the paper, we often use the big $O$, $\Omega$, and Vinogradov notations to write inequalities succinctly and manipulate them efficiently. For any functions $f: \R \to \R$ and $g: \R \to \R_{> 0}$ (or quantities where $f$ is implicitly a function of $g$), we write $f = O(g)$ to mean that there exists an implicit constant $C > 0$ such that $|f| \leq Cg$. We say $f(x) = O(g(x))$ as $x \to \pm\infty$ if the previous inequality holds for $x$ sufficiently positively/negatively large. We similarly say $f(x) = O(g(x))$ as $x \to 0$. It will also be convenient for us to write $f = \Omega(g)$ to mean that there exists an implicit constant $C > 0$ such that $f \geq Cg$, and similar variants as above. Note that this is the notation of Knuth and not Hardy--Littlewood. We often simply use the symbols $O(g)$ and $\Omega(g)$ in an expression to stand for such types of quantities. We also write $f \ll g$ and $g \gg f$ which is equivalent to $f = O(g)$. If $f \ll g$ and $f \gg g$, then we write $f \asymp g$. For a normed vector space $(V, \|\cdot\|)$, we also use these symbols in the natural way for $V$-valued functions or quantities. We put subscripts on $O$, $\Omega$, $\ll$, $\gg$, and $\asymp$ to indicate other quantities which the implicit constant may depend on.

\begin{remark}
Throughout the paper, except in the introduction, we view $(G, H, \Gamma)$ as fixed and also view implicit constants depending on those groups or induced spaces as absolute. Consequently, we often omit writing those groups or induced spaces in the subscript of $O$, $\Omega$, $\ll$, or $\gg$. However, the dependence of the implicit constants on $(G, H, \Gamma)$ should be clear if one wishes to trace it in the proofs.
\end{remark}

\subsection{Organization of the paper}
In \cref{sec:Preliminaries}, we first give the necessary background for the rest of the paper. We then develop various tools regarding the Busemann function and Riemannian skew balls in \cref{sec:EstimatesForTheBusemannFunction,sec:EquidistributionOfRiemannianSkewBalls}. In \cref{sec:K-Equidistribution}, we focus on the effective equidistribution hypothesis for regions of maximal horospherical orbits in a $H$-orbit and use it to derive a similar effective equidistribution result for orbits of the maximal compact subgroup of $H$. We then use that to derive effective equidistribution of Riemannian skew balls in \cref{sec:EquidistributionOfRiemannianSkewBalls}. In \cref{sec:EffectiveDuality}, we use duality of double quotient spaces to relate the $\Gamma$-orbit count to the $G$-orbit integral in an effective fashion. In \cref{sec:LimitingDensity}, we relate the $G$-orbit integral to the known limiting density in an effective fashion. Finally, in \cref{sec:GeneralTheorems}, we assemble all the pieces and state our general theorems with full details.

\subsection*{Acknowledgements}
We thank Amir Mohammadi for suggesting this problem, explaining his work, and many other useful conversations to troubleshoot technical difficulties. We thank Hee Oh for references to prior related results and to related volume formulas. We finally thank Ralf Spatzier for the earlier reference to Knieper's volume formula.

\section{Preliminaries}
\label{sec:Preliminaries}
Let $G$ be a noncompact connected semisimple real algebraic group, i.e., a Lie group which is the identity component of the group of real points of a semisimple linear algebraic group defined over $\R$. Let $\LieG = \T_e(G)$ be its Lie algebra. We similarly use corresponding Fraktur letters for the Lie algebras of other Lie groups throughout the paper. Let $B: \LieG \times \LieG \to \R$ be the Killing form. Let $\theta: \LieG \to \LieG$ be a Cartan involution, i.e., the symmetric bilinear form $B_\theta: \LieG \times \LieG \to \R$ defined by $B_\theta(x, y) = -B(x, \theta(y))$ for all $x, y \in \LieG$ is positive definite. Then we have the decomposition $\LieG = \LieK \oplus \LieP$ into the eigenspaces of $\theta$ corresponding to the eigenvalues $+1$ and $-1$ respectively. Let $K < G$ be the maximal compact subgroup whose Lie algebra is $\mathfrak{k}$. Let $\LieA \subset \LieP$ be a maximal abelian subalgebra and $\Phi \subset \LieA^*$ be the associated restricted root system. Let $\Phi^\pm \subset \Phi$ be sets of positive and negative roots with respect to some lexicographic order on $\LieA^*$ and $\Pi \subset \Phi^+$ be the set of simple roots. We similarly use superscripts $\pm$ for other root systems as long as the order is clear. We can identify $\LieA \cong \LieA^*$ via the Killing form. Let $\LieA^+ \subset \LieA$ be the corresponding closed positive Weyl chamber. Then, we have the restricted root space decomposition
\begin{align*}
\LieG = \LieA \oplus \LieM \oplus \LieU^+ \oplus \LieU^- = \LieA \oplus \LieM \oplus \bigoplus_{\alpha \in \Phi} \LieG_\alpha
\end{align*}
where $\LieM = Z_{\LieK}(\LieA) \subset \LieK$ and $\LieU^\pm = \bigoplus_{\alpha \in \Phi^+} \LieG_{\mp\alpha}$. Define the Lie subgroups
\begin{align}
\label{eqn:SubgroupsOfG}
A &= \exp(\LieA) < G, & U^\pm &= \exp(\LieU^\pm) < G, & M = Z_K(A) < K < G.
\end{align}
Note that the latter need not be connected. The first subgroup in \cref{eqn:SubgroupsOfG} is a maximal $\R$-split torus of $G$ and $\rankG := \dim(\LieA)$ is the rank of $G$. Define the closed subset $A^+ = \exp(\LieA^+) \subset A$. Denote
\begin{align*}
a_v = \exp(v) \in A \qquad \text{for all $v \in \LieA$}.
\end{align*}
The middle two subgroups in \cref{eqn:SubgroupsOfG} are the maximal expanding and contracting horospherical subgroups, i.e.,
\begin{align*}
U^\pm = \Bigl\{u^\pm \in G: \lim_{t \to \pm\infty} a_{tv}u^\pm a_{-tv} = e\Bigr\}
\end{align*}
for any $v \in \interior(\LieA^+)$. We often denote $U := U^+$. We denote the corresponding minimal parabolic subgroups by $P^\pm = MAU^\pm$. Recall the following useful decompositions of $G$ using the above subgroups: the Cartan decomposition $G = KA^+K$; the Iwasawa decompositions $G = KAU^\pm$; the Bruhat decompositions $G = \bigsqcup_{w \in N_K(A)/M} P^\pm w P^\pm$ where $N_K(A)/M$ is the Weyl group, which also gives the dense subgroup $MAU^+U^- \subset G$. The Furstenberg boundary of $G$ is $\Fboundary = G/P^- \cong K/M$ where we have used the Iwasawa decomposition.

Let $w_0 \in K$ be a representative of the element in the Weyl group $N_K(A)/M$ such that $\Ad_{w_0}(\LieA^+) = -\LieA^+$.  We also use the notation $g^+ = gP^- \in \Fboundary$ and $g^- = gw_0P^- \in \Fboundary$ for all $g \in G$. We define the \emph{opposition involution} $\involution = -\Ad_{w_0}: \LieA \to \LieA$ so that $\involution(\LieA^+) = \LieA^+$ and $\involution^2 = \Id_{\LieA}$. Note that we have the property
\begin{align*}
a_{-v} = w_0a_{-\Ad_{w_0}(v)}w_0^{-1} = w_0a_{\involution(v)}w_0^{-1} \qquad \text{for all $v \in \LieA^+$}.
\end{align*}

We fix the left $G$-invariant and right $K$-invariant Riemannian metric on $G$ induced by $B_\theta$ and denote the corresponding inner product and norm on any of its tangent spaces by $\langle \cdot, \cdot \rangle$ and $\|\cdot\|$ respectively. We use the same notations for the induced inner products and norms for any of its induced spaces. In particular, we have an inner product (which is just the Killing form $B$) and norm on $\LieA$ which is invariant under the Weyl group $N_K(A)/M$. Identifying $\LieA^* \cong \LieA$ using the inner product on $\LieA$ gives an inner product and norm on $\LieA^*$. Note that this norm coincides with the norm of the gradient and the operator norm:
\begin{align*}
\|\alpha\| = \|\nabla \alpha\| = \|\alpha\|_{\mathrm{op}} = \sup_{v \in \LieA, \|v\| = 1} \alpha(v) \qquad \text{for all $\alpha \in \LieA^*$}.
\end{align*}
In the $\rankG = 1$ case, we fix the constant
\begin{align}
\label{eqn:Eta1}
\eta_1 := \|\alpha\| > 0
\end{align}
for the simple root $\alpha \in \Pi$. We have that $\theta$ is an orthogonal involution and the decomposition $\LieG = \LieK \oplus \LieP$ is orthogonal, both with respect to $\langle \cdot, \cdot \rangle = B_\theta$. The restricted root space decomposition is also orthogonal with respect to $\langle \cdot, \cdot \rangle = B_\theta$ by \cite[Chapter VI, \S 4, Proposition 6.40]{Kna02}. The Riemannian symmetric space associated to $G$ is $G/K$. We fix the reference point $o = K \in G/K$.

We denote by $d_G$ the metric on $G$. We denote by $d_X$ the metric on any induced space $X$ obtained from the Riemannian metric on $X$. Henceforth, we will drop the subscript only for $G$, $\Gamma \backslash G$, and $G/K$ for brevity. We denote by $B_r^X(x) \subset X$ the open ball of radius $r > 0$ centered at $x \in X$, and more generally, we denote by $B_r^X(S) \subset X$ the open $r$-neighborhood of the subset $S \subset X$.

Let us characterize the metric on $G/K$. Recall that for any $g \in G$, if $g = k_1a_vk_2 \in KA^+K$ is its Cartan decomposition, then $v \in \LieA^+$ is unique, called the Cartan projection of $g$, and we have the distance $d(o, go) = \|v\|$. Applying the adjoint representation gives a singular value decomposition $\Ad_g = \Ad_{k_1}\Ad_{a_v}\Ad_{k_2}$ and hence we get
\begin{align}
\label{eqn:DistanceFormulaByOperatorNorm}
d(o, go) = \|v\| \asymp \log(\sigma_1) = \log\|\Ad_g\|_{\mathrm{op}}
\end{align}
where $\sigma_1 = \max_{\alpha \in \Phi^+} e^{\alpha(v)}$ is the maximum singular value of $\Ad_g$.

Let $(X, d)$ be any metric space and let $C(X) := C(X, \R)$ be the space of continuous real-valued functions on $X$. Since, we will only work with real-valued functions in this paper, we will drop the codomain $\R$ in all kinds of function spaces. For all $\phi \in C(X)$, the $\chi$-H\"{o}lder seminorm and norm of $\phi$ are defined by
\begin{align*}
|\phi|_{C^{0, \chi}} &= \sup_{d(x, y) \leq 1} \frac{|\phi(x) - \phi(y)|}{d(x, y)^{\chi}}, & \|\phi\|_{C^{0, \chi}} &= \|\phi\|_{\infty} + |\phi|_{C^{0, \chi}}.
\end{align*}
Define the Banach space of $\chi$-H\"{o}lder continuous functions $C^{0, \chi}(X) = \{\phi \in C(X): \|\phi\|_{C^{0, \chi}} < \infty\}$. If $X$ is a Riemannian manifold and $\phi \in C^\infty(X)$, the $L^2$ Sobolev norm of order $\ell \in \N$ is defined by
\begin{align*}
\mathcal{S}^\ell(\phi) = \left(\sum_{j = 0}^\ell \int_X |\nabla^j \phi|^2 \, d\vol\right)^{\frac{1}{2}}
\end{align*}
where $\nabla$ is the Levi-Civita connection and $\vol$ is the volume form.

We denote the Haar measure on $G$ compatible with the fixed Riemannian metric on $G$ by $\mu_G$, and the induced measure on any induced space $X$ by $\mu_X$. We often denote the Lebesgue measure $\mu_\LieA$ simply by $dv$. We also use similar notations for the Lebesgue measure on other subspaces of $\LieG$. Note that $\mu_M$ and $\mu_K$ are not normalized and not necessarily probability measures. For convenience, we normalize $\mu_{\Gamma \backslash G}$ to a probability measure $\hat{\mu}_{\Gamma \backslash G}$ and $\hat{\mu}_{U_H} := \mu_{U_H}\bigl(B_1^{U_H}(e)\bigr)^{-1}\mu_{U_H}$. Using the Cartan decomposition, we have the following integral formula with respect to $\mu_G$ (see \cite[Chapter I, \S 5, Theorem 5.8]{Hel00} and also its proof for the correct constant coefficient). For all $\alpha \in \Phi^+$, we denote its multiplicity by $m_\alpha := \dim(\LieG_\alpha)$. We denote their sum by $m_{\Phi^+} := \sum_{\alpha \in \Phi^+} m_\alpha$. Denote by $\rho := \frac{1}{2} \sum_{\alpha \in \Phi^+} m_\alpha \alpha$ half the sum of positive roots with multiplicity. Define the function $\xi: \LieA^+ \to \R$ by
\begin{align*}
\xi(v) = \prod_{\alpha \in \Phi^+} \sinh^{m_\alpha}(\alpha(v)) \qquad \text{for all $v \in \LieA^+$}.
\end{align*}
Then, we have
\begin{align}
\label{eqn:IntegralFormulaForG}
\int_G f \, d\mu_G = \frac{1}{\mu_M(M)}\int_K \int_{\LieA^+} \int_K f(k_1 a_v k_2) \cdot \xi(v) \, d\mu_K(k_1) \, dv \, d\mu_K(k_2)
\end{align}
for all $f \in C_{\mathrm{c}}(G)$.

\subsection{Busemann function}
\label{subsec:BusemannFunction}
We also need the Busemann function for the symmetric space $G/K$ of arbitrary rank $\rankG$. The Iwasawa decomposition actually gives a $C^\infty$-diffeomorphism $K \times \LieA \times U^- \rightarrow G$ defined by $(k,v,u) \mapsto ka_vu$. We use this to make the following definitions.

\begin{definition}[Iwasawa cocycle]
The \emph{Iwasawa cocycle} $\sigma: G \times \Fboundary \rightarrow \LieA$ gives the unique element $\sigma(g,\xi)$ such that $gk \in Ka_{\sigma(g,\xi)}U^-$ and satisfies the cocycle relation $\sigma(gh,\xi) = \sigma(g,h\xi) + \sigma(h,\xi)$ for all $g, h \in G$ and $\xi = kM \in \Fboundary \cong K/M$.
\end{definition}

\begin{definition}[Busemann function]
The \emph{Busemann function} $\beta: \Fboundary \times G/K \times G/K \to \LieA$ is defined by
\begin{align*}
\beta_\xi(x, y) = \sigma(g^{-1}, \xi) - \sigma(h^{-1}, \xi)
\end{align*}
for all $\xi \in \Fboundary$, $x = go \in G/K$, and $y = ho \in G/K$.
\end{definition}

For all $\xi \in \Fboundary$, and $x = go, y, z \in G/K$, and $h \in G$, the Busemann function satisfies the properties
\begin{enumerate}
\item $\beta_\xi(x, o) = \sigma(g^{-1}, \xi)$,
\item $\beta_{h\xi}(hx, hy) = \beta_\xi(x, y)$,
\item $\beta_\xi(x, z) = \beta_\xi(x, y) + \beta_\xi(y, z)$
\end{enumerate}
which are derived from the properties of the Iwasawa cocycle. We have the following geometric characterization of the Busemann function. For all $\xi = kM \in \Fboundary \cong K/M$, and $x, y \in G/K$, and $v \in \LieA^+$ with $\|v\| = 1$, we have
\begin{align*}
\langle \beta_\xi(x, y), v\rangle = \lim_{t \to +\infty} (d(x, \xi_t) - d(y, \xi_t))
\end{align*}
where $\xi_{\boldsymbol{\cdot}}: \R \to G/K$ is any geodesic in the direction $v$ whose forward limit point is $\xi$, say, $\xi_t = k a_{tv}o \in G/K$ for all $t \in \R$. In fact, we can prove a more precise version given in \cref{pro:PreciseGeometricBusemannFunction}.

\subsection{\texorpdfstring{The subgroup $H < G$}{The subgroup H < G}}
\label{subsec:TheSubgroupH}
Let $H < G$ be a noncompact semisimple maximal proper Lie subgroup of rank $\rankH \leq \rankG$. In the subsequent paragraphs, we discuss some related properties. We also give reasons why this class of Lie subgroups is natural and how to find numerous examples.

Recall that a maximal proper Lie subgroup of a semisimple Lie group is either a maximal parabolic subgroup or a maximal reductive subgroup \cite[Chapter VIII, \S 10, Corollary 1]{Bou05}. Thus, as long as we avoid parabolic subgroups, it is not hard for a maximal proper Lie subgroup to be semisimple. Indeed, such subgroups have been completely classified in the works of Dynkin \cite{Dyn51,Dyn52a,Dyn52b} and Malcev \cite{Mal44}. However, non-semisimple reductive maximal proper Lie subgroups may also exist (see \cite[Chapter 6, \S 1.6, Theorem 1.9]{VGO90}, \cite{Mos61}, and \cite{Hm66}).

Many examples of the desired $H < G$ come from the class of symmetric subgroups. We provide the necessary background here and refer the reader to Loos' series of books \cite{Loo69a,Loo69b} and \cite{Sch84} for further details about such subgroups. A Lie subgroup $H < G$ is said to be symmetric if $H^\circ < H < H^\Sigma$ where the latter is the (necessarily closed) subgroup of fixed points of an involutive automorphism $\Sigma: G \to G$. In terms of Lie algebras, $H < G$ is any (necessarily closed) subgroup corresponding to the Lie subalgebra obtained as the eigenspace of the involutive automorphism $\sigma := (d\Sigma)_e: \LieG \to \LieG$ corresponding to the eigenvalue $+1$. In fact, there is a generalized Cartan decomposition $\LieG = \LieH \oplus \LieQ$ where $\LieQ$ is the eigenspace of $\sigma$ corresponding to the eigenvalue $-1$. The associated homogeneous space $G/H$ is called an affine symmetric space. Berger \cite{Ber57} called them \emph{irreducible} whenever $(\ad_\LieH, \LieQ)$ is an irreducible representation. It is well-known that any symmetric subgroup $H < G$ is reductive \cite[Lemma E]{Koh65}, ruling out the possibility of being maximal parabolic by the above discussion. We now address the question of maximality. Clearly, $H < G$ is maximal proper if and only if $\LieH \subset \LieG$ is a maximal proper Lie subalgebra and hence $H = H^\Sigma$. The latter is simply a condition to be imposed so we focus on the former. We have the following proposition which is essentially known in the literature (see \cite[Chapter IV, Corollary 2]{Loo69a}, \cite[Appendix]{Bor95}, and \cite[Proposition 5.2(2)]{Sir08}).

\begin{proposition}
\label{pro:MaximalityAndIrreducibility}
The Lie subalgebra $\LieH \subset \LieG$ is maximal proper if and only if $(\ad_\LieH, \LieQ)$ is an irreducible representation.
\end{proposition}

In \cite{Ber57}, Berger completely classified affine symmetric spaces and compiled them in \cite[Tableau II]{Ber57}. His table indicates when certain properties, including irreducibility, are satisfied. Thus, in light of \cref{pro:MaximalityAndIrreducibility}, one can verify the maximality condition for $H < G$ simply by checking the irreducibility condition in \cite[Tableau II]{Ber57} which turns out to be satisfied for all but a handful of exceptions. Moreover, of the remaining instances, many are noncompact semisimple as well; note that we certainly need to omit some which are maximal compact which come from the Cartan involution and we also need to omit a few which are non-semisimple reductive. Alternatively, Borel \cite[Appendix]{Bor95} gave an abstract proof of the following supposedly well-known theorem in the theory of affine symmetric spaces.

\begin{theorem}
Suppose $\LieG$ is simple and $\LieH$ is semisimple. Then, the Lie subalgebra $\LieH \subset \LieG$ is maximal proper.
\end{theorem}

In light of the above discussion, symmetric subgroups provide a large number of examples of the desired $H < G$, where the associated homogeneous space $G/H$ is an affine symmetric space.

Now, in the above fashion as with $G$, we have objects associated to the noncompact semisimple Lie group $H$. We will decorate them by a subscript $H$ to distinguish them from the objects associated to $G$ unless otherwise stated. We can also guarantee a few additional properties. Since $H < G$ is semisimple, it is unimodular and hence there exists an induced left $G$-invariant measure $\mu_{G/H}$ on $G/H$ such that
\begin{align*}
d\mu_G = d\mu_H \, d\mu_{G/H}.
\end{align*}
Also, there exists a simultaneous Cartan decomposition $H = K_HA_H^+K_H$ for a compact subgroup $K_H < H$ and a $\R$-split torus $A_H < H$, i.e., it satisfies
\begin{align*}
A_H &< A, & A_H^+ &\subset A^+, & K_H &< K
\end{align*}
(see \cite[Theorem 6]{Mos55} and \cite{Kar53}).

As in the introduction and in \cite{GW07}, we introduce a type of open subsets of $H$, called Riemannian skew balls of $H$ in $G$. For all $g_1, g_2 \in G$ and $T > 0$, we define
\begin{align*}
H_T[g_1, g_2] := H \cap g_1^{-1} KB_T^G(e)K g_2^{-1} = \{h \in H: d(o, g_1hg_2o) < T\}.
\end{align*}
This generalizes the Riemannian balls of $H$ centered at $e$ defined using the Riemannian metric on $G/K$ which we simply denote by
\begin{align*}
H_T := H_T[e, e] = H \cap KB_T^G(e)K = \{h \in H: d(o, ho) < T\} \qquad \text{for all $T > 0$}.
\end{align*}
We call sets of the form $H_{T_2}[g_1, g_2] \setminus \overline{H_{T_1}[g_1, g_2]}$ for any $g_1, g_2 \in G$ and $T_2 > T_1 > 0$ a Riemannian skew annulus.

As it will be used extensively, let us record the integral formula with respect to $\mu_H$. In accordance with the above convention, let $m_{H, \alpha} := \dim(\LieH_\alpha)$ be the multiplicity for all $\alpha \in \Phi_H^+$, $m_{\Phi_H^+} := \sum_{\alpha \in \Phi_H^+} m_{H, \alpha}$ be their sum, and $\rho_H = \frac{1}{2} \sum_{\alpha \in \Phi_H^+} m_{H, \alpha} \alpha$ half the sum of positive roots with multiplicity. Define the function $\xi_H: \LieA_H^+ \to \R$ by
\begin{align*}
\xi_H(v) = \prod_{\alpha \in \Phi_H^+} \sinh^{m_{H, \alpha}}(\alpha(v)) \qquad \text{for all $v \in \LieA_H^+$}.
\end{align*}
Then, according to \cref{eqn:IntegralFormulaForG}, we have
\begin{align}
\label{eqn:IntegralFormulaForH}
\int_H f \, d\mu_H = \frac{1}{\mu_{M_H}(M_H)}\int_{K_H} \int_{\LieA_H^+} \int_{K_H} f(k_1 a_v k_2) \cdot \xi_H(v) \, d\mu_{K_H}(k_1) \, dv \, d\mu_{K_H}(k_2)
\end{align}
for all $f \in C_{\mathrm{c}}(H)$.

\subsection{Lattices}
Finally, we introduce the primary object of study in this paper. Let $\Gamma < G$ be a lattice, i.e., $\Gamma \backslash G$ has a finite right $G$-invariant measure. Namely, the Haar measure $\mu_G$ on $G$ descends to a such a measure $\mu_{\Gamma \backslash G}$ on $\Gamma \backslash G$. We normalize it to a probability measure $\hat{\mu}_{\Gamma \backslash G}$. As explained in the introduction, our main objective is to study dense $\Gamma$-orbits in $G/H$ in an effective fashion.

\section{Estimates for the Busemann function}
\label{sec:EstimatesForTheBusemannFunction}
We prove \cref{pro:PreciseGeometricBusemannFunction} which gives fundamental estimates for the Busemann function (see \cref{subsec:BusemannFunction} for the definition). \Cref{cor:PreciseGeometricBusemannFunction} is a simplified statement which follows immediately from the proposition. This is required in \cref{sec:VolumeCalculations} where we derive precise asymptotic formulas for the volume of Riemannian skew balls of $H$ in $G$. We prepare with many lemmas and propositions which provide useful Lie theoretic bounds and fundamental Lie theoretic identities.

\subsection{Lie theoretic bounds and identities}
We fix the positive constant
\begin{align}
\label{eqn:Constant_cPhi}
c_{\Phi} = \max_{\alpha \in \Phi} \|\alpha\| > 0
\end{align}
for the rest of the paper. Recall that the induced norm on $\LieA^*$ coincides with the operator norm: $\|\alpha\| = \|\alpha\|_{\mathrm{op}}$ for all $\alpha \in \LieA^*$. Recall also that $\|\Ad_g\|_{\mathrm{op}} = \sup_{w \in \LieG, \|w\| = 1} \|\Ad_g(w)\|$ and $\|\Ad_g^{-1}\|_{\mathrm{op}}^{-1} = \inf_{w \in \LieG, \|w\| = 1} \|\Ad_g(w)\|$.

\begin{lemma}
\label{lem:NormOfAdjointEstimate}
For all $g \in G$, we have
\begin{align*}
\|\Ad_g\|_{\mathrm{op}} &\leq e^{c_{\Phi}d(o, go)}, & \|\Ad_g^{-1}\|_{\mathrm{op}}^{-1} &\geq e^{-c_{\Phi}d(o, go)}.
\end{align*}
\end{lemma}

\begin{proof}
Since the Riemannian metric on $G$ is left $G$-invariant and right $K$-invariant, $\Ad_K$ acts orthogonally with respect to $\langle \cdot, \cdot \rangle = B_\theta$ on $\LieG$. Therefore, $\|\Ad_k\|_{\mathrm{op}} = 1$ for all $k \in K$.

Let $g \in G$. It suffices to prove only the first inequality since the second follows from the first applied to $g^{-1}$. Using the Cartan decomposition, we write $g = k_1 a_v k_2$ for some $k_1, k_2 \in K$ and a unique $v \in \LieA^+$. Then
\begin{align*}
\|\Ad_g\|_{\mathrm{op}} = \|\Ad_{k_1} \Ad_{a_v} \Ad_{k_2}\|_{\mathrm{op}} \leq \|\Ad_{a_v}\|_{\mathrm{op}}
\end{align*}
by the above fact. Using $C_{g'} \circ \exp = \exp \circ \Ad_{g'}$ for all $g' \in G$, where $C_{g'}: G \to G$ denotes the conjugation map by $g'$, and $\Ad_{\exp(v')} = \exp(\ad_{v'})$ for all $v' \in \LieG$, we calculate that
\begin{align}
\label{eqn:logOfConjugation}
\begin{aligned}
\Ad_{a_v}\Biggl(\sum_{\alpha \in \Phi \cup \{0\}} w_\alpha\Biggr) &= \sum_{\alpha \in \Phi \cup \{0\}}e^{\ad_v}(w_\alpha) = \sum_{\alpha \in \Phi \cup \{0\}}\sum_{j = 0}^\infty \frac{1}{j!} (\ad_v)^j(w_\alpha) \\
&= \sum_{\alpha \in \Phi \cup \{0\}}\sum_{j = 0}^\infty \frac{1}{j!} \alpha(v)^j w_\alpha = \sum_{\alpha \in \Phi \cup \{0\}}e^{\alpha(v)}w_\alpha
\end{aligned}
\end{align}
for all $w = \sum_{\alpha \in \Phi \cup \{0\}} w_\alpha \in \LieG = \sum_{\alpha \in \Phi \cup \{0\}} \LieG_\alpha$. Recall from \cref{sec:Preliminaries} that the restricted root space decomposition is orthogonal. Also, $d(o, go) = \|v\|$. So, we conclude that
\begin{align*}
\|\Ad_g\|_{\mathrm{op}} = \|\Ad_{a_v}\|_{\mathrm{op}} \leq e^{c_{\Phi}\|v\|} = e^{c_{\Phi}d(o, go)}.
\end{align*}
\end{proof}

For the following lemma, recall that $\log$ is well-defined on all of $U^\pm$. Note also that we can identify $\T_o(G/K) \cong \LieG/\LieK \cong \LieP$. By abuse of notation, $\exp_{\mathrm{Rem}}: \LieG \to G$ and $\exp_{\mathrm{Rem}}: \LieP \to G/K$ denote the Riemannian exponential maps.

\begin{lemma}
\label{lem:LieTheoreticBoundsI}
The following holds.
\begin{enumerate}
\item Let $\epsilon_G \in (0, 1)$ such that $\exp, \exp_{\mathrm{Rem}}: \LieG \to G$ and $\exp_{\mathrm{Rem}}: \LieP \to G/K$ are injective on $B_{\epsilon_G}^\LieG(0)$ and $B_{\epsilon_G}^\LieP(0)$, respectively. There exists $C_G \geq 1$ such that:
\begin{enumerate}
\item $C_G^{-1}\|\log(g)\| \leq d(e, g)$ for all $g \in \exp\bigl(B_{\epsilon_G}^\LieG(0)\bigr)$;
\item for any Lie subgroup $\tilde{G} < G$, we have $d_{\tilde{G}}(e, g) \leq \|\log(g)\|$ for all $g \in \exp(\tilde{\LieG})$ so that $\log(g)$ is defined.
\end{enumerate}
\item For all $x = ho \in G/K$ and $w \in \LieG$ such that $\|\Ad_{h^{-1}}(w)\| = 1$ and $\Ad_{h^{-1}}(w)$ forms an acute/right angle of $\omega \in \bigl(0, \frac{\pi}{2}\bigr]$ with $\LieK$, there exists $C_\omega \geq 1$ depending continuously on $\omega$ such that:
\begin{enumerate}
\item $d(x, gx) \leq \min\bigl\{e^{c_{\Phi}d(o, x)} d(e, g), 2d(o, x) + d(o, go)\bigr\}$ for all $g \in G$;
\item $C_\omega^{-1} e^{-c_\Phi d(o, x)} d(e, \exp(tw)) \leq d(x, \exp(tw)x)$ for all $t \in (0, \epsilon_G)$.
\end{enumerate}
\item We have $d(o, u^\pm o) \asymp \log(1 + \|\log(u^\pm)\|)$ for all $u^\pm \in U^\pm$.
\end{enumerate}
\end{lemma}

\begin{proof}
Property~(1)(a) of the lemma follows from the fact that $\exp: \LieG \to G$ is smooth and $d\exp_0 = \Id_\LieG$. For property~(1)(b), let $\tilde{G} < G$ be a Lie subgroup and $g = \exp(v)$ for some $v \in \tilde{\LieG}$. Now, $[0, 1] \to \tilde{G}$ defined by $t \mapsto \exp(tv)$ is a curve from $e$ to $g$. Thus, by left $G$-invariance of the Riemannian metric on $G$, we have
\begin{align*}
d_{\tilde{G}}(e, g) &\leq \int_0^1 \Bigl\|\frac{d}{dt}\Bigr|_{t = s}\exp(tv)\Bigr\| \, ds = \int_0^1 \Bigl\|\frac{d}{dt}\Bigr|_{t = s}\exp(-sv)\exp(tv)\Bigr\| \, ds \\
&= \int_0^1 \Bigl\|\frac{d}{dt}\Bigr|_{t = s} \exp((t - s)v)\Bigr\| \, ds = \int_0^1 \|d\exp_0(v)\| \, ds = \int_0^1 \|\Id_\LieG(v)\| \, ds \\
&= \|v\| = \|\log(g)\|.
\end{align*}

For property~(2), let $x = ho$, $w$, $g$, and $t$ be as in the lemma. We first prove property~(2)(a). The second bound follows from triangle inequality and left $G$-invariance of the metric. Denote by $\pi_{G/K}: G \to G/K$ the quotient map which we note is smooth. Observe that $d(\pi_{G/K})_e: \LieG \to \LieP$ is an orthogonal projection with respect to $\langle \cdot, \cdot \rangle = B_\theta$. By left $G$-invariance of the Riemannian metric on both $G$ and $G/K$, we can deduce the operator norm $\|d(\pi_{G/K})_{g'}\|_{\mathrm{op}} = 1$ for all $g' \in G$. Let $m_{g'}^{\mathrm{L}}: G \to G$ denote the left multiplication map by $g' \in G$ and $C_{g'}: G \to G$ denote the conjugation map by $g' \in G$. Using the identity $C_{h^{-1}} \circ m_{g'}^{\mathrm{L}} = m_{h^{-1}g'h}^{\mathrm{L}} \circ C_{h^{-1}}$, left $G$-invariance of the Riemannian metric on $G$, and the upper bound in \cref{lem:NormOfAdjointEstimate}, we have
\begin{align}
\label{eqn:ConjugationDifferential}
\|d(C_{h^{-1}})_{g'}\|_{\mathrm{op}} = \|d(C_{h^{-1}})_e\|_{\mathrm{op}} = \|\Ad_{h^{-1}}\|_{\mathrm{op}} \leq e^{c_{\Phi}d(o, ho)} = e^{c_{\Phi}d(o, x)}
\end{align}
for all $g' \in G$. Thus, using left $G$-invariance of the metric and both operator norm calculations, we have
\begin{align*}
d(x, gx) = d(o, h^{-1}gho) \leq d(e, h^{-1}gh) \leq e^{c_{\Phi}d(o, x)}d(e, g).
\end{align*}

Now, we prove property~(2)(b). Note that the derivative of the curve $t \mapsto \exp(t\Ad_{h^{-1}}(w))o$ at $0$ is $d(\pi_{G/K})_e(\Ad_{h^{-1}}(w)) \in \LieP$. We use the hypothesis that $\Ad_{h^{-1}}(w)$ forms an acute/right angle of $\omega \in \bigl(0, \frac{\pi}{2}\bigr]$ and also the maps $\exp_{\mathrm{Rem}}|_{B_{\epsilon_G}^\LieP(0)}$ and $\exp_{\mathrm{Rem}}|_{B_{\epsilon_G}^\LieG(0)}$ which are diffeomorphisms onto their images. Also, we use the lower bound in \cref{lem:NormOfAdjointEstimate}. We obtain
\begin{align*}
d(x, \exp(tw)x) &= d(o, h^{-1}\exp(tw)ho) = d(o, \exp(t\Ad_{h^{-1}}(w))o) \\
&\gg \|td(\pi_{G/K})_e(\Ad_{h^{-1}}(w))\| \gg_\omega \|t\Ad_{h^{-1}}(w)\| \\
&\gg d(e, \exp(t\Ad_{h^{-1}}(w))) = d(e, h^{-1}\exp(tw)h) \\
&\geq e^{-c_\Phi d(o, ho)}d(e, \exp(tw)) = e^{-c_\Phi d(o, x)}d(e, \exp(tw))
\end{align*}
where the second implicit constant is in $(0, 1]$ and depends continuously on $\omega$.

For property~(3), we focus on $u^+ \in U^+$ since the case $u^- \in U^-$ is similar. When $u^+ \in U^+ \cap \exp\bigl(B_{\epsilon_G}^\LieG(0)\bigr)$, the desired inequality follows from the other proven properties since
\begin{align}
\label{eqn:DistanceIn_G_and_G/K}
\log(1 + \|\log(u^+)\|) \asymp \|\log(u^+)\| \asymp d(e, u^+) \asymp d(o, u^+o).
\end{align}
Note that the last relation above is trivial for $u^+ = e$ and otherwise, by continuity of $C_{\omega(w)}$ in $w \in \LieU^+$ and compactness of the unit sphere centered at $0$ in $\LieU^+$, we can use the constant $\max_{w \in \LieU^+, \|w\| = 1} C_{\omega(w)}$, where $\omega(w)$ is the acute/right angle formed by $w \in \LieU^+ \setminus \{0\} \subset \LieG \setminus \LieK$ with $\LieK$.

Now suppose $u^+ \in U^+ \setminus \exp\bigl(B_{\epsilon_G}^\LieG(0)\bigr)$. Note that the inequality $d(o, u^+o) \ll \log(1 + \|\log(u^+)\|)$ is easier to prove. One way is by using the triangle inequality and left $G$-invariance of the metric on $d(o, u^+o) = d(o, a_{-tv}\tilde{u}^+a_{tv}o)$ with a fixed unit vector $v \in \interior(\LieA^+)$, $\tilde{u}^+ \in U^+ \cap \exp\bigl(B_{\epsilon_G}^\LieG(0)\bigr)$, and $t \asymp \log(\|\log(u^+)\|)$. We give another proof since the tools are needed for the harder reverse inequality. We use $d(o, u^+o) \asymp \log\|\Ad_{u^+}\|_{\mathrm{op}}$ from \cref{eqn:DistanceFormulaByOperatorNorm}. Let us write $n := \dim(G)$. Recall that the operator norm $\|\Ad_{u^+}\|_{\mathrm{op}}$ can be calculated as the square root of the maximum eigenvalue, $\sqrt{\lambda_1} > 0$, of the self-adjoint positive semidefinite operator $\Ad_{u^+}^*\Ad_{u^+}$ which has nonnegative eigenvalues $\lambda_1 \geq \lambda_2 \geq \dotsb \geq \lambda_n \geq 0$ with $\lambda_1 > 0$. Observe that $\lambda_1 = \max_{1 \leq j \leq n}\lambda_j \asymp \sum_{j = 1}^n \lambda_j = \tr(\Ad_{u^+}^*\Ad_{u^+})$ by comparing the max norm with the $L^1$ norm. All in all, it suffices to show that
\begin{align*}
\log\tr(\Ad_{u^+}^*\Ad_{u^+}) \asymp \log\|\log(u^+)\|.
\end{align*}
To this end, we first show that $\tr(\Ad_{u^+}^*\Ad_{u^+}) \ll \|\log(u^+)\|^{2(n - 1)}$ and then show that $\tr(\Ad_{u^+}^*\Ad_{u^+}) \gg \|\log(u^+)\|^{\frac{2}{n - 1}}$. Recall $\Ad_{u^+} = \exp\bigl(\ad_{\log(u^+)}\bigr)$. Also recall $\ad: \LieG \to \LieSL(\LieG)$ and $\Ad: G \to \SL(\LieG)$ since $G$ is semisimple. We can then fix an appropriate \emph{orthonormal} basis $\beta \subset \LieG$ such that elements of $\ad_{\log(\LieU^+)}$ and $\Ad_{\LieU^+}$ are all upper triangular nilpotent and unipotent matrices with respect to $\beta$, respectively. Note that the entries of $[\ad_{\log(u^+)}]_\beta$ are $O(\|\log(u^+)\|)$. Now, we use the fact that $\exp|_{\ad(\LieU^+)}$ is a \emph{polynomial} map of degree at most $n - 1$ since $\ad(\LieU^+)$ is a nilpotent Lie algebra. Then, the entries of $[\Ad_{u^+}]_\beta = [\exp\bigl(\ad_{\log(u^+)}\bigr)]_\beta = \exp[\ad_{\log(u^+)}]_\beta$ are $O(\|\log(u^+)\|^{n - 1})$ and the entries of $[\Ad_{u^+}^*\Ad_{u^+}]_\beta$ are $O(\|\log(u^+)\|^{2(n - 1)})$. This implies
\begin{align*}
\tr(\Ad_{u^+}^*\Ad_{u^+}) \ll \|\log(u^+)\|^{2(n - 1)}
\end{align*}
as desired. For the reverse inequality, we need to compute the trace more explicitly. We write the upper triangular matrices
\begin{align*}
[\ad_{\log(u^+)}]_\beta &= (w_{j, k})_{1 \leq j \leq n, 1 \leq k \leq n} =
\begin{pmatrix}
0 & w_{1, 2} & \dotsb & w_{1, n} \\
& \ddots & \ddots & \vdots \\
& & \ddots & w_{n - 1, n} \\
& & & 0
\end{pmatrix}
, \\
[\Ad_{u^+}]_\beta &= (W_{j, k})_{1 \leq j \leq n, 1 \leq k \leq n} =
\begin{pmatrix}
1 & W_{1, 2} & \dotsb & W_{1, n} \\
& \ddots & \ddots & \vdots \\
& & \ddots & W_{n - 1, n} \\
& & & 1
\end{pmatrix}
\end{align*}
with $w_{j, k} = 0$ for all $1 \leq k \leq j \leq n$, and $W_{j, k} = 0$ for all $1 \leq k < j \leq n$, and $W_{j, j} = 1$ for all $1 \leq j \leq n$. Let $\bigl(e_1, e_2, \dots, e_{n}\bigr)$ be the standard basis for $\R^{n}$. Then the restricted root system for $\LieSL(\LieG)$ is isomorphic to $\{e_j - e_k: 1 \leq j < k \leq n\}$ which is generated by the set of simple roots $\{e_j - e_{j + 1}: 1 \leq j \leq n - 1\}$. By explicitly calculating $[\Ad_{u^+}]_\beta = \exp[\ad_{\log(u^+)}]_\beta$, for all $1 \leq j < k \leq n$, we obtain the entry
\begin{align*}
W_{j, k} = \sum_{d = 1}^{k - j} \frac{1}{d!} \sum_{\substack{\{(j'_p, k'_p)\}_{p = 1}^d:\\ 1 \leq j'_p < k'_p \leq n \, \forall 1 \leq p \leq d \\ \sum_{p = 1}^d (e_{j'_p} - e_{k'_p}) = e_j - e_k}} \prod_{p = 1}^d w_{j'_p, k'_p}
\end{align*}
where the condition $\sum_{p = 1}^d (e_{j'_p} - e_{k'_p}) = e_j - e_k$ is equivalent to conditions $j'_1 = j$ and $k'_d = k$ and $j'_{p + 1} = k'_p$ for all $1 \leq p \leq d - 1$. Consequently, we note that the number of terms in the second sum above is the number of partitions of $k - j$ into $d$ parts. For example, for $n = 4$, we have the entry
\begin{align*}
W_{1, 4} = w_{1, 4} + \frac{1}{2!}(w_{1, 2}w_{2, 4} + w_{1, 3}w_{3, 4}) + \frac{1}{3!}w_{1, 2}w_{2, 3}w_{3, 4}.
\end{align*}
For $1 \leq k \leq n$, the $k$-th diagonal entry, i.e., the $(k, k)$ entry, of $[\Ad_{u^+}^*\Ad_{u^+}]_\beta = [\Ad_{u^+}]_\beta^{\mathrm{T}} [\Ad_{u^+}]_\beta$ is $1 + \sum_{j = 1}^{k - 1} W_{j, k}^2$. Thus, we obtain the trace
\begin{align*}
\tr(\Ad_{u^+}^*\Ad_{u^+}) &= \sum_{k = 1}^{n} \Biggl(1 + \sum_{j = 1}^{k - 1} W_{j, k}^2\Biggr) = n + \sum_{k = 2}^{n} \sum_{j = 1}^{k - 1} W_{j, k}^2 \\
&= n + \sum_{k = 2}^{n}\sum_{j = 1}^{k - 1} \left(\sum_{d = 1}^{k - j} \frac{1}{d!} \sum_{\substack{\{(j'_p, k'_p)\}_{p = 1}^d:\\ 1 \leq j'_p < k'_p \leq n \, \forall 1 \leq p \leq d \\ \sum_{p = 1}^d (e_{j'_p} - e_{k'_p}) = e_j - e_k}} \prod_{p = 1}^d w_{j'_p, k'_p}\right)^2.
\end{align*}
For example, for $n = 3$, we obtain the trace
\begin{align*}
\tr(\Ad_{u^+}^*\Ad_{u^+}) = 3 + w_{1, 2}^2 + w_{2, 3}^2 + \biggl(w_{1, 3} + \frac{1}{2}w_{1, 2}w_{2, 3}\biggr)^2.
\end{align*}
For the sake of contradiction, suppose that for all $C > 0$, there exists a corresponding $u^+ \in U^+ \setminus \exp\bigl(B_{\epsilon_G}^\LieG(0)\bigr)$ such that $\tr(\Ad_{u^+}^*\Ad_{u^+}) \leq C\|\log(u^+)\|^{\frac{2}{n - 1}}$. Let $C > 0$ (to be determined later) and take such a corresponding $u^+ \in U^+ \setminus \exp\bigl(B_{\epsilon_G}^\LieG(0)\bigr)$. Since the formula for the trace involves a sum of squares, we immediately get $|W_{j, k}| \leq C\|\log(u^+)\|^{\frac{1}{n - 1}}$ for all $1 \leq j < k \leq n$. In particular, we have $|w_{j, j + 1}| = |W_{j, j + 1}| \leq C\|\log(u^+)\|^{\frac{1}{n - 1}}$ for all $1 \leq j \leq d - 1$. By induction on the root system, we conclude that $|w_{j, k}| \leq C_{j, k}\|\log(u^+)\|^{\frac{k - j}{n - 1}} \leq C_{j, k}\|\log(u^+)\|$ where $C_{j, k}C^{-1}$ is some absolute constant depending only on $j$ and $k$, for all $1 \leq j < k \leq n$. This contradicts the fact that $\max_{1 \leq j < k \leq n} |w_{j, k}| \asymp \|\log(u^+)\|$ by choosing $C > 0$ sufficiently small so that $\max_{1 \leq j < k \leq n} C_{j, k}$ is smaller than the implicit constant here, concluding the proof.
\end{proof}

We fix $\epsilon_G$ and $C_G$ to be the ones from \cref{lem:LieTheoreticBoundsI} for the rest of the paper.

For the results in the rest of the section, the triangle inequality gives better bounds (see property~(2)(a) in \cref{lem:LieTheoreticBoundsI}) for $t \ll 1$. However, we exclude them since we obtain the right nontrivial bounds for $t \gg 1$ which is the regime of interest.

\begin{lemma}
\label{lem:LieTheoreticBoundsII}
Let $x \in G/K$, $v \in \LieA$ with $\|v\| = 1$, $u^\pm \in U^\pm$ with the unique decomposition
\begin{align*}
\log(u^\pm) = \sum_{\alpha \in \Phi^+} u_\alpha^\pm \in \LieU^\pm = \bigoplus_{\alpha \in \Phi^+} \LieG_{\mp\alpha},
\end{align*}
and $\Phi^+_{u^\pm} = \{\alpha \in \Phi^+: u_\alpha^\pm \neq 0\}$. If $\eta := \min_{\alpha \in \Phi^+_{u^\pm}}\alpha(v) > 0$, then
\begin{align*}
d(x, a_{\pm tv}u^\pm a_{\mp tv}x) \leq e^{c_{\Phi}d(o, x)} e^{O(d(o, u^\pm o))} e^{-\eta t} \qquad \text{for all $t \geq 0$}.
\end{align*}
\end{lemma}

\begin{proof}
Let $x$, $v$, $u^\pm$, and $\Phi^+_{u^\pm}$ be as in the lemma. We focus on the case $u^+ \in U^+$ since the case $u^- \in U^-$ is similar. Suppose $\eta := \frac{1}{2}\min_{\alpha \in \Phi^+_{u^+}} \alpha(v) > 0$. For all $t \in \R$, we calculate similar to \cref{eqn:logOfConjugation} that
\begin{align*}
\log\bigl(a_{tv}u^+a_{-tv}\bigr) = \Ad_{a_{tv}}\left(\sum_{\alpha \in \Phi^+_{u^+}} u_\alpha^+\right) = \sum_{\alpha \in \Phi^+_{u^+}}e^{-\alpha(v)t}u_\alpha^+.
\end{align*}
Thus, using properties~(2)(a) and (1)(b) in \cref{lem:LieTheoreticBoundsI} with the same notation, for all $t \geq 0$, we have
\begin{align*}
d(x, a_{tv}u^+a_{-tv}x) &\leq e^{c_{\Phi}d(o, x)}\left\|\sum_{\alpha \in \Phi^+_{u^+}} e^{-\alpha(v)t}u_\alpha^+ \right\| \\
&\leq e^{c_{\Phi}d(o, x)}e^{-\eta t} \sum_{\alpha \in \Phi^+_{u^+}} e^{(\eta - \alpha(v))t} \|u_\alpha^+\|.
\end{align*}
Thus, using orthogonality of the restricted root space decomposition, and then property~(3) in \cref{lem:LieTheoreticBoundsI}, we get
\begin{align*}
d(x, a_{tv}u^+a_{-tv}x) &\leq e^{c_{\Phi}d(o, x)} \|\log(u^+)\| e^{-\eta t} \\
&= e^{c_{\Phi}d(o, x)} e^{\log\|\log(u^+)\|} e^{-\eta t} \\
&\leq e^{c_{\Phi}d(o, x)} e^{O(d(o, u^+o))} e^{-\eta t}.
\end{align*}
\end{proof}

Let $v \in \partial\LieA^+ \setminus \{0\}$ which exists if and only if $\rankG \geq 2$. We introduce the following notations for the rest of this section. Define
\begin{align*}
\Phi(v) &= \{\alpha \in \Phi: \alpha(v) = 0\} \subset \Phi, & \Pi(v) &= \Pi \cap \Phi(v)^+ \subset \Phi(v)^+.
\end{align*}
Using the subspace $V(v) = \{\alpha \in \LieA^*: \alpha(v) = 0\} \subset \LieA^*$, it is clear that $\Phi(v) = \Phi \cap V(v) \subset \Phi$ is a proper root subsystem, where properness is due to the fact that there exists a simple root $\alpha_0 \in \Phi^+$ such that $\alpha_0(v) \neq 0$. By \cite[Chapter 6, \S 1.7, Corollary 3 and Proposition 24]{Bou02}, $\Pi(v) \subset \Phi(v)^+$ is the set of simple roots. We also use the common notation $\{h_\alpha\}_{\alpha \in \Pi} \subset \LieA$ for the set of dual roots of $\Pi$ such that $\alpha = \frac{2B(h_\alpha, \cdot)}{B(h_\alpha, h_\alpha)}$ for all $\alpha \in \Pi$.

\begin{proposition}
\label{pro:CommutingSemisimpleSubgroup}
Suppose $\rankG \geq 2$. Let $v \in \partial\LieA^+ \setminus \{0\}$. Then, the following holds.
\begin{enumerate}
\item Corresponding to the proper root subsystem $\Phi(v) \subset \Phi$, there exist a semisimple proper Lie subalgebra $\LieG(v) \subset \LieG$ and a corresponding connected semisimple proper Lie subgroup $G(v) \subset G$ such that $\bigoplus_{\alpha \in \Phi(v)} \LieG_\alpha \subset \LieG(v)$ and
\begin{align*}
a_{\dot{v}}g = ga_{\dot{v}} \qquad \text{for all $\dot{v} \in \bigcap_{\alpha \in \Pi(v)} \ker(\alpha)$ and $g \in G(v)$}.
\end{align*}
\item We have $\bigoplus_{\alpha \in \Pi(v)} \R h_\alpha \subset \LieG(v)$ and the decomposition
\begin{align*}
\LieA = \bigcap_{\alpha \in \Pi(v)} \ker(\alpha) \oplus \bigoplus_{\alpha \in \Pi(v)} \R h_\alpha.
\end{align*}
\item The subspaces $\bigcap_{\alpha \in \Pi(v)} \ker(\alpha) \subset \LieA \subset \LieG$ and $\LieG(v) \subset \LieG$ are orthogonal.
\end{enumerate}
\end{proposition}

\begin{proof}
Suppose $\rankG \geq 2$. Let $v \in \partial\LieA^+ \setminus \{0\}$. We first prove property~(1). We will construct a semisimple proper Lie subalgebra $\LieG(v) \subset \LieG$ and a connected semisimple proper Lie subgroup $G(v) \subset G$ in the following fashion. Let $\Phi^\C$ be the root system of the complexification $\LieG^\C = \C \otimes_\R \LieG$ with respect to a Cartan subalgebra $\LieG_0^\C$ containing $\LieA$, whose ordering is compatible with that of $\Phi$. Then, we have the root space decomposition $\LieG^\C = \LieG_0^\C \oplus \bigoplus_{\beta \in \Phi^\C} \LieG_\beta^\C$ where it is well-known that the root spaces are of one complex dimension. Now, let
\begin{align*}
\Phi^\C(v) = \bigl\{\beta \in \Phi^\C: \beta|_\LieA \in \Phi(v)\bigr\} = \bigl\{\beta \in \Phi^\C: \beta(v) = 0\bigr\} \subset \Phi^\C
\end{align*}
be a proper root subsystem. Let $\Pi^\C(v) \subset \Phi^\C(v)^+$ be the set of simple roots. Using the Chevalley basis
\begin{align}
\label{eqn:ChevalleyBasis}
\bigl\{h_\beta^\C \in \LieG_0^\C: \beta \in \Pi^\C(v)\bigr\} \cup \bigl\{e_{\pm\beta}^\C \in \LieG_{\pm\beta}^\C: \beta \in \Phi^\C(v)^+\bigr\}
\end{align}
corresponding to $\Phi^\C(v)$, we obtain a complex semisimple proper Lie subalgebra $\LieG^\C(v) \subset \LieG^\C$. As before, it has the root space decomposition
\begin{align*}
\LieG^\C(v) = \LieG^\C(v)_0 \oplus \bigoplus_{\beta \in \Phi^\C(v)} \LieG^\C(v)_\beta
\end{align*}
where $\LieG^\C(v)_0 = \bigoplus_{\beta \in \Pi^\C(v)} \C h_\beta^\C$ is the Cartan subalgebra and $\LieG^\C(v)_{\pm\beta} = \LieG_{\pm\beta}^\C = \C e_{\pm\beta}^\C$ for all $\beta \in \Phi^\C(v)^+$. Denote by $\sigma: \LieG^\C \to \LieG^\C$ the real linear involutive automorphism given by complex conjugation, i.e., $\sigma(x + iy) = x - iy$ for all $x, y \in \LieG$. We claim that
\begin{align*}
\sigma\bigl(\LieG^\C(v)\bigr) = \LieG^\C(v)
\end{align*}
so that $\sigma$ restricts to a real linear involutive automorphism on $\LieG^\C(v)$. Firstly, for all $\beta \in \Phi^\C$, it is easy to check that $\sigma\bigl(\LieG_\beta^\C\bigr) = \LieG_{\beta^\sigma}^\C$ where we define $\beta^\sigma = \overline{\beta \circ \sigma}$ which we note satisfies $(\beta^\sigma)^\sigma = \beta$. Secondly, for all $\beta \in \Phi^\C(v)$, we have $\beta^\sigma|_\LieA = \beta|_\LieA \in \Phi(v)$ from definitions which implies $\beta^\sigma \in \Phi^\C(v)$ (cf. \cite[Chapter VI, \S 3, Theorem 3.4]{Hel01}). So indeed, $\Phi^\C(v)$ and hence also $\LieG^\C(v)$ are closed under $\sigma$. Thus, we can take the real part of $\LieG^\C(v)$, i.e., the fixed points of $\sigma|_{\LieG^\C(v)}$, and obtain the \emph{real} semisimple Lie algebra $\LieG(v) = \LieG \cap \LieG^\C(v)$ which is a proper real Lie subalgebra of both $\LieG$ and $\LieG^\C(v)$. Note that $\LieG(v)$ is indeed semisimple since the Killing form $B|_{\LieG(v) \times \LieG(v)}$ is nondegenerate, being the restriction of a nondegenerate complex-valued Killing form on $\LieG^\C(v)$. In fact, $\LieG(v)$ which we constructed has a (not necessarily connected) Satake diagram corresponding to the data $(\Phi^\C(v), \sigma)$ obtained by deleting the Galois orbits of some white nodes of the Satake diagram corresponding to the data $(\Phi^\C, \sigma)$. We finally take $G(v) \subset G$ to be the connected semisimple proper Lie subgroup corresponding to the semisimple proper Lie subalgebra $\LieG(v) \subset \LieG$.

Denote
\begin{align*}
\Phi^\C(v)_\alpha &= \bigl\{\beta \in \Phi^\C(v): \beta|_\LieA = \alpha\bigr\} \subset \Phi^\C(v) \qquad \text{for all $\alpha \in \Phi(v)$}, \\
\Phi^\C(v)_\LieA &= \bigl\{\beta \in \Phi^\C(v): \beta|_\LieA \neq 0\bigr\} \subset \Phi^\C(v).
\end{align*}
By similar arguments as above, $\Phi^\C(v)_\alpha$ is closed under $\sigma$ for all $\alpha \in \Phi(v)$, and hence so is $\Phi^\C(v)_\LieA$. Consequently,
\begin{align*}
\LieG_\alpha = \LieG \cap \bigoplus_{\beta \in \Phi^\C(v)_\alpha} \LieG_\beta^\C.
\end{align*}
Therefore, we have (cf. \cite[Chapter VI, \S 3, Theorem 3.4]{Hel01})
\begin{align*}
\bigoplus_{\alpha \in \Phi(v)} \LieG_\alpha = \LieG \cap \bigoplus_{\beta \in \Phi^\C(v)_\LieA} \LieG_\beta^\C \subset \LieG(v).
\end{align*}

Now, let $\dot{v} \in \bigcap_{\alpha \in \Pi(v)} \ker(\alpha) \subset \LieA$. We have $[\dot{v}, x] = \beta(v)x = 0$ for all $x \in \LieG_\beta^\C$ and $\beta \in \Phi^\C(v)$. Using the Chevalley basis from \cref{eqn:ChevalleyBasis} and the Jacobi identity, it extends to $[\dot{v}, x] = 0$ for all $x \in \LieG^\C(v)$. In particular,
\begin{align*}
[\dot{v}, x] = 0 \qquad \text{for all $x \in \LieG(v)$}.
\end{align*}
Calculating as in \cref{eqn:logOfConjugation} on any neighborhood of the identity element $e \in G(v)$ such as $\exp\bigl(B_{\epsilon_G}^{\LieG(v)}(0)\bigr)$ and recalling that it generates $G(v)$ by connectedness, we conclude that
\begin{align*}
a_{\dot{v}}g = ga_{\dot{v}} \qquad \text{for all $\dot{v} \in \bigcap_{\alpha \in \Pi(v)} \ker(\alpha)$ and $g \in G(v)$}.
\end{align*}

Next, we prove property~(2). The decomposition is a simple exercise in linear algebra and for the containment, it suffices to show that $h_\alpha \in \LieG(v)$ for all $\alpha \in \Phi(v)$. Let $\alpha \in \Phi(v)$. We take any nonzero $e_\alpha \in \LieG_\alpha \subset \LieG(v)$ and $e_{-\alpha} = \theta(e_\alpha) \in \LieG_{-\alpha} \subset \LieG(v)$ and use the element $[e_\alpha, e_{-\alpha}] \in \LieG(v)$. We compute using the Jacobi identity that $[e_\alpha, e_{-\alpha}] \in \LieG_0 = \LieA \oplus \LieM$ and by applying $\theta$ we see that $[e_\alpha, e_{-\alpha}] \in \LieP$. Hence $[e_\alpha, e_{-\alpha}] \in (\LieA \oplus \LieM) \cap \LieP = \LieA$ Finally, for all $w \in \LieA$, we compute that
\begin{align*}
B(w, [e_\alpha, e_{-\alpha}]) &= B([w, e_\alpha], e_{-\alpha}) = \alpha(w) B(e_\alpha, e_{-\alpha}) \\
&= -\alpha(w)B_\theta(e_\alpha, e_\alpha) = -\|e_\alpha\|^2\alpha(w)
\end{align*}
and since $e_\alpha \neq 0$, this shows that $[e_\alpha, e_{-\alpha}] \in \R h_\alpha \setminus \{0\}$.

Finally, we prove property~(3). Let $\dot{v} \in \bigcap_{\alpha \in \Pi(v)} \ker(\alpha) \subset \LieA$. We will show that $\langle \dot{v}, x\rangle = B_\theta(\dot{v}, x) = 0$ for all $x \in \LieG(v)$. Extend the Cartan involution to $\theta: \LieG^\C \to \LieG^\C$ and the associated negative definite bilinear form to $B_\theta: \LieG^\C \times \LieG^\C \to \C$ by complex linearity. Since $\LieG(v) \subset \LieG^\C(v)$, it suffices to show by linearity that $B_\theta(\dot{v}, x) = 0$ for all $x \in \LieG^\C(v)_\beta$ and $\beta \in \Phi^\C(v) \cup \{0\}$. To this end, first let $\beta \in \Phi^\C(v)$. As $\beta \neq 0$, there exists $w_\beta \in \LieG^\C(v)_0$ such that $\beta(w_\beta) \neq 0$. Since $\dot{v}, w_\beta \in \LieG_0^\C$, for all $x \in \LieG_\beta^\C$, we have
\begin{align*}
\beta(w_\beta)B_\theta(\dot{v}, x) &= \beta(w_\beta)B(\dot{v}, x) = B(\dot{v}, [w_\beta, x]) = B([\dot{v}, w_\beta], x) = 0
\end{align*}
which implies $B_\theta(\dot{v}, x) = 0$. Next, recalling that $\LieG^\C(v)_0 = \bigoplus_{\beta \in \Pi^\C(v)} \C h_\beta^\C$, we have
\begin{align*}
B_\theta\bigl(\dot{v}, h_\beta^\C\bigr) &= B\bigl(\dot{v}, \bigl[e_\beta^\C, e_{-\beta}^\C\bigr]\bigr) = B\bigl(\bigl[\dot{v}, e_\beta^\C\bigr], e_{-\beta}^\C\bigr) = \beta(\dot{v})B\bigl(e_\beta^\C, e_{-\beta}^\C\bigr) = 0
\end{align*}
for all $\beta \in \Pi^\C(v)$.
\end{proof}

For all $v \in \partial\LieA^+ \setminus \{0\}$, we continue to denote by $\LieG(v)$ and $G(v)$ to be the ones provided by \cref{pro:CommutingSemisimpleSubgroup} for the rest of the section.

Suppose $\rankG \geq 2$. For all $\alpha \in \Pi$, it determines a closed half-space $H^\LieA_\alpha \subset \LieA$ on which $\alpha$ is nonnegative. The closed positive Weyl chamber can then be written as
\begin{align*}
\LieA^+ = \bigcap_{\alpha \in \Pi} H^\LieA_\alpha.
\end{align*}
Its walls are
\begin{align*}
W^\LieA_\alpha = \ker(\alpha) \cap \bigcap_{\alpha' \in \Pi \setminus \{\alpha\}} H^\LieA_{\alpha'} \qquad \text{for all $\alpha \in \Pi$}.
\end{align*}
There exists $\epsilon_\Pi > 0$ such that the spherical simplex $\LieA^+ \cap \partial B_1^\LieA(0)$ is not covered by $\{B_{\epsilon_\Pi}^\LieA(W^\LieA_\alpha)\}_{\alpha \in \Pi}$ which we fix for the rest of this section. In fact, the optimal $\epsilon_\Pi > 0$ can be found using the incenter of the spherical simplex $\LieA^+ \cap \partial B_1^\LieA(0)$. In any case, we can guarantee that $\bigcap_{\alpha \in \Pi} B_{\epsilon_\Pi}^\LieA(W^\LieA_{\alpha}) = \varnothing$. We now prove the following proposition. It is crucial that we treat the region $\bigcup_{\alpha \in \Pi} B_{\epsilon_\Pi}^\LieA(W^\LieA_\alpha)$ uniformly rather than the walls $\partial\LieA^+ = \bigcup_{\alpha \in \Pi} W^\LieA_\alpha$ separately in order to ensure that the obtained constant $C_{w, u^+, u^-}$ is truly independent of $v \in \bigcup_{\alpha \in \Pi} B_{\epsilon_\Pi}^\LieA(W^\LieA_\alpha)$. The obstruction in the later approach is that the exponential decay obtained for $v \in \bigcup_{\alpha \in \Pi} B_{\epsilon_\Pi}^\LieA(W^\LieA_\alpha) \setminus \bigcup_{\alpha \in \Pi} W^\LieA_\alpha$ has an exponential rate \emph{which goes to $0$} as $v$ approaches the walls $\partial\LieA^+ = \bigcup_{\alpha \in \Pi} W^\LieA_\alpha$.

\begin{proposition}
\label{pro:LieTheoreticBoundsIII}
Suppose $\rankG \geq 2$. Let $v \in \bigcup_{\alpha \in \Pi} B_{\epsilon_\Pi}^\LieA(W^\LieA_\alpha)$ with $\|v\| = 1$, $w \in \LieA$, $u^+ \in U^+$, and $u^- \in U^-$. Then, we have
\begin{align*}
|d(o, u^-a_{tv}a_wu^+o) - (t + \langle w, v\rangle)| \ll e^{O(\|w\| + d(o, u^+ o) + d(o, u^- o))} t^{-1}
\end{align*}
for all $t > 0$.
\end{proposition}

\begin{proof}
Suppose $\rankG \geq 2$. We fix the constant
\begin{align*}
c = \min_{\alpha \in \Pi} \min_{v' \in \ker(\alpha)^\perp, \|v'\| = 1} \alpha(v').
\end{align*}
Let $v$, $w$, $u^+$, and $u^-$ be as in the proposition. Note that $c$ is independent of $v$. For the sake of readability, we focus on the decay rate at first pass. Throughout the proof, the dependence of implicit constants on $w$, $u^+$, and $u^-$ are all continuous. In fact, at the end, we show that the final constant coefficient depending continuously on $w$, $u^+$, and $u^-$ can indeed be taken to be $O\bigl(e^{O(\|w\| + d(o, u^+ o) + d(o, u^- o))}\bigr)$. Let
\begin{align*}
\dot{\Pi}(v) = \{\alpha \in \Pi: v \in B_{\epsilon_\Pi}^\LieA(W^\LieA_\alpha)\} \subset \Pi
\end{align*}
which is a proper subset due to the choice of $\epsilon_\Pi$. Let $\dot{\Phi}(v) \subset \Phi$ be the maximal root subsystem generated by $\dot{\Pi}(v)$ which is automatically proper. Then by construction, we have
\begin{align*}
\alpha(v) \geq c\epsilon_\Pi \qquad \text{for all $\alpha \in \Phi^+ \setminus \dot{\Phi}(v)^+$}.
\end{align*}
Denote by $\dot{v}$ the orthogonal projection of $v$ onto $\bigcap_{\alpha \in \dot{\Pi}(v)} \ker(\alpha)$. Then in fact, $\dot{\Pi}(v) = \Pi(\dot{v})$ and $\dot{\Phi}(v) = \Phi(\dot{v})$ and we can write
\begin{align*}
v &= \dot{v} + \hat{v}, & w &= \dot{w} + \hat{w}
\end{align*}
with $\hat{v}, \hat{w} \in \LieG(\dot{v})$ using the decomposition $\LieA = \bigcap_{\alpha \in \dot{\Pi}(v)} \ker(\alpha) \oplus \bigoplus_{\alpha \in \dot{\Pi}(v)} \R h_\alpha$ given by property~(2) in \cref{pro:CommutingSemisimpleSubgroup}. Write $\log(u^\pm) = \sum_{\alpha \in \Phi^+} u_\alpha^\pm \in \bigoplus_{\alpha \in \Phi^+} \LieG_{\mp \alpha}$. We make the decompositions
\begin{align*}
u^+ &= \hat{u}^+\dot{u}^+, & u^- &= \dot{u}^-\hat{u}^-
\end{align*}
such that $\log(\dot{u}^\pm) = \sum_{\alpha \in \dot{\Phi}(v)^+} u_\alpha^\pm \in \LieG(\dot{v})$ by property~(1) of \cref{pro:CommutingSemisimpleSubgroup}. By the Baker--Campbell--Hausdorff formula, we necessarily have
\begin{align*}
\log(\hat{u}^\pm) \in \bigoplus_{\alpha \in \Phi^+ \setminus \dot{\Phi}(v)^+} \LieG_{\mp \alpha}.
\end{align*}
Then, by the reverse triangle inequality and \cref{lem:LieTheoreticBoundsII}, for all $t \geq 0$, we have
\begin{align*}
&|d(o, u^-a_{tv}a_wu^+o) - d(o, \dot{u}^-a_{tv}a_wu^+o)| \\
\leq{}&d(\dot{u}^-a_{tv}a_wu^+o, \dot{u}^-\hat{u}^-a_{tv}a_wu^+o) \\
={}&d(a_wu^+o, a_{-tv}\hat{u}^-a_{tv}a_wu^+o) \\
\leq{}&C_1e^{-c\epsilon_\Pi t}
\end{align*}
and also
\begin{align*}
&|d(o, \dot{u}^-a_{tv}a_wu^+o) - d(o, \dot{u}^-a_{tv}a_w\dot{u}^+o)| \\
={}&|d(o, \dot{u}^-a_{tv}a_w\hat{u}^+\dot{u}^+o) - d(\dot{u}^-a_{tv}a_w\hat{u}^+a_{-w}a_{-tv}(\dot{u}^-)^{-1}o, \dot{u}^-a_{tv}a_w\hat{u}^+\dot{u}^+o)| \\
\leq{}&d(o, \dot{u}^-a_{tv}a_w\hat{u}^+a_{-w}a_{-tv}(\dot{u}^-)^{-1}o) \\
={}&d(a_{-w}(\dot{u}^-)^{-1}o, a_{tv}\hat{u}^+a_{-tv}a_{-w}(\dot{u}^-)^{-1}o) \\
\leq{}&C_2e^{-c\epsilon_\Pi t}
\end{align*}
where $C_1, C_2 > 0$ are constants which depend continuously on $w$, $u^+$, and $u^-$, but independent of $v$.

Let $t \geq 0$. In light of the above estimates, it suffices to treat $d(o, \dot{u}^-a_{tv}a_w\dot{u}^+o)$. We first study the Cartan decomposition of $\dot{u}^-a_{tv}a_w\dot{u}^+$ using \cref{pro:CommutingSemisimpleSubgroup}. Recalling that $\dot{u}^\pm \in G(\dot{v})$, we use property~(1) in \cref{pro:CommutingSemisimpleSubgroup} to obtain
\begin{align*}
\dot{u}^-a_{tv}a_w\dot{u}^+ = a_{t\dot{v}}a_{\dot{w}}\dot{u}^-a_{t\hat{v}}a_{\hat{w}}\dot{u}^+.
\end{align*}
Then, $a_{t\hat{v}}, a_{\hat{w}} \in G(\dot{v})$ also and hence $\dot{u}^-a_{t\hat{v}}a_{\hat{w}}\dot{u}^+ \in G(\dot{v})$. Now, as $G(\dot{v}) \subset G$ is semisimple, $G(\dot{v})$ itself has a Cartan decomposition $G(\dot{v}) = K(\dot{v})A(\dot{v})^+K(\dot{v})$ and as in \cref{sec:Preliminaries}, we can arrange it so that
\begin{align*}
A(\dot{v}) &< A, & A(\dot{v})^+ &\subset A^+, & K(\dot{v}) &< K.
\end{align*}
Denote by $\LieA(\dot{v})^+$ the Lie algebra of $A(\dot{v})^+$. So there exist $k_1, k_2 \in K(\dot{v}) < K$ and $w_0 \in \LieA(\dot{v})^+$ such that $\dot{u}^-a_{t\hat{v}}a_{\hat{w}}\dot{u}^+ = k_1 a_{w_0} k_2$. Again using property~(1) in \cref{pro:CommutingSemisimpleSubgroup}, we obtain the Cartan decomposition
\begin{align}
\label{eqn:a_tvCommutesWithKTilde}
\dot{u}^-a_{tv}a_w\dot{u}^+ = a_{t\dot{v}}a_{\dot{w}}\dot{u}^-a_{t\hat{v}}a_{\hat{w}}\dot{u}^+ = a_{t\dot{v}}a_{\dot{w}}k_1 a_{w_0} k_2 = k_1 a_{t\dot{v}}a_{\dot{w}}a_{w_0} k_2
\end{align}
where by property~(3) in \cref{pro:CommutingSemisimpleSubgroup}, we have
\begin{align}
\label{eqn:OrthogonalityProperties}
\langle \dot{v}, \hat{v}\rangle = \langle \dot{v}, \hat{w}\rangle = \langle \dot{v}, w_0\rangle = \langle \dot{w}, \hat{v}\rangle = \langle \dot{w}, \hat{w}\rangle = \langle \dot{w}, w_0\rangle = 0.
\end{align}
Recall that the decomposition $\dot{u}^-a_{t\hat{v}}a_{\hat{w}}\dot{u}^+ = k_1 a_{w_0} k_2$ with respect to the Cartan decomposition of $G$ is unique up to the equivalence relation $(k_1, k_2) \sim (k_1m, m^{-1}k_2)$ for all $m \in M$. Since $M = Z_K(A)$, the same \cref{eqn:a_tvCommutesWithKTilde} actually holds for any decomposition $\dot{u}^-a_{t\hat{v}}a_{\hat{w}}\dot{u}^+ = k_1 a_{w_0} k_2$ with $k_1, k_2 \in K$ and $w_0 \in \LieA^+$ (which automatically implies $w_0 \in \LieA(v)^+$).

Now, we estimate $\|w_0\| = d(o, k_1 a_{w_0} k_2o) = d(o, \dot{u}^-a_{t\hat{v}}a_{\hat{w}}\dot{u}^+o)$.
Using similar techniques as before, for all $t \geq 0$, we get
\begin{align*}
&|d(o, \dot{u}^-a_{t\hat{v}}a_{\hat{w}}\dot{u}^+o) - d(o, a_{t\hat{v}}a_{\hat{w}}\dot{u}^+o)| \\
\leq{}&d(a_{t\hat{v}}a_{\hat{w}}\dot{u}^+o, \dot{u}^-a_{t\hat{v}}a_{\hat{w}}\dot{u}^+o) \\
={}&d(a_{\hat{w}}\dot{u}^+o, a_{-t\hat{v}}\dot{u}^-a_{t\hat{v}}a_{\hat{w}}\dot{u}^+o) \\
\leq{}&C_3e^{-c\|\hat{v}\| t}
\end{align*}
and also
\begin{align*}
&|d(o, a_{t\hat{v}}a_{\hat{w}}\dot{u}^+o) - d(o, a_{t\hat{v}}a_{\hat{w}}o)| \\
={}&|d(o, a_{t\hat{v}}a_{\hat{w}}\dot{u}^+o) - d(a_{t\hat{v}}a_{\hat{w}}\dot{u}^+a_{-\hat{w}}a_{-t\hat{v}}o, a_{t\hat{v}}a_{\hat{w}}\dot{u}^+o)| \\
\leq{}&d(o, a_{t\hat{v}}a_{\hat{w}}\dot{u}^+a_{-\hat{w}}a_{-t\hat{v}}o) \\
={}&d(a_{-\hat{w}}o, a_{t\hat{v}}\dot{u}^+a_{-t\hat{v}}a_{-\hat{w}}o) \\
\leq{}&C_4e^{-c\|\hat{v}\| t}
\end{align*}
where $C_3, C_4 > 0$ are constants which depend continuously on $w$, $u^+$, and $u^-$, but independent of $v$. Using the above estimates, for all $t \geq 0$, we have
\begin{align*}
\|w_0\| &= d(o, \dot{u}^-a_{t\hat{v}}a_{\hat{w}}\dot{u}^+o) = d(o, a_{t\hat{v}}a_{\hat{w}}o) + O_{w, u^+, u^-}\bigl(e^{-c\|\hat{v}\| t}\bigr) \\
&= \|t\hat{v} + \hat{w}\| + O_{w, u^+, u^-}\bigl(e^{-c\|\hat{v}\| t}\bigr).
\end{align*}
Using the Cartan decomposition in \cref{eqn:a_tvCommutesWithKTilde}, the orthogonality properties in \cref{eqn:OrthogonalityProperties}, and the estimate for $\|w_0\|$, for all $t \geq 0$, we obtain
\begin{align*}
&d(o, \dot{u}^-a_{tv}a_w\dot{u}^+o) = \|t\dot{v} + \dot{w} + w_0\| = \sqrt{\|t\dot{v} + \dot{w}\|^2 + \|w_0\|^2} \\
={}&\sqrt{\|t\dot{v} + \dot{w}\|^2 + \bigl(\|t\hat{v} + \hat{w}\| + O_{w, u^+, u^-}\bigl(e^{-c\|\hat{v}\| t}\bigr)\bigr)^2} \\
={}&\sqrt{\substack{t^2(\|\dot{v}\|^2 + \|\hat{v}\|^2) + 2t(\langle \dot{w}, \dot{v}\rangle + \langle \hat{w}, \hat{v}\rangle) + (\|\dot{w}\|^2 + \|\hat{w}\|^2) \\ + O_{w, u^+, u^-}\bigl(\|t\hat{v} + \hat{w}\|e^{-c\|\hat{v}\| t}\bigr)}} \\
={}&\sqrt{(t^2 + 2t\langle w, v\rangle + \langle w, v\rangle^2) + (\|w\|^2 - \langle w, v\rangle^2) + O_{w, u^+, u^-}\bigl(\|t\hat{v} + \hat{w}\|e^{-c\|t\hat{v}\|}\bigr)}.
\end{align*}
Clearly, $\|w\|^2 - \langle w, v\rangle^2 \leq \|w\|^2 = O_w(1)$. Now, consider the nonnegative continuous map $\LieA \to \R_{\geq 0}$ given by $v' \mapsto \|v' + \hat{w}\| e^{-c\|v'\|} \leq \|v'\| e^{-c\|v'\|} + \|\hat{w}\|$ which varies continuously in $w$. Along each ray $\R_{\geq 0} v_0'$ for some $v_0' \in \LieA$ with $\|v_0'\| = 1$, the map attains some maximum $M_{v_0', w}$ which depends continuously on $v_0'$ and $w$. Therefore, by compactness of the unit sphere centered at $0$ in $\LieA$, the map attains a maximum $M_w = \max_{v_0' \in \LieA, \|v_0'\| = 1} M_{v_0', w}$ on $\LieA$ which depends continuously on $w$, but is independent of $v$. In fact, $M_w \ll (ce)^{-1} + \|\hat{w}\| \leq c^{-1} + \|w\|$. Consequently, $O_{w, u^+, u^-}\bigl(\|t\hat{v} + \hat{w}\|e^{-c\|t\hat{v}\|}\bigr) = O_{w, u^+, u^-}(1)$. Note that the vector $t\hat{v}$ with exactly that scaling in the first factor was crucial for this. Thus, continuing the calculation using Taylor's theorem, for all $t \gg_{w, u^+, u^-} 1$ (where we can take the implicit constant to be the same as that of $O_{w, u^+, u^-}(t^{-2})$ in the second equality below), we obtain
\begin{align}
\label{eqn:DistanceEstimate_udot-atvawudot+}
\begin{aligned}
d(o, \dot{u}^-a_{tv}a_w\dot{u}^+o) &= \sqrt{(t + \langle w, v \rangle)^2 + O_{w, u^+, u^-}(1)} \\
&= (t + \langle w, v \rangle) \sqrt{1 + O_{w, u^+, u^-}(t^{-2})} \\
&= (t + \langle w, v \rangle) \bigl(1 + O_{w, u^+, u^-}(t^{-2})\bigr) \\
&= (t + \langle w, v \rangle) + O_{w, u^+, u^-}(t^{-1})
\end{aligned}
\end{align}
as desired.

Let us now focus on the constant coefficients. By \cref{lem:LieTheoreticBoundsII}, triangle inequality, and left $G$-invariance of the metric, we have
\begin{align*}
C_1 &= O\bigl(e^{O(d(o, a_wu^+o) + d(o, \hat{u}^-o))}\bigr) \leq O\bigl(e^{O(\|w\| + d(o, u^+o) + d(o, \hat{u}^-o))}\bigr), \\
C_2 &= O\bigl(e^{O(d(o, \dot{u}^-a_wo) + d(o, \hat{u}^+o))}\bigr) \leq O\bigl(e^{O(\|w\| + d(o, \dot{u}^-o) + d(o, \hat{u}^+o))}\bigr), \\
C_3 &= O\bigl(e^{O(d(o, a_{\hat{w}}\dot{u}^+o) + d(o, \dot{u}^-o))}\bigr) \leq O\bigl(e^{O(\|\hat{w}\| + d(o, \dot{u}^+o) + d(o, \dot{u}^-o))}\bigr), \\
C_4 &= O\bigl(e^{O(d(o, a_{\hat{w}}o) + d(o, \dot{u}^+o))}\bigr) \leq O\bigl(e^{O(\|\hat{w}\| + d(o, \dot{u}^+o))}\bigr).
\end{align*}
We have $\|\hat{w}\| \leq \|w\|$. We also have $\|\log(\dot{u}^\pm)\| \leq \|\log(u^\pm)\|$ from definitions and the orthogonality of the restricted root space decomposition and so together with property~(3) in \cref{lem:LieTheoreticBoundsI}, we get
\begin{align*}
d(o, \dot{u}^\pm o) \ll \log(1 + \|\log(\dot{u}^\pm)\|) \leq \log(1 + \|\log(u^\pm)\|) \ll d(o, u^\pm o).
\end{align*}
Again by triangle inequality as above, we also get
\begin{align*}
d(o, \hat{u}^\pm o) = d(o, u^\pm (\dot{u}^\pm)^{-1} o) \leq d(o, u^\pm o) + d(o, \dot{u}^\pm o) \ll d(o, u^\pm o).
\end{align*}
Thus, $C_j = O\bigl(e^{O(\|w\| + d(o, u^+ o) + d(o, u^- o))}\bigr)$ for all $1 \leq j \leq 4$. These constants propagate through the rest of the proof, possibly increasing only by factors of absolute constants. Two other implicit constants which additionally appear towards the end of the argument above are both $O\bigl(e^{O(\|w\|)}\bigr)$. We carry these constants through in \cref{eqn:DistanceEstimate_udot-atvawudot+}. Thus, we arrive at the final constant coefficient of
\begin{align*}
O\bigl(e^{O(\|w\| + d(o, u^+ o) + d(o, u^- o))}\bigr)
\end{align*}
as desired.
\end{proof}

\subsection{Statement and derivation of the main estimates}
In \cref{pro:PreciseGeometricBusemannFunction,cor:PreciseGeometricBusemannFunction}, we denote by $d_{\LieA}: G/K \times G/K \to \LieA^+$ the left $G$-invariant and right $K$-invariant vector-valued metric on the symmetric space $G/K$, i.e., for all $x = go \in G/K$ and $y = ho \in G/K$, we define $d_{\LieA}(x, y) \in \LieA^+$ to be the Cartan projection of $g^{-1}h$ which is the unique element such that $g^{-1}h \in Ka_{d_{\LieA}(x, y)}K$. Consequently, $d(x, y) = \|d_{\LieA}(x, y)\|$ for all $x, y \in G/K$. Also, we allow $\xi_{\boldsymbol{\cdot}}: \R \to G/K$ to be in a larger class of curves than geodesics. In fact, it is a geodesic if and \emph{only} if $g_2 \in K$. Although we do not need the vector-valued metric version of property~(1) in \cref{pro:PreciseGeometricBusemannFunction}, we provide it for completeness and also for independent interest.

\begin{proposition}
\label{pro:PreciseGeometricBusemannFunction}
Let $x = g_xo \in G/K$ and $y = g_yo \in G/K$. Let $\xi = kM \in \Fboundary \cong K/M$, $g_1 \in kAU^-$, $g_2 \in G$, and $v \in \LieA^+$ with $\|v\| = 1$. Define the curve $\xi_{\boldsymbol{\cdot}}: \R \to G/K$ by $\xi_t = g_1a_{tv}g_2o$ for all $t \in \R$. By the Iwasawa decomposition, there exist unique elements with unique decompositions
\begin{align}
\label{eqn:RootSpaceDecomposition}
\sum_{\alpha \in \Phi^+} u_{x, \alpha}^- &\in \bigoplus_{\alpha \in \Phi^+} \LieG_\alpha, & \sum_{\alpha \in \Phi^+} u_{y, \alpha} &\in \bigoplus_{\alpha \in \Phi^+} \LieG_\alpha, & \sum_{\alpha \in \Phi^+} u_{g_2, \alpha} &\in \bigoplus_{\alpha \in \Phi^+} \LieG_{-\alpha}
\end{align}
and $w_x, w_y \in \LieA$ such that
\begin{align*}
g_1^{-1}g_x &\in e^{\sum_{\alpha \in \Phi^+} u_{x, \alpha}}a_{w_x}K, & g_1^{-1}g_y &\in e^{\sum_{\alpha \in \Phi^+} u_{y, \alpha}}a_{w_y}K, & g_2 &\in e^{\sum_{\alpha \in \Phi^+} u_{g_2, \alpha}}AK.
\end{align*}
Let
\begin{align}
\label{eqn:NonzeroRootSpaceComponents}
\Phi^+_z = \{\alpha \in \Phi^+: u_{z, \alpha} \neq 0\} \qquad \text{for all $z \in \{x, y, g_2\}$}.
\end{align}
Then, there exists $C = O\bigl(e^{O(d(o, g_1^{-1}x) + d(o, g_1^{-1}y) + d(o, g_2o))}\bigr)$, such that the following holds.
\begin{enumerate}
\item If either of the alternatives
\begin{enumerate}
\item $\eta := \min_{\alpha \in \Phi^+_x \cup \Phi^+_y} \alpha(v) > 0$ and $\beta_\xi(x, y) = w_y - w_x = 0$;
\item $\eta := \min_{\alpha \in \Phi^+_x \cup \Phi^+_y \cup \Phi^+_{g_2}} \alpha(v) > 0$;
\end{enumerate}
is satisfied, then we have
\begin{align*}
\|(d_{\LieA}(x, \xi_t) - d_{\LieA}(y, \xi_t)) - \beta_\xi(x, y)\| \leq Ce^{-\eta t} \qquad \text{for all $t \geq 0$}.
\end{align*}
\item We have
\begin{align*}
|(d(x, \xi_t) - d(y, \xi_t)) - \langle d_{\LieA}(x, \xi_t) - d_{\LieA}(y, \xi_t), v \rangle| \leq
\begin{cases}
0, & \rankG = 1 \\
Ct^{-1}, & \rankG \geq 2
\end{cases}
\end{align*}
for all $t > 0$.
\item We have
\begin{align*}
|(d(x, \xi_t) - d(y, \xi_t)) - \langle \beta_\xi(x, y), v \rangle| \leq Ct^{-1} \qquad \text{for all $t > 0$}.
\end{align*}
\end{enumerate}
\end{proposition}

\begin{proof}
Let $x = g_xo$, $y = g_yo$, $\xi$, $g_1$, $g_2$, $v$, and $\xi_{\boldsymbol{\cdot}}$ be as in the proposition. Applying the isometry $g_1^{-1} \in G$, without loss of generality, we may assume that $\xi_t = a_{tv}g_2o$ for all $t \in \R$. Using property~(3) for the Busemann function in \cref{subsec:BusemannFunction}, it suffices to prove properties~(1) and (2) of the proposition for $x = o \in G/K$, i.e., we may assume $g_x = e$. For convenience, let us relabel $g := g_y$ and $h := g_2$. Using the Iwasawa decomposition in the form $G = U^-AK$, without loss of generality, we may assume that $g = u_g^-a_{w_g} \in U^-A$. Then, $g^{-1} = a_{-w_g}(u_g^-)^{-1} \in KAU^-$, and hence $\sigma(g^{-1}, \xi) = -w_g$. Thus, $\beta_\xi(x, y) = w_g$. Similarly, using the Iwasawa decomposition in the form $G = U^+AK$, without loss of generality, we may assume that $h = u_h^+a_{w_h} \in U^+A$. As in the proof of \cref{pro:LieTheoreticBoundsIII}, we focus on the decay rate at first pass and then take care of the constant coefficients at the very end.

Let us first prove property~(1). Suppose either property~(1)(a) or property~(1)(b) holds and let $\eta > 0$ be fixed accordingly. We will use several times the usual triangle inequality for the norm $\|\cdot \|$ on $\LieA$ as well as the reverse triangle inequalities for $d_{\LieA}$:
\begin{align*}
\|d_{\LieA}(x, y') - d_{\LieA}(x, y)\| &\leq \|d_{\LieA}(y, y')\| & \text{and} &  &\|d_{\LieA}(x, y) - d_{\LieA}(x', y)\| &\leq \|d_{\LieA}(x, x')\|
\end{align*}
for all $x, x', y, y' \in G/K$. They follow from \cite[Lemma 2.3]{Kas08} (see also \cite[Eq. (2-7)]{KLP18} and \cite[Section 3.8]{KLM09}). We calculate for all $t \in \R$ that
\begin{align}
\label{eqn:BusemannFirstInequality}
\begin{aligned}
&\|(d_{\LieA}(x, \xi_t) - d_{\LieA}(y, \xi_t)) - \beta_\xi(x, y)\| \\
={}&\|d_{\LieA}(o, a_{tv}ho) - d_{\LieA}(go, a_{tv}ho) - w_g\| \\
\leq{}&\|d_{\LieA}(o, a_{tv}ho) - d_{\LieA}(o, a_{tv - w_g}ho) - w_g\| \\
{}&+\|d_{\LieA}(o, a_{tv - w_g}ho) - d_{\LieA}(go, a_{tv}ho)\|.
\end{aligned}
\end{align}
Using \cref{lem:LieTheoreticBoundsII}, we bound the second term on the right hand side of \cref{eqn:BusemannFirstInequality} for all $t \geq 0$ as
\begin{align*}
&\|d_{\LieA}(o, a_{tv - w_g}ho) - d_{\LieA}(go, a_{tv}ho)\| \\
={}&\bigl\|d_{\LieA}\bigl(u_g^-a_{w_g}o, u_g^-a_{tv}ho\bigr) - d_{\LieA}\bigl(u_g^-a_{w_g} o, a_{tv}ho\bigr)\bigr\| \\
\leq{}&\bigl\|d_{\LieA}\bigl(a_{tv}ho, u_g^-a_{tv}ho\bigr)\bigr\| \\
={}&\bigl\|d_{\LieA}\bigl(ho, a_{-tv}u_g^-a_{tv}ho\bigr)\bigr\| \\
={}&d\bigl(ho, a_{-tv}u_g^-a_{tv}ho\bigr) \\
\leq{}&C_1e^{-\eta t}
\end{align*}
where $C_1 \geq 0$ is a constant depending continuously on $g$ and $h$, but independent of $v$. We bound the first term on the right hand side of \cref{eqn:BusemannFirstInequality} in the following fashion. For all $t \in \R$, if $w_g = 0$, then clearly
\begin{align*}
\|d_{\LieA}(o, a_{tv}ho) - d_{\LieA}(o, a_{tv - w_g}ho) - w_g\| = 0
\end{align*}
and if $w_g \neq 0$, then
\begin{align}
\label{eqn:eqn:BusemannSecondInequality}
\begin{aligned}
&\|d_{\LieA}(o, a_{tv}ho) - d_{\LieA}(o, a_{tv - w_g}ho) - w_g\| \\
\leq{}&\|d_{\LieA}(o, a_{tv}ho) - (tv + w_h)\| + \|(tv + w_h - w_g) - d_{\LieA}(o, a_{tv - w_g}ho)\|.
\end{aligned}
\end{align}
Again using \cref{lem:LieTheoreticBoundsII}, we bound the first term on the right hand side of \cref{eqn:eqn:BusemannSecondInequality} for all $t \geq 0$ as
\begin{align*}
&\|d_{\LieA}(o, a_{tv}ho) - (tv + w_h)\| \\
={}&\bigl\|d_{\LieA}\bigl(o, a_{tv}u_h^+a_{w_h}o\bigr) - d_{\LieA}(o, a_{tv + w_h}o)\bigr\| \\
={}&\bigl\|d_{\LieA}\bigl(o, a_{tv}u_h^+a_{w_h}o\bigr) - d_{\LieA}\bigl(a_{tv}u_h^+a_{-tv}o, a_{tv}u_h^+a_{w_h}o\bigr)\bigr\| \\
\leq{}&\bigl\|d_{\LieA}\bigl(o, a_{tv}u_h^+a_{-tv}o\bigr)\bigr\| \\
={}&d\bigl(o, a_{tv}u_h^+a_{-tv}o\bigr) \\
\leq{}&C_2e^{-\eta t}
\end{align*}
where $C_2 \geq 0$ is a constant depending continuously on $h$, but independent of $v$. By the same calculation with $tv - w_g$ in place of $tv$, we bound the second term on the right hand side of \cref{eqn:eqn:BusemannSecondInequality} as
\begin{align*}
\|(tv + w_h - w_g) - d_{\LieA}(o, a_{tv - w_g}ho)\| &\leq \bigl\|d_{\LieA}\bigl(o, a_{tv - w_g}u_h^+a_{-(tv - w_g)}o\bigr)\bigr\| \\
&\leq d\bigl(a_{w_g}o, a_{tv}u_h^+a_{-tv}a_{w_g}o\bigr) \\
&\leq C_3e^{-\eta t}
\end{align*}
where $C_3 \geq 0$ is a constant depending continuously on $g$ and $h$, but independent of $v$. Combining the inequalities concludes the proof of property~(1).

Next, we prove property~(2). When $\rankG = \dim(\LieA) = 1$, it is immediate. Now suppose $\rankG \geq 2$. The vectors $d_{\LieA}(x, \xi_t)$ and $d_{\LieA}(y, \xi_t)$ are Cartan projections of $a_{tv}h$ and $g^{-1}a_{tv}h$, respectively. By \cite[Lemma 4.6]{Ben97}, there exists a compact subset $\mathcal{C}_{g, h} \subset \LieA$ such that
\begin{align*}
d_{\LieA}(x, \xi_t) &= tv + e_x(t), & d_{\LieA}(y, \xi_t) &= tv + e_y(t),
\end{align*}
for some $e_x(t), e_y(t) \in \mathcal{C}_{g, h}$, for all $t \in \R$. In fact, for all $t \geq 0$, we have the precise bounds
\begin{align*}
\|e_x(t)\| &= \|d_{\LieA}(x, \xi_t) - tv\| = \|d_{\LieA}(o, a_{tv}ho) - d_{\LieA}(o, a_{tv}o)\| \\
&\leq \|d_{\LieA}(a_{tv}o, a_{tv}ho)\| = \|d_{\LieA}(o, ho)\| = d(o, ho)
\end{align*}
and similarly
\begin{align*}
\|e_y(t)\| &\leq d(o, a_{-tv}ga_{tv}o) + d(o, ho) \\
&\leq d(o, a_{-tv}u_g^-a_{tv}o) + d(a_{-tv}u_g^-a_{tv}o, a_{-tv}u_g^-a_{tv}a_{w_g}o) + d(o, ho) \\
&\leq d\bigl(o, u_g^-o\bigr) + \|w_g\| + d(o, ho)
\end{align*}
using \cref{eqn:logOfConjugation} property~(3) in \cref{lem:LieTheoreticBoundsI}. For all $t \gg_{g, h} 1$, decomposing $e_x(t) = \langle e_x(t), v \rangle v + e_{x, v^\perp}(t)$ in an orthogonal fashion for some $e_{x, v^\perp}(t) \in \LieA$, we have
\begin{align*}
d(x, \xi_t) &= \|d_{\LieA}(x, \xi_t)\| = \sqrt{(t + \langle e_x(t), v\rangle)^2 + e_{x, v^\perp}(t)^2} \\
&= (t + \langle e_x(t), v\rangle)\sqrt{1 + O_{g, h}(t^{-2})} \\
&= (t + \langle e_x(t), v\rangle)(1 + O_{g, h}(t^{-2})) \\
&= (t + \langle e_x(t), v\rangle) + O_{g, h}(t^{-1}) \\
&= \langle d_{\LieA}(x, \xi_t), v\rangle + O_{g, h}(t^{-1})
\end{align*}
where the implicit constant depends continuously on $g$ and $h$, but is independent on $v$. The same calculation holds for $y$ in place of $x$. Applying the triangle inequality finishes the proof of property~(2).

Finally, we prove property~(3). If $\rankG = 1$, then it is just a weakening of property~(1) which has been proven. Now suppose $\rankG \geq 2$. It suffices to treat the regions $\LieA^+ \cap \bigcup_{\alpha \in \Pi} B_{\epsilon_\Pi}^\LieA(W^\LieA_\alpha)$ and $\LieA^+ \setminus \bigcup_{\alpha \in \Pi} B_{\epsilon_\Pi}^\LieA(W^\LieA_\alpha)$ separately. If $v \in \LieA^+ \setminus \bigcup_{\alpha \in \Pi} B_{\epsilon_\Pi}^\LieA(W^\LieA_\alpha)$, then property~(3) follows from properties~(1) and (2) since the corresponding $\eta > 0$ is bounded away from $0$ uniformly in $v$. Now suppose $v \in \bigcup_{\alpha \in \Pi} B_{\epsilon_\Pi}^\LieA(W^\LieA_\alpha)$. By appropriate conjugations, we write $g = u_g^-a_{w_g} = a_{w_g} (\tilde{u}_g^-)^{-1} \in AU^-$ and $h = u_h^+a_{w_h} = a_{w_h} \tilde{u}_h^+ \in AU^+$. For all $t > 0$, we get
\begin{align*}
&|(d(x, \xi_t) - d(y, \xi_t)) - \langle \beta_\xi(x, y), v \rangle| \\
={}&|d(o, a_{tv}ho) - d(go, a_{tv}ho) - \langle w_g, v\rangle| \\
\leq{}&|d(o, a_{tv}ho) - (t + \langle w_h, v\rangle)| + |d(o, g^{-1}a_{tv}ho) - (t + \langle w_h - w_g, v\rangle)| \\
={}&\bigl|d\bigl(o, a_{tv}a_{w_h}\tilde{u}_h^+o\bigr) - (t + \langle w_h, v\rangle)\bigr| + \bigl|d\bigl(o, \tilde{u}_g^- a_{tv}a_{w_h - w_g}\tilde{u}_h^+o\bigr) - (t + \langle w_h - w_g, v\rangle)\bigr|.
\end{align*}
Applying \cref{pro:LieTheoreticBoundsIII} to both terms finishes the proof of property~(3).

Let us now deal with the constant coefficients. First, we establish some useful bounds. By Harish-Chandra's inequality from \cite[Chapter I, \S 6, Theorem 6.2]{JL01} (cf. \cite[Chapter I, \S 6, Theorem 6.1]{JL01} and \cite[Lemma 35]{Har58}), we have $\|w_g\| \leq \|d_{\LieA}(o, go)\| = d(o, go)$ and $\|w_h\| \leq \|d_{\LieA}(o, ho)\| = d(o, ho)$. Using this inequality, triangle inequality, and left $G$-invariance of the metric, we also get
\begin{align*}
d\bigl(o, u_g^-o\bigr) \leq d\bigl(o, u_g^-a_{w_g}o\bigr) + d\bigl(u_g^-a_{w_g}o, u_g^-o\bigr) = d(o, go) + \|w_g\| \ll d(o, go).
\end{align*}
Similarly
\begin{align*}
d\bigl(o, u_h^+o\bigr) &\leq d(o, ho), & d\bigl(o, \tilde{u}_g^-o\bigr) &\leq d(o, go), & d\bigl(o, \tilde{u}_h^+o\bigr) &\leq d(o, ho).
\end{align*}
We will utilize these bounds to derive the formulas for the constant coefficients. For property~(1), we use \cref{lem:LieTheoreticBoundsII} and triangle inequality as before to get
\begin{align*}
C_1 &= O\bigl(e^{O(d(o, ho) + d(o, u_g^-o))}\bigr) \leq O\bigl(e^{O(d(o, go) + d(o, ho))}\bigr), \\
C_2 &= O\bigl(e^{O(d(o, u_h^+o))}\bigr) \leq O\bigl(e^{O(d(o, ho))}\bigr), \\
C_3 &= O\bigl(e^{O(d(o, a_{w_g}o) + d(o, u_h^+o))}\bigr) \leq O\bigl(e^{O(d(o, go) + d(o, ho))}\bigr),
\end{align*}
giving the final constant coefficient of $O\bigl(e^{O(d(o, go) + d(o, ho))}\bigr)$. Similarly, for property~(2), we obtain the final constant coefficient of $O\bigl(e^{O(d(o, go) + d(o, ho))}\bigr)$ since both $\|e_x(t)\|$ and $\|e_y(t)\|$ are of the same order which we simply carry through. The constant provided by \cref{pro:LieTheoreticBoundsIII} in the proof of property~(3) is
\begin{align*}
O\bigl(e^{O(\|w_g\| + \|w_h\| + d(o, \tilde{u}_g^-o)) + d(o, \tilde{u}_h^+o))}\bigr)
\end{align*}
which again gives the final constant coefficient of $O\bigl(e^{O(d(o, go) + d(o, ho))}\bigr)$. Finally, for all three properties, reverting the reductions from the beginning of the proof and recalling $x = g_xo$ and $y = g_yo$ gives the final constant coefficient of
\begin{align*}
O\bigl(e^{O(d(o, g_1^{-1}x) + d(o, g_1^{-1}y) + d(o, g_2o))}\bigr).
\end{align*}
\end{proof}

For the following corollary, recall the constant $\eta_1 > 0$ from \cref{eqn:Eta1}.

\begin{corollary}
\label{cor:PreciseGeometricBusemannFunction}
Let $x, y \in G/K$. Let $\xi = kM \in \Fboundary \cong K/M$, $g_1 \in kAU^-$, $g_2 \in G$, and $v \in \LieA^+$ with $\|v\| = 1$. Define the curve $\xi_{\boldsymbol{\cdot}}: \R \to G/K$ by $\xi_t = g_1a_{tv}g_2o$ for all $t \in \R$. Then, there exists $C = O\bigl(e^{O(d(o, g_1^{-1}x) + d(o, g_1^{-1}y) + d(o, g_2o))}\bigr)$ such that
\begin{align*}
|(d(x, \xi_t) - d(y, \xi_t)) - \langle \beta_\xi(x, y), v\rangle| \leq
\begin{cases}
Ce^{-\eta_1 t}, & \rankG = 1 \\
Ct^{-1}, & \rankG \geq 2
\end{cases}
\qquad \text{for all $t > 0$}.
\end{align*}
\end{corollary}

\begin{remark}
It is evident in the proof of \cref{pro:PreciseGeometricBusemannFunction} that it is greatly simplified when the unit vectors $v \in \LieA^+$ are uniformly bounded away from the walls $\partial\LieA^+$ since \cref{pro:LieTheoreticBoundsIII} is not needed in that case. For our purposes later in the paper (see the derivation of \cref{eqn:SecondLieASkewBallContainment}), we actually take unit vectors $v \in \LieA_H^+ \subset \LieA^+$ which are bounded away from the walls $\partial\LieA_H^+$ and so the simpler proof of \cref{pro:PreciseGeometricBusemannFunction} suffices if $\LieA_H^+ \not\subset \partial \LieA^+$. However, this is not the case in many interesting settings. The simplest example for which $\LieA_H^+ \subset \partial \LieA^+$ occurs is $(G, H) = (\SL_4(\R), \SO_Q(\R)^\circ)$ where we take the quadratic form $Q(x_1, x_2, x_3, x_4) = x_2^2 + x_3^2 - 2x_1x_4$ so that $\SO_Q(\R)^\circ \cong \SO(3, 1)^\circ$. More generally, $\LieA_H^+ \subset \partial \LieA^+$ for $(G, H) = (\SL_n(\R), \SO_Q(\R)^\circ)$ where we take the quadratic form $Q(x_1, x_2, \dotsc, x_{p + q}) = \sum_{j = q + 1}^p x_j^2 - 2\sum_{j = 1}^q x_jx_{p + q + 1 - j}$ so that $\SO_Q(\R)^\circ \cong \SO(p, q)^\circ$, whenever $p, q, n \in \N$ with $n = p + q$ and $p \geq q$ satisfies $n - 2q \geq 2$. This is because in this case
\begin{align*}
\LieA_H^+ = \{\diag(t_1, \dotsc, t_q, 1, \dotsc, 1, -t_q, \dotsc, -t_1): t_1, \dotsc, t_q \in \R\}
\end{align*}
and one of the roots in $\Phi$ vanishes on $\LieA_H^+$ since there are at least two entries with 1's. Note also that such subgroups $H$ are noncompact semisimple symmetric maximal proper subgroups (see \cref{subsec:TheSubgroupH}).
\end{remark}

\section{Effective volume calculations for Riemannian skew balls}
\label{sec:VolumeCalculations}
This section is valid for any noncompact semisimple Lie subgroup $H < G$. We prove \cref{thm:SkewBallVolumeAsymptotic} which gives precise asymptotic formulas for the volume of Riemannian skew balls of $H$ in $G$, strengthening \cite[Theorem 9.3]{GW07}. \Cref{cor:RiemannianVolumeRatioAsymptotic} immediately follows. Related volume formulas appear as early as in \cite{Kni97} and also in \cite{Mau07,GOS09,BO12}.

\begin{notation}
As we restrict our attention to the subgroup $H < G$ in this subsection, we \emph{drop the subscript $H$} for all the objects introduced in \cref{sec:Preliminaries} (except measures and ranks) for convenience, unless otherwise mentioned. We warn the reader that these objects should not be confused with those associated to $G$.
\end{notation}

For the rest of the paper, we fix the following. Fix $\beta_r := \pi^{\frac{r}{2}} \Gamma\bigl(\frac{r}{2} + 1\bigr)^{-1} > 0$ to be the volume of a unit ball of dimension $r \geq 0$. Fix $v_{2\rho} \in \LieA$ to be the unit vector in the direction of maximal growth for $2\rho \in \LieA^*$ so that we have the growth rate $\delta_{2\rho} := \|2\rho\| = \max_{v \in \LieA, \|v\| = 1} 2\rho(v) = 2\rho(v_{2\rho})$. It is well-known that $v_{2\rho} \in \interior(\LieA^+)$ (see for example \cite[Lemma 9.2]{GW07} and its proof). Also, recall the constant $\eta_1 \in (0, \delta_{2\rho})$ from \cref{eqn:Eta1}.

\begin{theorem}
\label{thm:SkewBallVolumeAsymptotic}
Fix the constant
\begin{align*}
\omega_0 &:= 2^{\frac{\rankH - 1}{2}}\delta_{2\rho}^{-\frac{\rankH + 1}{2}} \beta_{\rankH - 1} \int_0^{+\infty} e^{-x} x^{\frac{\rankH - 1}{2}} \, dx.
\end{align*}
Let $g_1, g_2 \in G$. There exist
\begin{align*}
\varpi[g_1, g_2] &:= \beta_{e^+}(g_1^{-1}o, o) + \involution(\beta_{e^-}(g_2o, o)) = O(d(o, g_1o) + d(o, g_2o)), \\
C[g_1, g_2] &:= \frac{\omega_0}{2^{m_{\Phi^+}} \mu_M(M)}\int_K \int_K e^{-\delta_{2\rho}\langle \varpi[g_1k_1, k_2g_2], v_{2\rho}\rangle} \, d\mu_K(k_1) \, d\mu_K(k_2) \\
E[g_1, g_2] &= O\bigl(e^{O(d(o, g_1o) + d(o, g_2o))}\bigr)
\end{align*}
varying continuously in $g_1$ and $g_2$ such that
\begin{align*}
\mu_H(H_T[g_1, g_2]) =
\begin{cases}
C[g_1, g_2] e^{\delta_{2\rho} T} + E[g_1, g_2] e^{(\delta_{2\rho} - \eta_1)T}, & \rankG = 1 \\
C[g_1, g_2] T^{\frac{\rankH - 1}{2}} e^{\delta_{2\rho}T} + E[g_1, g_2] \log(T)^{\frac{1}{2}}T^{\frac{\rankH - 2}{2}}e^{\delta_{2\rho}T}, & \rankG \geq 2
\end{cases}
\end{align*}
for all $T \geq 0$ in the $\rankG = 1$ case and for all $T \geq \Omega\bigl(e^{\Omega(d(o, g_1o) + d(o, g_2o))}\bigr)$ in the $\rankG \geq 2$ case.
\end{theorem}

As a corollary, we also obtain a precise asymptotic formula for the ratio of the volume of Riemannian skew balls, which is an effective version of property~D2 in \cite{GW07}.

\begin{corollary}
\label{cor:RiemannianVolumeRatioAsymptotic}
Let $g_1, g_2 \in G$. Fix the positive constants $C[g_1, g_2^{-1}]$ and $C[e, e]$ provided by \cref{thm:SkewBallVolumeAsymptotic}. Fix
\begin{align*}
\tilde{\alpha}(g_1, g_2) := \frac{C[g_1, g_2^{-1}]}{C[e, e]} = \frac{2^{m_{\Phi^+}} \mu_M(M)}{\omega_0 \mu_K(K)^2} C[g_1, g_2^{-1}].
\end{align*}
There exists $\tilde{\alpha}'(g_1, g_2) = O\bigl(e^{O(d(o, g_1o) + d(o, g_2o))}\bigr)$ varying continuously in $g_1$ and $g_2$ such that
\begin{align*}
\frac{\mu_H(H_T[g_1, g_2^{-1}])}{\mu_H(H_T)} =
\begin{cases}
\tilde{\alpha}(g_1, g_2) + \tilde{\alpha}'(g_1, g_2) e^{-\eta_1 T}, & \rankG = 1 \\
\tilde{\alpha}(g_1, g_2) + \tilde{\alpha}'(g_1, g_2) T^{-\frac{1}{2}}\log(T)^{\frac{1}{2}}, & \rankG \geq 2
\end{cases}
\end{align*}
for all $T \geq 0$ in the $\rankG = 1$ case and for all $T \geq \Omega\bigl(e^{\Omega(d(o, g_1o) + d(o, g_2o))}\bigr)$ in the $\rankG \geq 2$ case.
\end{corollary}

\begin{proof}
Let $g_1$, $g_2$, $C[g_1, g_2^{-1}]$, $C[e, e]$, and $T$ be as in the corollary. Let $E[g_1, g_2^{-1}] = O\bigl(e^{O(d(o, g_1o) + d(o, g_2o))}\bigr)$ and $E[e, e]$ be the constants provided by \cref{thm:SkewBallVolumeAsymptotic}. Suppose $\rankG = 1$. Using \cref{thm:SkewBallVolumeAsymptotic} and Taylor's theorem, we calculate that
\begin{align*}
\frac{\mu_H(H_T[g_1, g_2^{-1}])}{\mu_H(H_T)} &= \frac{C[g_1, g_2^{-1}] e^{\delta_{2\rho} T} + E[g_1, g_2^{-1}] e^{(\delta_{2\rho} - \eta_1)T}}{C[e, e] e^{\delta_{2\rho} T} + E[e, e] e^{(\delta_{2\rho} - \eta_1)T}} \\
&= \frac{\tilde{\alpha}(g_1, g_2) + O\bigl(e^{O(d(o, g_1o) + d(o, g_2o))}\bigr) e^{-\eta_1 T}}{1 + O\bigl(e^{-\eta_1 T}\bigr)} \\
&= \tilde{\alpha}(g_1, g_2) + O\bigl(e^{O(d(o, g_1o) + d(o, g_2o))}\bigr) e^{-\eta_1 T}.
\end{align*}
Suppose $\rankG \geq 2$. Again, using \cref{thm:SkewBallVolumeAsymptotic} and Taylor's theorem, we calculate that
\begin{align*}
\frac{\mu_H(H_T[g_1, g_2^{-1}])}{\mu_H(H_T)} &= \frac{C[g_1, g_2^{-1}] T^{\frac{\rankH - 1}{2}}e^{\delta_{2\rho} T} + E[g_1, g_2^{-1}] \log(T)^{\frac{1}{2}}T^{\frac{\rankH - 2}{2}}e^{\delta_{2\rho}T}}{C[e, e] T^{\frac{\rankH - 1}{2}}e^{\delta_{2\rho} T} + E[e, e] \log(T)^{\frac{1}{2}}T^{\frac{\rankH - 2}{2}}e^{\delta_{2\rho}T}} \\
&= \frac{\tilde{\alpha}(g_1, g_2) + O\bigl(e^{O(d(o, g_1o) + d(o, g_2o))}\bigr) T^{-\frac{1}{2}}\log(T)^{\frac{1}{2}}}{1 + O\bigl(T^{-\frac{1}{2}}\log(T)^{\frac{1}{2}}\bigr)}\\
&= \tilde{\alpha}(g_1, g_2) + O\bigl(e^{O(d(o, g_1o) + d(o, g_2o))}\bigr) T^{-\frac{1}{2}}\log(T)^{\frac{1}{2}}.
\end{align*}
\end{proof}

We need some preparation to prove \cref{thm:SkewBallVolumeAsymptotic}. Let $g_1, g_2 \in G$ and $T > 0$. Define
\begin{align*}
\LieA^+_T[g_1, g_2] &:= \{v \in \LieA^+: d(o, g_1 a_v g_2o) < T\}, & \LieA^+_T &:= \LieA^+_T[e, e].
\end{align*}
Then by the Cartan decomposition, we can write
\begin{align}
\label{eqn:CartanDecompositionOfRiemannianSkewBalls}
H_T[g_1, g_2] = \bigcup_{k_1, k_2 \in K} k_1\LieA^+_T[g_1k_1, k_2g_2]k_2.
\end{align}
As such, the problem of finding asymptotic formulas for $\mu_H(H_T[g_1, g_2])$ as $T \to +\infty$ reduces to the same for $\int_{\LieA^+_T[g_1, g_2]} \xi(v) \, dv$. For the latter problem, we need \cref{lem:LinearAlgebraIntegralExpansion,cor:LinearAlgebraConeIntegralExpansion} which are general results in linear algebra giving precise expansions for certain integrals. This strengthens the asymptotic formula in \cite[Lemma 9.4]{GW07}.

\subsection{Asymptotic formulas for integrals of exponentials of linear forms}
We first record a lemma which will be used in the proof of \cref{lem:LinearAlgebraIntegralExpansion}.

\begin{lemma}
\label{lem:IntegralTailExponentialDecay}
Let $P, Q: \R \to \R$ be polynomials and $\alpha \in \R$. Suppose that the leading term of $Q$ is larger than $T$, i.e., $\lim_{T \to +\infty} (Q(T) - T) = +\infty$. Then, there exists $\eta > 0$ such that for all $T \gg_{P, Q, \alpha} 1$, we have
\begin{align*}
\left|P(T)e^T\int_{Q(T)}^{+\infty} e^{-x} x^\alpha \, dx\right| \leq e^{-\eta T}.
\end{align*}
\end{lemma}

\begin{proof}
Let $P, Q: \R \to \R$ be polynomials and $\alpha \in \R$. Take any $n \in \Z_{\geq 0}$ such that $n \geq \alpha$. We calculate that
\begin{align*}
\left|P(T)e^T\int_{Q(T)}^{+\infty} e^{-x} x^\alpha \, dx\right| &\leq |P(T)|e^T\int_{Q(T)}^{+\infty} e^{-x} x^n \, dx \\
&= |P(T)|e^T [-e^{-x}R(x)]_{Q(T)}^{+\infty} \\
&= |P(T)|e^T \cdot R(Q(T)) e^{-Q(T)} \\
&= |P(T)|R(Q(T)) e^{T - Q(T)}
\end{align*}
for all $T \gg_Q 1$, where $R: \R \to \R$ is some polynomial resulting from iterated integration by parts. The lemma follows.
\end{proof}

In analogy with the notations related to Riemannian skew balls, for an inner product space $V$ over $\R$, we denote by $V_T \subset V$ the open ball of radius $T > 0$ centered at $0 \in V$.

\begin{lemma}
\label{lem:LinearAlgebraIntegralExpansion}
Let $V$ be an inner product space over $\R$ of dimension $r := \dim(V)$, $\lambda \in V^* \setminus \{0\}$, and $\delta_\lambda := \max_{v \in \overline{V_1}} \lambda(v) > 0$. Then, we have the following:
\begin{enumerate}
\item if $r = 1$, then for all $T > 0$, we have
\begin{align*}
\int_{V_T} e^{\lambda(v)} \, dv = \frac{1}{\delta_\lambda}\bigl(e^{\delta_\lambda T} - e^{-\delta_\lambda T}\bigr) = \frac{2}{\delta_\lambda}\sinh(\delta_\lambda T);
\end{align*}
\item if $r \geq 2$, then there exists $\{\omega_k\}_{k = 0}^\infty \subset \R$ defined by
\begin{align}
\label{eqn:omega_k}
\omega_k := (-1)^k \frac{2^{\frac{r - 1}{2} - k}}{\delta_\lambda^{\frac{r + 1}{2} + k}} \binom{\frac{r - 1}{2}}{k} \beta_{r - 1} \int_0^{+\infty} e^{-x} x^{\frac{r - 1}{2} + k} \, dx
\end{align}
such that for all $n \in \N$, there exists $C > 0$ such that for all $T > 0$, we have
\begin{align*}
\left|\int_{V_T} e^{\lambda(v)} \, dv - \sum_{k = 0}^n \omega_k T^{\frac{r - 1}{2} - k} e^{\delta_\lambda T}\right| \leq CT^{\frac{r - 1}{2} - (n + 1)} e^{\delta_\lambda T}.
\end{align*}
\end{enumerate}
\end{lemma}

\begin{proof}
Let $V$, $r$, $\lambda$, and $\delta_\lambda$ be as in the lemma. By rescaling $\lambda$, we may assume without loss of generality that $\delta_\lambda = 1$. Fix $v_\lambda \in V$ to be the unit vector in the direction of maximal growth so that $\lambda(v_\lambda) = \delta_\lambda = 1$.

First suppose that $r = 1$. We can directly calculate that
\begin{align*}
\int_{V_T} e^{\lambda(v)} \, dv = \int_{-T}^T e^{\lambda(xv_\lambda)} \, dx = \int_{-T}^T e^x \, dx = e^T - e^{-T} \qquad \text{for all $T > 0$}.
\end{align*}

Now suppose $r \geq 2$. We have
\begin{align*}
\vol(V_T \cap (xv_\lambda + \ker(\lambda))) = \beta_{r - 1} (T^2 - x^2)^{\frac{r - 1}{2}}
\end{align*}
for all $T > 0$ and $-T \leq x \leq T$. We calculate that for all $T > 0$, we have
\begin{align*}
\int_{V_T} e^{\lambda(v)} \, dv ={}&\int_{-T}^T e^{\lambda(xv_\lambda)} \vol(V_T \cap (xv_\lambda + \ker(\lambda))) \, dx \\
={}&\beta_{r - 1}\int_{-T}^T e^x (T^2 - x^2)^{\frac{r - 1}{2}} \, dx \\
={}&\beta_{r - 1}e^T\int_0^{2T} e^{-x} x^{\frac{r - 1}{2}} (2T - x)^{\frac{r - 1}{2}} \, dx \\
={}&2^{\frac{r - 1}{2}} \beta_{r - 1} T^{\frac{r - 1}{2}} e^T\int_0^{2T} e^{-x} x^{\frac{r - 1}{2}} \left(1 - \frac{x}{2T}\right)^{\frac{r - 1}{2}} \, dx \\
={}&2^{\frac{r - 1}{2}} \beta_{r - 1} T^{\frac{r - 1}{2}} e^T\int_0^{\frac{3T}{2}} e^{-x} x^{\frac{r - 1}{2}} \left(1 - \frac{x}{2T}\right)^{\frac{r - 1}{2}} \, dx \\
{}&+ 2^{\frac{r - 1}{2}} \beta_{r - 1} T^{\frac{r - 1}{2}} e^T\int_{\frac{3T}{2}}^{2T} e^{-x} x^{\frac{r - 1}{2}} \left(1 - \frac{x}{2T}\right)^{\frac{r - 1}{2}} \, dx.
\end{align*}
We call the second term $E_1$ and continue the calculation. For any $n \in \N$ and $T > 0$, we have
\begin{align*}
&\int_{V_T} e^{\lambda(v)} \, dv \\
={}&2^{\frac{r - 1}{2}} \beta_{r - 1} T^{\frac{r - 1}{2}} e^T\int_0^{\frac{3T}{2}} e^{-x} x^{\frac{r - 1}{2}} \left(\sum_{k = 0}^n (-1)^k \binom{\frac{r - 1}{2}}{k} \left(\frac{x}{2T}\right)^k + O\biggl(\left(\frac{x}{2T}\right)^{n + 1}\biggr)\right) \, dx \\
{}&+ E_1 \\
={}&2^{\frac{r - 1}{2}} \beta_{r - 1} T^{\frac{r - 1}{2}} e^T\int_0^{\frac{3T}{2}} e^{-x} x^{\frac{r - 1}{2}} \sum_{k = 0}^n (-1)^k \binom{\frac{r - 1}{2}}{k} \left(\frac{x}{2T}\right)^k \, dx \\
{}&+ 2^{\frac{r - 1}{2}} \beta_{r - 1} T^{\frac{r - 1}{2}} e^T\int_0^{\frac{3T}{2}} e^{-x} x^{\frac{r - 1}{2}} O\biggl(\left(\frac{x}{2T}\right)^{n + 1}\biggr) \, dx + E_1 \\
={}&\sum_{k = 0}^n (-1)^k 2^{\frac{r - 1}{2} - k} \binom{\frac{r - 1}{2}}{k} \beta_{r - 1} T^{\frac{r - 1}{2} - k} e^T\int_0^{\frac{3T}{2}} e^{-x} x^{\frac{r - 1}{2} + k} \, dx + E_1 + E_2
\end{align*}
where the implicit constant is an absolute constant depending only on $r$ and $n$ and we call the second integral term $E_2$. We can define the constants
\begin{align*}
\eta_k = \int_0^{+\infty} e^{-x} x^{\frac{r - 1}{2} + k} \, dx \qquad \text{for all $k \in \Z_{\geq 0}$}
\end{align*}
since the integrals converge. We then continue the calculation and get
\begin{align*}
&\int_{V_T} e^{\lambda(v)} \, dv \\
={}&\sum_{k = 0}^n (-1)^k 2^{\frac{r - 1}{2} - k} \binom{\frac{r - 1}{2}}{k} \beta_{r - 1} T^{\frac{r - 1}{2} - k} e^T\int_0^{+\infty} e^{-x} x^{\frac{r - 1}{2} + k} \, dx \\
&{}- \sum_{k = 0}^n (-1)^k 2^{\frac{r - 1}{2} - k} \binom{\frac{r - 1}{2}}{k} \beta_{r - 1} T^{\frac{r - 1}{2} - k} e^T\int_{\frac{3T}{2}}^{+\infty} e^{-x} x^{\frac{r - 1}{2} + k} \, dx + E_1 + E_2 \\
={}&\sum_{k = 0}^n (-1)^k 2^{\frac{r - 1}{2} - k} \binom{\frac{r - 1}{2}}{k} \beta_{r - 1} \eta_k T^{\frac{r - 1}{2} - k} e^T + E_1 + E_2 + E_3
\end{align*}
where we call the second integral term $E_3$. Thus, it suffices to prove that there exists a constant $C > 0$ depending only on $r$ and $n$ such that for all $T > 0$, we have
\begin{align}
\label{eqn:E1E2E3Estimate}
|E_1 + E_2 + E_3| \leq CT^{\frac{r - 1}{2} - (n + 1)} e^T.
\end{align}

We estimate $|E_1|$ using \cref{lem:IntegralTailExponentialDecay} with
\begin{align*}
P(T) &= 2^{\frac{r - 1}{2}} \beta_{r - 1} T^{\frac{r - 1}{2}}, & Q(T) &= \frac{3T}{2}, & \alpha &= \frac{r - 1}{2}.
\end{align*}
We then obtain constants $\eta_1 > 0$ and $T_1 > 0$ both depending only on $r$ such that for all $T \geq T_1$, we have
\begin{align*}
|E_1| &= \left|2^{\frac{r - 1}{2}} \beta_{r - 1} T^{\frac{r - 1}{2}} e^T\int_{\frac{3T}{2}}^{2T} e^{-x} x^{\frac{r - 1}{2}} \left(1 - \frac{x}{2T}\right)^{\frac{r - 1}{2}} \, dx\right| \\
&\leq 2^{-\frac{r - 1}{2}} \beta_{r - 1} T^{\frac{r - 1}{2}} e^T\int_{\frac{3T}{2}}^{2T} e^{-x} x^{\frac{r - 1}{2}} \, dx \\
&\leq 2^{-\frac{r - 1}{2}} \beta_{r - 1} T^{\frac{r - 1}{2}} e^T\int_{\frac{3T}{2}}^{+\infty} e^{-x} x^{\frac{r - 1}{2}} \, dx \\
&\leq e^{-\eta_1 T}.
\end{align*}

We estimate $|E_2|$ as
\begin{align*}
|E_2| &= \left|2^{\frac{r - 1}{2}} \beta_{r - 1} T^{\frac{r - 1}{2}} e^T\int_0^{\frac{3T}{2}} e^{-x} x^{\frac{r - 1}{2}} O\biggl(\left(\frac{x}{2T}\right)^{n + 1}\biggr) \, dx\right| \\
&\leq C'2^{\frac{r - 1}{2} - (n + 1)} \beta_{r - 1} T^{\frac{r - 1}{2} - (n + 1)} e^T\int_0^{+\infty} e^{-x} x^{\frac{r - 1}{2} + (n + 1)} \, dx \\
&= C'2^{\frac{r - 1}{2} - (n + 1)} \beta_{r - 1} \eta_{n + 1}T^{\frac{r - 1}{2} - (n + 1)} e^T \\
&\leq C''T^{\frac{r - 1}{2} - (n + 1)}e^T
\end{align*}
for all $T > 0$, where $C' > 0$ and $C'' > 0$ are constants depending only on $r$ and $n$.

We estimate $|E_3|$ using \cref{lem:IntegralTailExponentialDecay} for all $0 \leq k \leq n$ with
\begin{align*}
P(T) &= (-1)^k 2^{\frac{r - 1}{2} - k} \binom{\frac{r - 1}{2}}{k} \beta_{r - 1} T^{\frac{r - 1}{2} - k}, & Q(T) &= \frac{3T}{2}, & \alpha &= \frac{r - 1}{2} + k.
\end{align*}
We then obtain constants $\eta_3 > 0$ and $T_3 > 0$ both depending only on $r$ and $n$ such that for all $T \geq T_3$, we have
\begin{align*}
|E_3| \leq \sum_{k = 0}^n \left|2^{\frac{r - 1}{2} - k} \binom{\frac{r - 1}{2}}{k} \beta_{r - 1} T^{\frac{r - 1}{2} - k} e^T\int_{\frac{3T}{2}}^{+\infty} e^{-x} x^{\frac{r - 1}{2} + k} \, dx\right| \leq e^{-\eta_3 T}.
\end{align*}

Combining the above three estimates gives \cref{eqn:E1E2E3Estimate}, concluding the proof.
\end{proof}

For an inner product space $V$ over $\R$, a fixed $v_0 \in V$ with $\|v_0\| = 1$, and $\tau \in (0, 2]$, we denote a circular open cone about $v_0$ by
\begin{align}
\label{eqn:ConeDefinition}
\mathcal{C}^V_{v_0, \tau} = \biggl\{v \in V \setminus \{0\}: \biggl\|\frac{v}{\|v\|} - v_0\biggr\| < \tau\biggr\}.
\end{align}
Note that $\mathcal{C}^V_{v_0, \tau}$ is an open convex cone if and only if $\tau \in (0, \sqrt{2}]$. If $\dim(V) = 1$, then we omit the parameter $\tau$ since it is redundant and simply write $\mathcal{C}^V_{v_0}$.

\begin{corollary}
\label{cor:LinearAlgebraConeIntegralExpansion}
Let $V$ be an inner product space over $\R$ of dimension $r := \dim(V)$, $\lambda \in V^* \setminus \{0\}$, and $\delta_\lambda := \lambda(v_\lambda) = \max_{v \in \overline{V_1}} \lambda(v) > 0$ for some $v_\lambda \in V$ with $\|v_\lambda\| = 1$, and $\tau \in (0, \sqrt{2}]$. Then, we have the following:
\begin{enumerate}
\item if $r = 1$, then for all $T > 0$, we have
\begin{align*}
\int_{\mathcal{C}^V_{v_\lambda} \cap V_T} e^{\lambda(v)} \, dv = \frac{1}{\delta_\lambda}\bigl(e^{\delta_\lambda T} - 1\bigr);
\end{align*}
\item if $r \geq 2$, then
\begin{enumerate}
\item for all $T > 0$, we have
\begin{align*}
\left|\int_{V_T \setminus \mathcal{C}^V_{v_\lambda, \tau}} e^{\lambda(v)} \, dv\right| \leq \beta_rT^re^{\delta_\lambda(1 - \frac{\tau^2}{2})T};
\end{align*}
\item there exists $\{\omega_k\}_{k = 0}^\infty \subset \R$ defined as in \cref{eqn:omega_k} such that for all $n \in \N$, there exists $C > 0$ independent of $\tau$ such that for all $T > 0$, we have
\begin{multline*}
\left|\int_{\mathcal{C}^V_{v_\lambda, \tau} \cap V_T} e^{\lambda(v)} \, dv - \sum_{k = 0}^n \omega_k T^{\frac{r - 1}{2} - k} e^{\delta_\lambda T}\right| \\
\leq CT^{\frac{r - 1}{2} - (n + 1)} e^{\delta_\lambda T} + \beta_r T^re^{\delta_\lambda(1 - \frac{\tau^2}{2})T}.
\end{multline*}
\end{enumerate}
\end{enumerate}
\end{corollary}

\begin{proof}
Let $V$, $r$, $\lambda$, $\delta_\lambda$, $v_\lambda$, $\tau$, and $T$ be as in the lemma. First suppose that $r = 1$ for property~(1). We can directly calculate that
\begin{align*}
\int_{\mathcal{C}^V_{v_\lambda} \cap V_T} e^{\lambda(v)} \, dv = \int_0^T e^{\lambda(xv_\lambda)} \, dx = \int_0^T e^{\delta_\lambda x} \, dx = \frac{1}{\delta_\lambda}\bigl(e^{\delta_\lambda T} - 1\bigr).
\end{align*}
Now suppose $r \geq 2$. For all $v \in \overline{V_1} \setminus \mathcal{C}^V_{v_\lambda, \tau}$, we have
\begin{align*}
\lambda(v) = \delta_\lambda\langle v, v_\lambda\rangle = \delta_\lambda\|v\|\cos(\theta(v, v_\lambda)) = \delta_\lambda\|v\|\left(1 - \frac{\|v - v_\lambda\|^2}{2}\right) \leq \delta_\lambda\left(1 - \frac{\tau^2}{2}\right)
\end{align*}
where $\theta(v, v_\lambda) \in (0, \frac{\pi}{2}]$ is the angle between the vectors $v$ and $v_\lambda$. Here, we have used the cosine rule to obtain the third equality. The final inequality is valid since $\tau \in (0, \sqrt{2}]$ or equivalently since $\mathcal{C}^V_{v_\lambda, \tau}$ is an open convex cone. Moreover, the inequality is sharp since it is an equality if and only if $v \in \partial\mathcal{C}^V_{v_\lambda, \tau}$ with $\|v\| = 1$. Hence, we calculate that
\begin{align*}
\left|\int_{V_T \setminus \mathcal{C}^V_{v_\lambda, \tau}} e^{\lambda(v)} \, dv\right| &\leq \vol\bigl(V_T \setminus \mathcal{C}^V_{v_\lambda, \tau}\bigr) \max_{v \in \overline{V_T} \setminus \mathcal{C}^V_{v_\lambda, \tau}} e^{\lambda(v)} \\
&\leq \vol(V_T) e^{\delta_\lambda(1 - \frac{\tau^2}{2})T} \\
&\leq \beta_rT^re^{\delta_\lambda(1 - \frac{\tau^2}{2})T}
\end{align*}
which establishes property~(2)(a) of the corollary. Property~(2)(b) of the corollary follows by combining this estimate with \cref{lem:LinearAlgebraIntegralExpansion}.
\end{proof}

\subsection{Proof of the asymptotic formulas for Riemannian skew balls}
We will use \cref{cor:LinearAlgebraConeIntegralExpansion} to prove \cref{thm:SkewBallVolumeAsymptotic}.

\begin{proof}[Proof of \cref{thm:SkewBallVolumeAsymptotic}]
Let $g_1, g_2 \in G$ and $T > 0$. Throughout the proof, the dependence of implicit constants on $g_1$ and $g_2$ are all continuous. In light of the integral formula from \cref{eqn:IntegralFormulaForH} and \cref{eqn:CartanDecompositionOfRiemannianSkewBalls}, we first focus on
\begin{align*}
\int_{\LieA^+_T[g_1, g_2]} \xi(v) \, dv.
\end{align*}
For all $v \in \LieA^+$, we have
\begin{align}
\label{eqn:xiFormula}
\xi(v) = \prod_{\alpha \in \Phi^+} \sinh^{m_\alpha}(\alpha(v)) = \frac{e^{\sum_{\alpha \in \Phi^+} m_\alpha \alpha(v)}}{2^{\sum_{\alpha \in \Phi ^+} m_\alpha}} + \sum_{j = 1}^N c_je^{\lambda_j(v)} = \frac{e^{2\rho(v)}}{2^{m_{\Phi^+}}} + \sum_{j = 1}^N c_je^{\lambda_j(v)}
\end{align}
for some $N \in \N$, $\{c_j\}_{j = 1}^N \subset \R$, and $\{\lambda_j\}_{j = 1}^N \subset \LieA^*$. Recall the unit vector $v_{2\rho} \in \interior(\LieA^+)$ in the direction of maximal growth and the growth rate $\delta_{2\rho} > 0$ for $2\rho$. Note that
\begin{align}
\label{eqn:lambdaLessThanTworho}
\lambda_j(v) < 2\rho(v) \qquad \text{for all $v \in \interior(\LieA^+)$ and $1 \leq j \leq N$}.
\end{align}
In fact, if $\rankG = 1$, identifying $\LieA \cong \R$ as inner product spaces, we have a stronger statement: for $\eta_H := \min_{\alpha \in \Phi^+} \|\alpha\| \in (0, \delta_{2\rho})$, we have
\begin{align}
\label{eqn:RankOnelambdaLessThanTworho}
\lambda_j(v) \leq 2\rho(v) - \eta_H v \leq (\delta_{2\rho} - \eta_1)v \qquad \text{for all $v \in \LieA$ and $1 \leq j \leq N$}.
\end{align}
Using the above formula gives
\begin{align}
\label{eqn:IntegralOfXi}
\int_{\LieA^+_T[g_1, g_2]} \xi(v) \, dv = \frac{1}{2^{m_{\Phi^+}}}\int_{\LieA^+_T[g_1, g_2]} e^{2\rho(v)} \, dv + \sum_{j = 1}^N c_j \int_{\LieA^+_T[g_1, g_2]} e^{\lambda_j(v)} \, dv.
\end{align}
We mainly focus on the asymptotic of $\int_{\LieA_T[g_1, g_2]} e^{2\rho(v)} \, dv$ since we will see later that it contributes the main term and the summation term is negligible.

Let $v \in \LieA^+$ with $\|v\| = 1$. Define
\begin{align*}
s_v(T) &= \inf\{s > 0: sv \notin \LieA_T^+[g_1, g_2]\} \\
S_v(T) &= \sup\{s > 0: sv \in \LieA_T^+[g_1, g_2]\}
\end{align*}
which records the first and last times at which the radial line $\{sv: s > 0\}$ leaves the set $\LieA_T^+[g_1, g_2]$. Whenever $T > d(o, g_1g_2o)$, we have
\begin{align}
\label{eqn:FirstLieASkewBallContainment}
\{tv: v \in \LieA^+, 0 \leq t < s_v(T)\} \subset \LieA_T^+[g_1, g_2] \subset \{tv: v \in \LieA^+, 0 \leq t < S_v(T)\}.
\end{align}
Observe that by continuity, we have
\begin{align*}
d(o, g_1a_{s_v(T) v}g_2o) = d(o, g_1a_{S_v(T) v}g_2o) = T.
\end{align*}
By the triangle and reverse triangle inequalities and left $G$-invariance of the metric, we immediately get the estimates
\begin{align}
\label{eqn:T_sv_Sv_Estimate}
|T - s_v(T)| &\leq d(o, g_1o) + d(o, g_2o), & |T - S_v(T)| &\leq d(o, g_1o) + d(o, g_2o).
\end{align}
We will find the precise asymptotic formula for $s_v(T)$ and $S_v(T)$ in terms of $T$. We define the vector (see \cref{sec:Preliminaries} for the definition of $\involution$)
\begin{align*}
\varpi := \beta_{e^+}(g_1^{-1}o, o) + \involution(\beta_{e^-}(g_2o, o)).
\end{align*}
The bound $\|\varpi\| \ll d(o, g_1o) + d(o, g_2o)$ is useful. It can be obtained using Harish-Chandra's inequality as in the end of the proof of \cref{pro:PreciseGeometricBusemannFunction}. We calculate that
\begin{align*}
&|T - s_v(T) - \langle \varpi, v\rangle| \\
={}&|d(o, g_1a_{s_v(T) v}g_2o) - d(o, a_{s_v(T) v}o) - \langle \beta_{e^+}(g_1^{-1}o, o), v\rangle - \langle \involution(\beta_{e^-}(g_2o, o)), v\rangle| \\
\leq{}&|d(o, g_1a_{s_v(T) v}g_2o) - d(o, a_{s_v(T) v}g_2o) - \langle \beta_{e^+}(g_1^{-1}o, o), v\rangle| \\
{}&+ |d(o, a_{s_v(T) v}g_2o) - d(o, a_{s_v(T) v}o) - \langle \involution(\beta_{e^-}(g_2o, o)), v\rangle| \\
={}&|d(g_1^{-1}o, a_{s_v(T) v}g_2o) - d(o, a_{s_v(T) v}g_2o) - \langle \beta_{e^+}(g_1^{-1}o, o), v\rangle| \\
{}&+ |d(g_2o, a_{-s_v(T) v}o) - d(o, a_{-s_v(T) v}o) - \langle \involution(\beta_{e^-}(g_2o, o)), v\rangle| \\
={}&|d(g_1^{-1}o, a_{s_v(T) v}g_2o) - d(o, a_{s_v(T) v}g_2o) - \langle \beta_{e^+}(g_1^{-1}o, o), v\rangle| \\
{}&+ |d(g_2o, w_0a_{\Ad_{w_0}(-s_v(T) v)}w_0^{-1}o) - d(o, w_0a_{\Ad_{w_0}(-s_v(T) v)}w_0^{-1}o) \\
{}&- \langle \involution(\beta_{w_0e^+}(g_2o, o)), v\rangle| \\
={}&|d(g_1^{-1}o, a_{s_v(T) v}g_2o) - d(o, a_{s_v(T) v}g_2o) - \langle \beta_{e^+}(g_1^{-1}o, o), v\rangle| \\
{}&+ |d(w_0^{-1}g_2o, a_{s_v(T)\involution(v)}o) - d(o, a_{s_v(T)\involution(v)}o) - \langle \beta_{e^+}(w_0^{-1}g_2o, o), \involution(v)\rangle|
\end{align*}
where we have used properties of the Busemann function and $w_0 \in K$. Thus, applying \cref{cor:PreciseGeometricBusemannFunction} gives
\begin{align*}
|T - s_v(T) - \langle \varpi, v\rangle| \leq
\begin{cases}
O\bigl(e^{O(d(o, g_1o) + d(o, g_2o))}\bigr) e^{-\eta_1 s_v(T)}, & \rankG = 1 \\
O\bigl(e^{O(d(o, g_1o) + d(o, g_2o))}\bigr) s_v(T)^{-1}, & \rankG \geq 2
\end{cases}
\end{align*}
where we recall the constant $\eta_1 > 0$ from \cref{eqn:Eta1}. The same calculations hold for $S_v(T)$ in place of $s_v(T)$. Thus, using the estimates from \cref{eqn:T_sv_Sv_Estimate} in the right hand side of the above bound, for all $T \gg d(o, g_1o) + d(o, g_2o)$, we get
\begin{align}
\label{eqn:sv_Sv_Formula}
|T - s_v(T) - \langle \varpi, v\rangle| &\leq E_\rankG(T), & |T - S_v(T) - \langle \varpi, v\rangle| &\leq E_\rankG(T)
\end{align}
where for some appropriate $C_1 = O\bigl(e^{O(d(o, g_1o) + d(o, g_2o))}\bigr)$, we define
\begin{align*}
E_\rankG(T) :=
\begin{cases}
C_1 e^{-\eta_1 T}, & \rankG = 1 \\
C_1 T^{-1}, & \rankG \geq 2.
\end{cases}
\end{align*}

Now, we prepare by introducing the open convex cone $\mathcal{C}^\LieA_{v_{2\rho}, \tau} \subset \interior(\LieA^+)$ for some sufficiently small parameter $\tau \in (0, \sqrt{2}]$ so that the containment holds. Recall that the parameter $\tau$ is only relevant when $\rankG \geq 2$. For all $v \in \mathcal{C}^\LieA_{v_{2\rho}, \tau}$ with $\|v\| = 1$, the Cauchy--Schwarz inequality and the definition of the open convex cone from \cref{eqn:ConeDefinition} gives the bound
\begin{align}
\label{eqn:varpiInnerProductBound}
|\langle \varpi, v\rangle - \langle \varpi, v_{2\rho}\rangle| \leq \|\varpi\| \cdot \|v - v_{2\rho}\| < \|\varpi\|\tau.
\end{align}
Thus, when $T > d(o, g_1g_2o)$, combining \cref{eqn:FirstLieASkewBallContainment} with \cref{eqn:sv_Sv_Formula,eqn:varpiInnerProductBound}, we deduce the containments
\begin{align}
\label{eqn:SecondLieASkewBallContainment}
\begin{aligned}
&\mathcal{C}^\LieA_{v_{2\rho}, \tau} \cap \LieA_{T - \langle \varpi, v_{2\rho}\rangle - \|\varpi\|\tau - E_\rankG(T)} \\
\subset{}&\mathcal{C}^\LieA_{v_{2\rho}, \tau} \cap \LieA_T^+[g_1, g_2] \\
\subset{}&\mathcal{C}^\LieA_{v_{2\rho}, \tau} \cap \LieA_{T - \langle \varpi, v_{2\rho}\rangle + \|\varpi\|\tau + E_\rankG(T)}.
\end{aligned}
\end{align}
Notice that the radius in the subscript has a main term $T - \langle \varpi, v_{2\rho}\rangle$ which is fixed on both sides of the containment and an error term $\|\varpi\|\tau + E_\rankG(T)$ which changes sign on both sides of the containment. See \cref{fig:SkewBallVolumeCalculation}.

\begin{figure}
\centering
\includegraphics{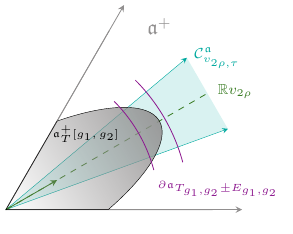}
\caption{This depicts the shrinking cone within which (the $\LieA^+$-component of) the Riemannian skew ball is sandwiched between two balls. Here, $T_{g_1, g_2} = T - \langle \varpi, v_{2\rho}\rangle$ and $E_{g_1, g_2} = \|\varpi\|\tau + E_\rankG(T)$.}
\label{fig:SkewBallVolumeCalculation}
\end{figure}

Now, we will focus on deriving the asymptotic formula for the integral in question in \cref{eqn:IntegralOfXi}. For the rest of the proof the implicit constant in the notation $O_{g_1, g_2}$ \emph{is always $O\bigl(e^{O(d(o, g_1o) + d(o, g_2o))}\bigr)$}. We do not explicitly write it for the sake of readability.

First suppose $\rankG = 1$. In this case, $\mathcal{C}^\LieA_{v'} \cap \LieA_{T'}[g_1, g_2] = \mathcal{C}^\LieA_{v'} \cap \LieA_{T'}^+[g_1, g_2] = \LieA_{T'}^+[g_1, g_2] \setminus \{0\}$ for all nonzero $v' \in \LieA^+$ and $T' > 0$. For $T \ll d(o, g_1o) + d(o, g_2o)$, the integral in question in \cref{eqn:IntegralOfXi} and the main term given by the theorem are just some uniformly bounded constants depending continuously on $g_1$ and $g_2$ of order $O\bigl(e^{O(d(o, g_1o) + d(o, g_2o))}\bigr)$ using \cref{eqn:T_sv_Sv_Estimate}, so assume $T \gg d(o, g_1o) + d(o, g_2o)$ with the same appropriate implicit constant of our choice as in the reverse inequality. The main term for the integral is contributed by the first term on the right hand side. Set $C_1' = C_1 + \|\varpi\| = O\bigl(e^{O(d(o, g_1o) + d(o, g_2o))}\bigr)$. Using \cref{eqn:SecondLieASkewBallContainment} and property~(1) in \cref{cor:LinearAlgebraConeIntegralExpansion}, setting the parameter $\tau = e^{-\eta_1 T}$, and using Taylor's theorem, we have
\begin{align}
\label{eqn:RankOneIntegralOfExpTwoRho}
\begin{aligned}
&\int_{\LieA_T^+[g_1, g_2]} e^{2\rho(v)} \, dv \leq \int_{\LieA_{T - \langle \varpi, v_{2\rho}\rangle + \|\varpi\|\tau + C_1e^{-\eta_1 T}}^+} e^{2\rho(v)} \, dv \\
={}&\frac{1}{\delta_{2\rho}}\bigl(e^{\delta_{2\rho}(T - \langle \varpi, v_{2\rho}\rangle + \|\varpi\|\tau + C_1e^{-\eta_1 T})} - 1\bigr) \\
\leq{}&\delta_{2\rho}^{-1}e^{-\delta_{2\rho}\langle \varpi, v_{2\rho}\rangle} e^{\delta_{2\rho}T} e^{\delta_{2\rho}C_1'e^{-\eta_1 T}} \\
={}&\delta_{2\rho}^{-1}e^{-\delta_{2\rho}\langle \varpi, v_{2\rho}\rangle} e^{\delta_{2\rho}T} \bigl(1 + O_{g_1, g_2}\bigl(e^{-\eta_1 T}\bigr)\bigr).
\end{aligned}
\end{align}
Henceforth, we introduce the notation of the form $C[g_1, g_2]$ for some varying symbol $C$ to indicate significant constants which depend continuously on $g_1$ and $g_2$. Thus, taking
\begin{align*}
\tilde{C}[g_1, g_2] := 2^{-m_{\Phi^+}}\delta_{2\rho}^{-1}e^{-\delta_{2\rho}\langle \varpi, v_{2\rho}\rangle} > 0
\end{align*}
and some other appropriate constant $\tilde{E}[g_1, g_2] = O\bigl(e^{O(d(o, g_1o) + d(o, g_2o))}\bigr)$, we obtain
\begin{align*}
\frac{1}{2^{m_{\Phi^+}}}\int_{\LieA_T^+[g_1, g_2]} e^{2\rho(v)} \, dv \leq \tilde{C}[g_1, g_2] e^{\delta_{2\rho}T} + \tilde{E}[g_1, g_2] e^{(\delta_{2\rho} - \eta_1)T}.
\end{align*}
We can calculate the reverse inequality in a similar fashion and adjust $\tilde{E}[g_1, g_2] = O\bigl(e^{O(d(o, g_1o) + d(o, g_2o))}\bigr)$ if necessary to obtain
\begin{align*}
\frac{1}{2^{m_{\Phi^+}}}\int_{\LieA_T^+[g_1, g_2]} e^{2\rho(v)} \, dv \geq \tilde{C}[g_1, g_2] e^{\delta_{2\rho}T} - \tilde{E}[g_1, g_2] e^{(\delta_{2\rho} - \eta_1)T}.
\end{align*}
Combining the two gives
\begin{align*}
\Biggl|\frac{1}{2^{m_{\Phi^+}}}\int_{\LieA_T^+[g_1, g_2]} e^{2\rho(v)} \, dv - \tilde{C}[g_1, g_2] e^{\delta_{2\rho}T}\Biggr| \leq \tilde{E}[g_1, g_2] e^{(\delta_{2\rho} - \eta_1)T}.
\end{align*}
The sum of integrals in \cref{eqn:IntegralOfXi} is negligible compared to the rest for the following reason. Using \cref{eqn:RankOnelambdaLessThanTworho} and calculating as in \cref{eqn:RankOneIntegralOfExpTwoRho} gives
\begin{align*}
\Biggl|\int_{\LieA^+_T[g_1, g_2]} e^{\lambda_j(v)} \, dv\Biggr| \leq \int_{\LieA^+_T[g_1, g_2]} e^{2\rho(v) - \eta_1 v} \, dv \leq \tilde{E}'[g_1, g_2]e^{(\delta_{2\rho} - \eta_1)T}
\end{align*}
for some constant $\tilde{E}'[g_1, g_2] > 0$ (in fact, we could have used the cruder upper bound $\vol(\LieA^+_T[g_1, g_2]) \max_{1 \leq j \leq N} \max_{v \in \overline{\LieA^+_T[g_1, g_2]}} e^{\lambda_j(v)}$ along with \cref{eqn:T_sv_Sv_Estimate}). Thus, combining the previous two inequalities and adjusting the constant $\tilde{E}[g_1, g_2] = O\bigl(e^{O(d(o, g_1o) + d(o, g_2o))}\bigr)$ gives
\begin{align*}
\Biggl|\int_{\LieA_T^+[g_1, g_2]} \xi(v) \, dv - \tilde{C}[g_1, g_2] e^{\delta_{2\rho}T}\Biggr| \leq \tilde{E}[g_1, g_2] e^{(\delta_{2\rho} - \eta_1)T}.
\end{align*}

Now suppose $\rankG \geq 2$. In this case we assume $T \geq \Omega\bigl(e^{\Omega(d(o, g_1o) + d(o, g_2o))}\bigr)$ for any appropriate implicit constants of our choice. The main term for the integral in \cref{eqn:IntegralOfXi} is contributed by the first term on the right hand side over the set $\mathcal{C}^\LieA_{v_{2\rho}, \tau} \cap \LieA_T^+[g_1, g_2]$. Using \cref{eqn:SecondLieASkewBallContainment} and property~(2)(b) in \cref{cor:LinearAlgebraConeIntegralExpansion}, and writing
\begin{align*}
\omega_0 := 2^{\frac{\rankH - 1}{2}}\delta_{2\rho}^{-\frac{\rankH + 1}{2}} \beta_{\rankH - 1} \int_0^{+\infty} e^{-x} x^{\frac{\rankH - 1}{2}} \, dx > 0
\end{align*}
as defined in \cref{eqn:omega_k} and $C_2 > 0$ for the constant $C$ from the corollary, we have
\begin{align}
\label{eqn:IntegralOfExpTwoRho}
\begin{aligned}
&\int_{\mathcal{C}^\LieA_{v_{2\rho}, \tau} \cap \LieA_T^+[g_1, g_2]} e^{2\rho(v)} \, dv \leq \int_{\mathcal{C}^\LieA_{v_{2\rho}, \tau} \cap \LieA_{T - \langle \varpi, v_{2\rho}\rangle + \|\varpi\|\tau + C_1T^{-1}}} e^{2\rho(v)} \, dv \\
\leq{}&\omega_0 \bigl(T - \langle \varpi, v_{2\rho}\rangle + \|\varpi\|\tau + C_1T^{-1}\bigr)^{\frac{\rankH - 1}{2}} e^{\delta_{2\rho} (T - \langle \varpi, v_{2\rho}\rangle + \|\varpi\|\tau + C_1T^{-1})} \\
{}&+ C_2\bigl(T - \langle \varpi, v_{2\rho}\rangle + \|\varpi\|\tau + C_1T^{-1}\bigr)^{\frac{\rankH - 3}{2}} e^{\delta_{2\rho} (T - \langle \varpi, v_{2\rho}\rangle + \|\varpi\|\tau + C_1T^{-1})} \\
{}&+ \beta_\rankH \bigl(T - \langle \varpi, v_{2\rho}\rangle + \|\varpi\|\tau + C_1T^{-1}\bigr)^\rankH e^{\delta_{2\rho}(1 - \frac{\tau^2}{2})(T - \langle \varpi, v_{2\rho}\rangle + \|\varpi\|\tau + C_1T^{-1})}.
\end{aligned}
\end{align}
Similarly, for the inequality in the other direction, we have
\begin{align}
\label{eqn:IntegralOfExpTwoRhoReverse}
\begin{aligned}
&\int_{\mathcal{C}^\LieA_{v_{2\rho}, \tau} \cap \LieA_T^+[g_1, g_2]} e^{2\rho(v)} \, dv \geq \int_{\mathcal{C}^\LieA_{v_{2\rho}, \tau} \cap \LieA_{T - \langle \varpi, v_{2\rho}\rangle - \|\varpi\|\tau - C_1T^{-1}}} e^{2\rho(v)} \, dv \\
\geq{}&\omega_0 \bigl(T - \langle \varpi, v_{2\rho}\rangle - \|\varpi\|\tau - C_1T^{-1}\bigr)^{\frac{\rankH - 1}{2}} e^{\delta_{2\rho} (T - \langle \varpi, v_{2\rho}\rangle - \|\varpi\|\tau - C_1T^{-1})} \\
{}&- C_2\bigl(T - \langle \varpi, v_{2\rho}\rangle - \|\varpi\|\tau - C_1T^{-1}\bigr)^{\frac{\rankH - 3}{2}} e^{\delta_{2\rho} (T - \langle \varpi, v_{2\rho}\rangle - \|\varpi\|\tau - C_1T^{-1})} \\
{}&- \beta_\rankH \bigl(T - \langle \varpi, v_{2\rho}\rangle - \|\varpi\|\tau - C_1T^{-1}\bigr)^\rankH e^{\delta_{2\rho}(1 - \frac{\tau^2}{2})(T - \langle \varpi, v_{2\rho}\rangle - \|\varpi\|\tau - C_1T^{-1})}.
\end{aligned}
\end{align}
In \cref{eqn:IntegralOfExpTwoRho,eqn:IntegralOfExpTwoRhoReverse}, we wish to show that by choosing $\tau$ appropriately, the main term of the asymptotic behavior is contributed simply by the leading term $T$ both in the middle factor and in the exponent of the last factor of the first term, and that all other terms contribute to either the error term or are negligible even compared to the error term. The term $\|\varpi\|\tau$ in the exponent is especially delicate and the factor $\tau$ is essential. If we trivially bound $|\langle \varpi, v\rangle| \leq \|\varpi\|$ instead of \cref{eqn:varpiInnerProductBound}, then due to its acquired \emph{sign difference} in \cref{eqn:IntegralOfExpTwoRho,eqn:IntegralOfExpTwoRhoReverse}, its two exponentials are constants, but each a reciprocal of the other. This is detrimental for the calculation of the precise main term. Thus, it was crucial to restrict to the open convex cone $\mathcal{C}^\LieA_{v_{2\rho}, \tau}$ and estimate the Busemann function. It was also crucial to use the precise form of the Busemann function and the simple but explicit estimate in \cref{eqn:varpiInnerProductBound}. Now, we can control the exponential contribution of the term $\|\varpi\|\tau$ by choosing $\tau$ appropriately. Namely, we will see later that it is sufficient to impose the criteria that $\tau$ is decreasing in $T$ in the fashion
\begin{align}
\label{eqn:tauConstriantI}
\tau = O(T^{-\epsilon})
\end{align}
for some $\epsilon > 0$ and an absolute implicit constant. However, we are not completely free to take any $\epsilon$ of our choice due to the last term in \cref{eqn:IntegralOfExpTwoRho,eqn:IntegralOfExpTwoRhoReverse} which is delicate to deal with.

Let us first focus on the last term in \cref{eqn:IntegralOfExpTwoRho} before returning to the first two terms. We wish to make it negligible compared to the rest by choosing $\tau$ appropriately. For this purpose also, it was crucial to restrict to the cone $\mathcal{C}^\LieA_{v_{2\rho}, \tau}$ and keep track of the precise improved factor $\bigl(1 - \frac{\tau^2}{2}\bigr)$ in the exponent of the last term in \cref{eqn:IntegralOfExpTwoRho} via property~(2)(b) in \cref{cor:LinearAlgebraConeIntegralExpansion}. Expanding the exponent of the last term in \cref{eqn:IntegralOfExpTwoRho}, it reads
\begin{align*}
&\delta_{2\rho}\biggl(1 - \frac{\tau^2}{2}\biggr)\bigl(T - \langle \varpi, v_{2\rho}\rangle + \|\varpi\|\tau + C_1T^{-1}\bigr) \\
={}&\delta_{2\rho}T - \delta_{2\rho}\langle \varpi, v_{2\rho}\rangle + \delta_{2\rho}\|\varpi\|\tau + \delta_{2\rho}C_1T^{-1} \\
{}&-2^{-1}\delta_{2\rho}\tau^2T + 2^{-1}\delta_{2\rho}\langle \varpi, v_{2\rho}\rangle\tau^2 - 2^{-1}\delta_{2\rho}\|\varpi\|\tau^3 - 2^{-1}\delta_{2\rho}C_1\tau^2T^{-1}.
\end{align*}
Thus, the last term in \cref{eqn:IntegralOfExpTwoRho} becomes
\begin{align*}
&\beta_\rankH e^{- \delta_{2\rho}\langle \varpi, v_{2\rho}\rangle} \cdot \underbrace{T^\rankH e^{-2^{-1}\delta_{2\rho}\tau^2T}}_{*} \cdot e^{\delta_{2\rho}T} \cdot \underbrace{\bigl(1 - \langle \varpi, v_{2\rho}\rangle T^{-1} + \|\varpi\|\tau T^{-1} + C_1T^{-2}\bigr)^\rankH}_{**}\\
{}&\cdot \underbrace{e^{\delta_{2\rho}\|\varpi\|\tau + \delta_{2\rho}C_1T^{-1} + 2^{-1}\delta_{2\rho}\langle \varpi, v_{2\rho}\rangle\tau^2 - 2^{-1}\delta_{2\rho}\|\varpi\|\tau^3 - 2^{-1}\delta_{2\rho}C_1\tau^2T^{-1}}}_{***}.
\end{align*}
Observe that in the * factor, we can control the exponent of the factor $T^\rankH$ using the other factor $e^{-2^{-1}\delta_{2\rho}\tau^2T}$. We do so in such a way that
\begin{align}
\label{eqn:StarFactor}
T^\rankH e^{-2^{-1}\delta_{2\rho}\tau^2T} \leq T^{\frac{\rankH - 3 - c}{2}}
\end{align}
for some $c > 0$, thereby achieving the objective that the last term in \cref{eqn:IntegralOfExpTwoRho} is negligible compared to the rest. We simply use $c = \rankH - 3$ for convenience; we see below that the choice of $c$ makes no difference up to a constant factor. Solving the above inequality, it holds provided that
\begin{align}
\label{eqn:tauConstriantII}
\tau \geq 2^{\frac{1}{2}}\rankH^{\frac{1}{2}}\delta_{2\rho}^{-\frac{1}{2}} T^{-\frac{1}{2}}\log(T)^{\frac{1}{2}}.
\end{align}
For the \emph{optimal} choice of parameter $\tau$ satisfying both the constraints in \cref{eqn:tauConstriantII,eqn:tauConstriantI} so that $\epsilon$ can be taken as large as possible, namely any $\epsilon \in \bigl(0, \frac{1}{2}\bigr)$, we set \cref{eqn:tauConstriantII} to equality:
\begin{align}
\label{eqn:tauConstriantIIEquality}
\tau = 2^{\frac{1}{2}}\rankH^{\frac{1}{2}}\delta_{2\rho}^{-\frac{1}{2}} T^{-\frac{1}{2}}\log(T)^{\frac{1}{2}}.
\end{align}
Note that we can still ensure that  $\tau \in (0, \sqrt{2}]$ is sufficiently small such that $\mathcal{C}^\LieA_{v_{2\rho}, \tau} \subset \interior(\LieA^+)$ since $T \geq \Omega\bigl(e^{\Omega(d(o, g_1o) + d(o, g_2o))}\bigr)$. Although this optimization does not affect the analysis of the last term in \cref{eqn:IntegralOfExpTwoRho}, it is of significance for the larger first two terms. We already have the estimate \cref{eqn:StarFactor} for the * factor by construction. Let us analyze the ** and *** factors. Using \cref{eqn:tauConstriantIIEquality} and Taylor's theorem, the ** factor is bounded above by
\begin{align*}
\bigl(1 + O_{g_1, g_2}\bigl(T^{-1}\bigr)\bigr)^\rankH = 1 + O_{g_1, g_2}\bigl(T^{-1}\bigr).
\end{align*}
Using \cref{eqn:tauConstriantIIEquality}, the worst term in the exponent of the *** factor is
\begin{align*}
\delta_{2\rho}\|\varpi\|\tau = O_{g_1, g_2}\bigl(T^{-\frac{1}{2}}\log(T)^{\frac{1}{2}}\bigr).
\end{align*}
Thus, we can guarantee that the exponent in the *** factor is sufficiently small by choice of the implicit constants in $T \geq \Omega\bigl(e^{\Omega(d(o, g_1o) + d(o, g_2o))}\bigr)$. Hence, by Taylor's theorem, the *** factor is bounded above by
\begin{align*}
e^{O_{g_1, g_2}(T^{-\frac{1}{2}}\log(T)^{\frac{1}{2}})} = 1 + O_{g_1, g_2}\bigl(T^{-\frac{1}{2}}\log(T)^{\frac{1}{2}}\bigr).
\end{align*}
Thus, compiling the components in the above discussion, the last term in \cref{eqn:IntegralOfExpTwoRho} is bounded above by
\begin{align*}
O_{g_1, g_2}\bigl(e^{\delta_{2\rho}T}\bigr).
\end{align*}

We now return to the first two terms in \cref{eqn:IntegralOfExpTwoRho}. Similar to above, using \cref{eqn:tauConstriantIIEquality} and Taylor's theorem, we calculate that the first term is bounded above by
\begin{align*}
&\omega_0 \Bigl(T - \langle \varpi, v_{2\rho}\rangle + 2^{\frac{1}{2}}\rankH^{\frac{1}{2}}\delta_{2\rho}^{-\frac{1}{2}}\|\varpi\| T^{-\frac{1}{2}}\log(T)^{\frac{1}{2}} + C_1T^{-1}\Bigr)^{\frac{\rankH - 1}{2}} \\
{}&\cdot e^{\delta_{2\rho}T - \delta_{2\rho}\langle \varpi, v_{2\rho}\rangle + 2^{\frac{1}{2}}\rankH^{\frac{1}{2}}\delta_{2\rho}^{\frac{1}{2}}\|\varpi\| T^{-\frac{1}{2}}\log(T)^{\frac{1}{2}} + \delta_{2\rho}C_1T^{-1}} \\
={}&\omega_0 e^{-\delta_{2\rho}\langle \varpi, v_{2\rho}\rangle} T^{\frac{\rankH - 1}{2}} e^{\delta_{2\rho}T} \bigl(1 + O_{g_1, g_2}\bigl(T^{-1}\bigr)\bigr)^{\frac{\rankH - 1}{2}} e^{O_{g_1, g_2}(T^{-\frac{1}{2}}\log(T)^{\frac{1}{2}})} \\
={}&\omega_0 e^{-\delta_{2\rho}\langle \varpi, v_{2\rho}\rangle} T^{\frac{\rankH - 1}{2}} e^{\delta_{2\rho}T} \bigl(1 + O_{g_1, g_2}\bigl(T^{-1}\bigr)\bigr) \bigl(1 + O_{g_1, g_2}\bigl(T^{-\frac{1}{2}}\log(T)^{\frac{1}{2}}\bigr)\bigr) \\
={}&\omega_0 e^{-\delta_{2\rho}\langle \varpi, v_{2\rho}\rangle} T^{\frac{\rankH - 1}{2}} e^{\delta_{2\rho}T} \bigl(1 + O_{g_1, g_2}\bigl(T^{-\frac{1}{2}}\log(T)^{\frac{1}{2}}\bigr)\bigr) \\
={}&\omega_0 e^{-\delta_{2\rho}\langle \varpi, v_{2\rho}\rangle} T^{\frac{\rankH - 1}{2}} e^{\delta_{2\rho}T} + O_{g_1, g_2}\bigl(\log(T)^{\frac{1}{2}}T^{\frac{\rankH - 2}{2}}e^{\delta_{2\rho}T}\bigr).
\end{align*}
By a similar calculation, the second term in \cref{eqn:IntegralOfExpTwoRho} is bounded above by
\begin{align*}
C_2 e^{-\delta_{2\rho}\langle \varpi, v_{2\rho}\rangle} T^{\frac{\rankH - 3}{2}} e^{\delta_{2\rho}T} + O_{g_1, g_2}\bigl(\log(T)^{\frac{1}{2}}T^{\frac{\rankH - 4}{2}}e^{\delta_{2\rho}T}\bigr).
\end{align*}

Combining the estimates and taking
\begin{align*}
\tilde{C}[g_1, g_2] := 2^{-m_{\Phi^+}}\omega_0 e^{-\delta_{2\rho}\langle \varpi, v_{2\rho}\rangle} > 0
\end{align*}
and some other appropriate constant $\tilde{E}[g_1, g_2] = O\bigl(e^{O(d(o, g_1o) + d(o, g_2o))}\bigr)$, \cref{eqn:IntegralOfExpTwoRho} becomes
\begin{align*}
\frac{1}{2^{m_{\Phi^+}}}\int_{\mathcal{C}^\LieA_{v_{2\rho}, \tau} \cap \LieA_T^+[g_1, g_2]} e^{2\rho(v)} \, dv \leq{}&\tilde{C}[g_1, g_2] T^{\frac{\rankH - 1}{2}} e^{\delta_{2\rho}T} \\
{}&+ \tilde{E}[g_1, g_2]\log(T)^{\frac{1}{2}}T^{\frac{\rankH - 2}{2}}e^{\delta_{2\rho}T}.
\end{align*}
Repeating calculations similar to above and adjusting the constant $\tilde{E}[g_1, g_2] = O\bigl(e^{O(d(o, g_1o) + d(o, g_2o))}\bigr)$ if necessary, \cref{eqn:IntegralOfExpTwoRhoReverse} becomes
\begin{align*}
\frac{1}{2^{m_{\Phi^+}}}\int_{\mathcal{C}^\LieA_{v_{2\rho}, \tau} \cap \LieA_T^+[g_1, g_2]} e^{2\rho(v)} \, dv \geq{}&\tilde{C}[g_1, g_2] T^{\frac{\rankH - 1}{2}} e^{\delta_{2\rho}T} \\
{}&- \tilde{E}[g_1, g_2]\log(T)^{\frac{1}{2}}T^{\frac{\rankH - 2}{2}}e^{\delta_{2\rho}T}.
\end{align*}
Combining the two gives
\begin{multline*}
\Biggl|\frac{1}{2^{m_{\Phi^+}}}\int_{\mathcal{C}^\LieA_{v_{2\rho}, \tau} \cap \LieA_T^+[g_1, g_2]} e^{2\rho(v)} \, dv - \tilde{C}[g_1, g_2] T^{\frac{\rankH - 1}{2}} e^{\delta_{2\rho}T}\Biggr| \\
\leq \tilde{E}[g_1, g_2]\log(T)^{\frac{1}{2}}T^{\frac{\rankH - 2}{2}}e^{\delta_{2\rho}T}.
\end{multline*}

Let us show that all the integrals on the right hand side of \cref{eqn:IntegralOfXi} are negligible over the set $\LieA_T^+[g_1, g_2] \setminus \mathcal{C}^\LieA_{v_{2\rho}, \tau}$ compared to the rest. Replacing \cref{eqn:varpiInnerProductBound} by the trivial bound $|\langle \varpi, v\rangle| \leq \|\varpi\|$ as we are no longer restricting to $\mathcal{C}^\LieA_{v_{2\rho}, \tau}$, and using \cref{eqn:tauConstriantIIEquality,eqn:StarFactor} and Taylor's theorem, it follows from property~(2)(a) in \cref{cor:LinearAlgebraConeIntegralExpansion} that
\begin{align}\label{eqn:estimate_out_cone}
\begin{aligned}
&\left|\int_{\LieA_T^+[g_1, g_2] \setminus \mathcal{C}^\LieA_{v_{2\rho}, \tau}} e^{2\rho(v)} \, dv\right| \leq \left|\int_{\LieA_{T + \|\varpi\| + C_1T^{-1}} \setminus \mathcal{C}^\LieA_{v_{2\rho}, \tau}} e^{2\rho(v)} \, dv\right| \\
\leq{}&\beta_\rankH (T + \|\varpi\| + C_1T^{-1})^\rankH e^{\delta_{2\rho}(1 - \frac{\tau^2}{2})(T + \|\varpi\| + C_1T^{-1})} \\
={}&\beta_\rankH e^{\delta_{2\rho}\|\varpi\|} e^{\delta_{2\rho} T}\bigl(1 + O_{g_1, g_2}\bigl(T^{-1}\bigr)\bigr)^\rankH \\
{}&\cdot e^{\delta_{2\rho} C_1T^{-1} - \rankH\|\varpi\|T^{-1}\log(T) - \rankH C_1T^{-2}\log(T)} \\
={}&\beta_\rankH e^{\delta_{2\rho}\|\varpi\|} e^{\delta_{2\rho} T}\bigl(1 + O_{g_1, g_2}\bigl(T^{-1}\bigr)\bigr) e^{O_{g_1, g_2}(T^{-1}\log(T))} \\
={}&\beta_\rankH e^{\delta_{2\rho}\|\varpi\|} e^{\delta_{2\rho} T}\bigl(1 + O_{g_1, g_2}\bigl(T^{-1}\bigr)\bigr) \bigl(1 + O_{g_1, g_2}\bigl(T^{-1}\log(T)\bigr)\bigr) \\
={}&O_{g_1, g_2}\bigl(e^{\delta_{2\rho} T}\bigr).
\end{aligned}
\end{align}

Let us also show that the integrals in the sum on the right hand side of \cref{eqn:IntegralOfXi} are negligible over the set $\LieA_T^+[g_1, g_2]$ compared to the rest. By \cref{eqn:lambdaLessThanTworho}, there exists $\eta \in (0, \delta_{2\rho})$ such that
\begin{align*}
\max_{1 \leq j \leq N} \max_{v \in \overline{\LieA_1^+}} \lambda_j(v) = \delta_{2\rho} - \eta.
\end{align*}
Again using $|\langle \varpi, v\rangle| \leq \|\varpi\|$ and using Taylor's theorem, for all $1 \leq j \leq N$, we get
\begin{align}\label{eqn:estimate_otherroot}
\begin{aligned}
&\left|\int_{\LieA_T^+[g_1, g_2]} e^{\lambda_j(v)} \, dv\right| \leq \int_{\LieA_{T + \|\varpi\| + C_1T^{-1}}^+} e^{\lambda_j(v)} \, dv \\
\leq{}&\vol\Bigl(\LieA_{T + \|\varpi\| + C_1T^{-1}}^+\Bigr) \max_{1 \leq j \leq N} \max_{v \in \overline{\LieA_{T + \|\varpi\| + C_1T^{-1}}^+}} e^{\lambda_j(v)} \\
\leq{}&\vol\bigl(\LieA_{T + \|\varpi\| + C_1T^{-1}}\bigr)e^{(\delta_{2\rho} - \eta)(T + \|\varpi\| + C_1T^{-1})} \\
={}&\beta_\rankH e^{(\delta_{2\rho} - \eta)\|\varpi\|} \bigl(T + \|\varpi\| + C_1T^{-1}\bigr)^\rankH e^{(\delta_{2\rho} - \eta)T} e^{O(T^{-1})} \\
={}&\beta_\rankH e^{(\delta_{2\rho} - \eta)\|\varpi\|} O_{g_1, g_2}(T^\rankH) e^{(\delta_{2\rho} - \eta)T} \bigl(1 + O\bigl(T^{-1}\bigr)\bigr) \\
={}&O_{g_1, g_2}\bigl(T^\rankH e^{(\delta_{2\rho} - \eta)T}\bigr).
\end{aligned}
\end{align}

Thus, combining the above inequalities and adjusting the constant $\tilde{E}[g_1, g_2] = O\bigl(e^{O(d(o, g_1o) + d(o, g_2o))}\bigr)$, we have
\begin{align*}
\Biggl|\int_{\LieA_T^+[g_1, g_2]} \xi(v) \, dv - \tilde{C}[g_1, g_2] T^{\frac{\rankH - 1}{2}} e^{\delta_{2\rho}T}\Biggr| \leq \tilde{E}[g_1, g_2]\log(T)^{\frac{1}{2}}T^{\frac{\rankH - 2}{2}}e^{\delta_{2\rho}T}.
\end{align*}

Finally, in both cases $\rankG = 1$ and $\rankG \geq 2$, defining
\begin{align*}
C[g_1, g_2] := \frac{1}{\mu_M(M)}\int_K \int_K \tilde{C}[g_1k_1, k_2g_2] \, d\mu_K(k_1) \, d\mu_K(k_2),
\end{align*}
and using \cref{eqn:IntegralFormulaForH,eqn:CartanDecompositionOfRiemannianSkewBalls} produces our desired formula of the theorem.
\end{proof}

\begin{remark}
\label{rem:OptimalErrorTerm}
Actually, when choosing $\tau$ as a function of $T$, we set the ``function to beat'' to be a polynomial which is at worst $T^{\frac{\rankH - 3}{2}}$ but we see later in the proof that the ``new function to beat'' is $\log(T)^{\frac{1}{2}}T^{\frac{\rankH - 2}{2}}$ which comes from the first term in \cref{eqn:IntegralOfExpTwoRho}, and one can repeat the process recursively. Alternatively, one can directly set the expression to beat to be $T^{\frac{\rankH - 1}{2}}\tau$. The most optimal choice for $\tau$ is then the solution to $T^\rankH e^{-2^{-1}\delta_{2\rho}\tau^2T} = T^{\frac{\rankH - 1}{2}}\tau$. Using the Lambert $W$ function, this can be solved explicitly as $\tau = \delta_{2\rho}^{-\frac{1}{2}}T^{-\frac{1}{2}} W(\delta_{2\rho}T^{\rankH + 2})^{\frac{1}{2}}$. Thus for the optimal error term in \cref{thm:SkewBallVolumeAsymptotic}, the factor $\log(T)^{\frac{1}{2}}$ is to be replaced with $W(\delta_{2\rho}T^{\rankH + 2})^{\frac{1}{2}}$. This can be carried through all the way to \cref{thm:MainTheorem}.
\end{remark}

\section{\texorpdfstring{Effective equidistribution of $K_H$-orbits from that of $U_H$-orbits}{Effective equidistribution of K\unichar{"005F}H-orbits from that of U\unichar{"005F}H-orbits}}
\label{sec:K-Equidistribution}
In this section, we prove \cref{thm:K_EffectiveEquidistribution} regarding effective equidistribution of $K_H$-orbits using the effective equidistribution of $U_H$-orbits given by \ShahEE, generalizing a theorem of Lindenstrauss--Mohammadi--Wang \cite[Theorem 1.4]{LMW23} which is in the special case $(G, H) = (\SL_2(\R) \times \SL_2(\R), \Delta(\SL_2(\R)))$. Its proof is also inspired by that of the latter. We also note that this effectivizes the argument for the passage from Shah's theorem \cite[Theorem 1.4]{Sha96} (proved using Ratner's theorem) to its corollary \cite[Corollary 1.2]{Sha96}.

\begin{notation}
As in \cref{sec:VolumeCalculations}, we again \emph{drop the subscript $H$} for all the objects introduced in \cref{sec:Preliminaries} (except measures and ranks) for convenience, unless otherwise mentioned.
\end{notation}

Recall the positive constant $c_{\Phi} = \max_{\alpha \in \Phi} \|\alpha\| > 0$ from \cref{eqn:Constant_cPhi}.

\begin{theorem}
\label{thm:K_EffectiveEquidistribution}
Suppose \ShahEE holds. Then, there exist $\kappa \in (0, \kappa_0)$, $\varrho \in (\varrho_0, 2\varrho_0)$, and a decreasing family $\{C_\varsigma := c_\varsigma + c_{\Phi} + 1\}_{\varsigma > 0}$ such that for all $x_0 \in \Gamma \backslash G$, $R \gg \inj_{\Gamma \backslash G}(x_0)^{-\varrho}$, $v \in \interior(\LieA^+)$ with $\|v\| = 1$ and $\varsigma := \min_{\alpha \in \Phi^+} \alpha(v)$, and $t \geq C_\varsigma\log(R)$, at least one of the following holds.
\begin{enumerate}
\item For all $\phi \in C_{\mathrm{c}}^\infty(\Gamma \backslash G)$ and $\varphi \in C^\infty(K)$, we have
\begin{align*}
\left|\int_K \phi(x_0ka_{tv}) \varphi(k) \, d\mu_K(k) - \int_{\Gamma \backslash G} \phi \, d\hat{\mu}_{\Gamma \backslash G} \cdot \int_K \varphi \, d\mu_K\right| \leq \mathcal{S}^\ell(\phi) \mathcal{S}^\ell(\varphi) R^{-\varsigma\kappa}.
\end{align*}
\item There exists $x \in \Gamma \backslash G$ with
\begin{align*}
d(x_0, x) \leq R^{C_\varsigma} t^{C_\varsigma} e^{-\varsigma t}
\end{align*}
such that $xH$ is periodic with $\vol(xH) \leq R$.
\end{enumerate}
\end{theorem}

\subsection{\texorpdfstring{Decomposition of $\LieK$ and estimates}{Decomposition of \unichar{"1D528} and estimates}}
To prove \cref{thm:K_EffectiveEquidistribution}, we need to define the following maps and establish a lemma on a Lie algebra decomposition of $\LieK$. Let $\pi_{\LieU^+}$, $\pi_{\LieU^-}$, $\pi_{\LieM}$, and $\pi_{\LieA}$ be the projections to each component with respect to the restricted root space decomposition $\LieG = \LieU^+ \oplus \LieU^- \oplus \LieA \oplus \LieM$. We introduce the subspace
\begin{align*}
\LieK^\star = \LieK \cap (\LieU^+ \oplus \LieU^-) \subset \LieK
\end{align*}
which is of central importance. Let
\begin{align*}
\pi^{+} = \pi_{\LieU^+}|_{\LieK^\star}: \LieK^\star \to \LieU^+.
\end{align*}
Let $\iota^{+}: \LieU^+ \to \LieK^\star$ be the map defined by
\begin{align*}
\iota^{+}(w) = w + \theta(w) \qquad \text{for all $w \in \LieU^+$}.
\end{align*}
Let us denote $\exp(s^+)$, $\exp(s^-)$, $\exp(\tau)$, and $\exp(\xi)$ by $u^+_{s^+}$, $u^-_{s^-}$, $a_{\tau}$, and $m_{\xi}$, for $s^+ \in \LieU^+$, $s^- \in \LieU^-$, $\tau \in \LieA$, and $\xi \in \LieM$, respectively.
Due to the restricted root space decomposition of $\LieG$ and the implicit function theorem, we obtain smooth maps
\begin{align*}
s^+&: \LieK^\star \to \LieU^+, & s^-&: \LieK^\star \to \LieU^-, \\
\tau&: \LieK^\star \to \LieA, & \xi&: \LieK^\star \to \LieM,
\end{align*}
such that
\begin{align*}
\exp(\omega) = u^{+}_{s^{+}(\omega)}u^{-}_{s^{-}(\omega)}a_{\tau(\omega)}m_{\xi(\omega)} \qquad \text{for all $\omega \in \LieK^\star$ with $\|\omega\| \leq \epsilon_G$}.
\end{align*}

\begin{lemma}\label{lem:Lie_transverse}
We have 
\begin{enumerate}
\item $\pi^+ \circ \iota^+ = \Id_{\LieU^+}$ and $\iota^{+} \circ \pi^{+} = \Id_{\LieK^\star}$, i.e., $\pi^+$ and $\iota^+$ are inverses of each other;
\item $\LieK = \LieK^\star \oplus \LieM$ orthogonally;
\item $\|\omega\| = \sqrt{2}\|\pi^+(\omega)\|$ for all $\omega \in \LieK^\star$.
\end{enumerate}
For all $\omega \in \LieK^\star$ with $\|\omega\| \ll 1$, we have
\begin{enumerate}\setcounter{enumi}{3}
\item $s^{+}(\omega) = \pi^{+}(\omega) + O(\|\omega\|^2) = \pi^{+}(\omega) + O(\|\pi^{+}(\omega)\|^2)$;
\item $s^{-}(\omega) = \theta(\pi^{+}(\omega)) + O(\|\omega\|^2) = \theta(\pi^{+}(\omega)) + O(\|\pi^{+}(\omega)\|^2)$;
\item $\tau(\omega) = O(\|\omega\|^2)$;
\item $\xi(\omega) = O(\|\omega\|^2)$.
\end{enumerate}
\end{lemma}

\begin{proof}
First, $\pi^+ \circ \iota^+ = \Id_{\LieU^+}$ follows trivially from definitions and $\theta(\LieU^+) = \LieU^-$ \cite[Chapter VI, \S 3, Lemma 3.3]{Hel01}. Now we show $\iota^{+} \circ \pi^{+} = \Id_{\LieK^\star}$. Note that for all $\omega \in \LieK^\star = \LieK \cap (\LieU^+ \oplus \LieU^-)$, we have $\omega - \iota^{+}(\pi^{+}(\omega)) \in \LieK \cap \LieU^-$ again using $\theta(\LieU^+) = \LieU^-$. But $\LieK \cap \LieU^- = \theta(\LieK \cap \LieU^-) =  \LieK \cap \LieU^+$. Therefore, $\LieK \cap \LieU^-  \subset \LieU^+ \cap \LieU^- = 0$, which proves property~(1).

For property~(2), we write every element $\omega \in \LieK$ as $\omega = \omega_{\LieA} + \omega_{\LieM} + \omega_{\LieU^+} + \omega_{\LieU^-}$ uniquely with respect to the restricted root space decomposition. Now, $\omega \in \LieK = \Fix(\theta)$ implies
\begin{align*}
\omega_{\LieA} + \omega_{\LieM} + \omega_{\LieU^+} + \omega_{\LieU^-} = -\omega_{\LieA} + \omega_{\LieM} + \theta(\omega_{\LieU^-}) + \theta(\omega_{\LieU^+})
\end{align*}
where $\theta(\omega_{\LieU^{\pm}}) \in \LieU^{\mp}$. Therefore, we have $\omega_{\LieA} = 0$ and $\theta(\omega_{\LieU^+}) = \omega_{\LieU^-}$, which proves property~(2). 

For property~(3), we recall the fact that the restricted root space decomposition is orthogonal and $\theta$ is an orthogonal involution, both with respect to $\langle \cdot, \cdot \rangle = B_\theta$. Then, for all $\omega \in \LieK^\star = \LieK \cap (\LieU^+ \oplus \LieU^-)$, we have
\begin{align*}
\|\omega\|^2 = \|\pi^{+}(\omega)\|^2 + \|\theta(\pi^{+}(\omega))\|^2 = 2\|\pi^{+}(\omega)\|^2
\end{align*}
which proves property~(3). 

Now, we prove the estimates in properties~(4), (5), (6), and (7). Define the smooth map $f: \LieG \to G$ by
\begin{align*}
f(x) = \exp(x_{\LieU^+})\exp(x_{\LieU^-})\exp(x_{\LieA})\exp(x_{\LieM})
\end{align*}
for all $x = x_{\LieU^+} + x_{\LieU^-} + x_{\LieA} + x_{\LieM} \in \LieG$ with respect to the restricted root space decomposition. In fact, it is a diffeomorphism on some sufficiently small open neighborhood $\mathcal{O} \subset B_{\epsilon_G}^\LieG(0)$ of $0 \in \LieG$ since $df_0 = \Id_\LieG$. We may assume $f(\mathcal{O}) \subset \exp\bigl(B_{\epsilon_G}^\LieG(0)\bigr)$ and take $\mathcal{O}' = \exp^{-1}(f(\mathcal{O})) \subset B_{\epsilon_G}^\LieG(0)$. Define the smooth map $\psi = (f|_{\mathcal{O}})^{-1} \circ \exp|_{\mathcal{O}'}: \mathcal{O}' \to \LieG$. Its derivative at $0 \in \mathcal{O}'$ is
\begin{align*}
d \psi_0 = d((f|_{\mathcal{O}})^{-1})_e \circ d\exp_0 = (df_0)^{-1} \circ d \exp_0 = \Id_\LieG \circ \Id_\LieG = \Id_\LieG.
\end{align*}
Now, $\psi(\omega) = s^{+}(\omega) + s^{-}(\omega) + \tau(\omega) + \xi(\omega)$ for all $\omega \in \LieK^\star$ with $\|\omega\| \ll 1$. So, by Taylor's theorem, for all $\omega \in \LieK^\star$ with $\|\omega\| \ll 1$, we have 
\begin{align*}
s^+(\omega) ={}& \pi_{\LieU^+}(\omega) + O(\|\omega\|^2), & s^-(\omega) ={}& \pi_{\LieU^-}(\omega) + O(\|\omega\|^2), \\
\tau(\omega) ={}& \pi_{\LieA}(\omega) + O(\|\omega\|^2), & \xi(\omega) ={}& \pi_{\LieM}(\omega) + O(\|\omega\|^2).
\end{align*}
Using definitions and properties~(2) and (3), we get the estimates. 
\end{proof}

\subsection{\texorpdfstring{Proof of effective equidistribution of $K$-orbits}
{Proof of effective equidistribution of K-orbits}}
Before proving \cref{thm:K_EffectiveEquidistribution}, we need the following lemma on the estimate of the expanding rate under the flow $\{a_{tv}\}_{t \in \mathbb{R}}$. Although it is not required for our purposes, the lemma holds for objects associated to $G$ and hence we include subscript $G$ for clarification. Recall the constant $c_\Phi$ from \cref{eqn:Constant_cPhi}.

\begin{lemma}\label{lem:maxexpand}
For all $x_1, x_2 \in \Gamma \backslash G$, $v \in \LieA_G$ with $\|v\| = 1$, $k \in K_G$, and $\tau \geq 2$, the following inequality holds: 
\begin{align*}
d(x_1, x_2) \leq e^{c_{\Phi}\tau} d(x_1 k a_{\tau v}, x_2 k a_{\tau v}) \qquad \text{for all } k \in K_G.
\end{align*}
\end{lemma}

\begin{proof}
Let $x_1 = \Gamma g_1$, $x_2 = \Gamma g_2$, $v$, $k$, and $\tau$ be as in the lemma. Since $d(x_1, x_2) = \min_{\gamma \in \Gamma} d(g_1, \gamma g_2)$, it suffices to show
\begin{align*}
d(g_1, g_2) \leq e^{c_{\Phi}\tau} d(g_1 k a_{\tau v}, g_2 k a_{\tau v}) \qquad \text{for all $k \in K_G$ and $\tau \geq 2$}.
\end{align*}
Since the metric is left $G$-invariant and right $K_G$-invariant, it suffices to show
\begin{align}\label{eqn:expansion}
d(e, g) \leq e^{c_{\Phi}\tau} d(e, a_{-\tau v} g a_{\tau v}) \qquad \text{for all $g \in G$ and $\tau \geq 2$}.
\end{align}
Let $g \in G$, $\tau \geq 2$, and $T = d(e, a_{-\tau v} g a_{\tau v})$. Let $\gamma:[0, T] \to G$ be a geodesic connecting $e$ and $a_{-\tau v} g a_{\tau v}$. Then, $\tilde{\gamma} = C_{a_{\tau v}} \circ \gamma$ is a smooth curve connecting $e$ and $g$, where $C_{a_{\tau v}}$ is the conjugation map by $a_{\tau v}$. Therefore, by left $G$-invariance of the metric, we have
\begin{align*}
d(e, g) \leq{}& \int_0^T \biggl\| \frac{d}{dt} \biggl|_{t = s} \tilde{\gamma}(t) \biggr\| \, ds = \int_0^T \biggl\| \frac{d}{dt}\biggl|_{t = s} {\tilde{\gamma}(s)}^{-1} \tilde{\gamma}(t) \biggr\| \, ds\\
={}& \int_0^T \biggl\| \frac{d}{dt}\biggl|_{t = s} C_{a_{\tau v}}\bigl( \gamma(s)^{-1} \gamma(t)\bigr) \biggr\| \, ds\\
\leq{}& \bigl\| \Ad_{a_{\tau v}} \bigr\|_{\mathrm{op}} d(e, a_{-\tau v} g a_{\tau v}).
\end{align*} 
Since $\bigl\| \Ad_{a_{\tau v}} \bigr\|_{\mathrm{op}} \leq e^{c_{\Phi}\tau}$ by \cref{lem:NormOfAdjointEstimate}, we have proven \cref{eqn:expansion} as desired.
\end{proof}

\begin{proof}[Proof of \cref{thm:K_EffectiveEquidistribution}]
Suppose \ShahEE holds. Recall the constants $\kappa_0$ and $\varrho_0$ from the hypothesis. Fix a sufficiently large $D \geq \max\{4, 2c_\Phi \varrho_0, 2c_\Phi + 1\}$ to be explicated later. Fix $\kappa = \frac{1}{2}\min\bigl\{\frac{1}{D}, \kappa_0\bigr\}$ and $\varrho = \bigl(1 - \frac{c_\Phi \varrho_0}{D}\bigr)^{-1}\varrho_0 \in (\varrho_0, 2\varrho_0)$ whose choices will be clear later. Let $C_\varsigma$, $x_0$, $R$, $v$, $t$, $\phi$, and $\varphi$ be as in the theorem. Let $\zeta = R^{-\frac{1}{D}} \in (0, \epsilon_G)$. For brevity, we write $Z = -\log(\zeta)$. It is clear from property~(4) in \cref{lem:Lie_transverse} that $s^+$ is a diffeomorphism on an open neighborhood $\mathcal{O}^\star \subset \LieK^\star$ of $0 \in \LieK^\star$. Define the open subsets
\begin{align*}
\tilde{B}_{\zeta}^{\LieK^\star}(0) &= (s^+|_{\mathcal{O}^\star})^{-1}\bigl(\Ad_{a_{Zv}} B^{\LieU^+}_{1}(0)\bigr) \subset \LieK^\star, \\
\tilde{B}_{\zeta}^{K} &= \exp\Bigl(\tilde{B}_{\zeta}^{\LieK^\star}(0)\Bigr) \cdot \exp\bigl(B^{\LieM}_{\zeta^{\varsigma}}(0)\bigr) \subset K.
\end{align*}
The latter is indeed an open subset due to the decomposition $\LieK = \LieK^\star \oplus \LieM$ from property~(2) in \cref{lem:Lie_transverse}. We have
\begin{align*}
&\int_K \phi(x_0 k a_{tv}) \varphi(k) \, d\mu_K(k) \\
={}&\frac{1}{\mu_K\bigl(\tilde{B}_{\zeta}^{K}\bigr)} \int_{\tilde{B}_{\zeta}^{K}} \int_K \phi(x_0 k_0 k a_{tv}) \varphi(k_0 k) \, d\mu_K(k_0) \, d\mu_K(k) \\
={}& \frac{1}{\mu_K\bigl(\tilde{B}_{\zeta}^{K}\bigr)} \int_{\tilde{B}_{\zeta}^{K}} \int_K \phi(x_0 k_0 k a_{tv}) \varphi(k_0) \, d\mu_K(k_0) \, d\mu_K(k) + O\bigl(\mathcal{S}^\ell(\phi)\mathcal{S}^\ell(\varphi)\zeta^{\varsigma}\bigr)\\
={}& \int_K \Biggl(\frac{1}{\mu_K\bigl(\tilde{B}_{\zeta}^{K}\bigr)} \int_{\tilde{B}_{\zeta}^{K}}  \phi(x_0 k_0 k a_{tv}) \, d\mu_K(k)\Biggr) \varphi(k_0) \, d\mu_K(k_0) + O\bigl(\mathcal{S}^\ell(\phi)\mathcal{S}^\ell(\varphi)\zeta^{\varsigma}\bigr).
\end{align*}

Now we study $\int_{\tilde{B}_{\zeta}^{K}}  \phi(x_0 k_0 k a_{tv}) \, d\mu_K(k)$. Using $\LieK = \LieK^\star \oplus \LieM$ and the fact that the Haar measure $\mu_K$ is induced by the (bi-invariant) Riemannian metric on $K$ (obtained by restricting the Riemannian metric on $G$), and shrinking $\mathcal{O}^\star$ if necessary, there exists a sufficiently small open neighborhood $\mathcal{O} = \mathcal{O}^\star \times \mathcal{O}' \subset \LieK^\star \times \LieM$ of $(0, 0) \in \LieK^\star \times \LieM$ and a positive smooth function $\varkappa \in C^\infty(\mathcal{O})$ bounded away from $0$ such that the pushforward of the measure $\varkappa(\omega^\star, \omega') \, d\omega^\star \, d\omega'$ on $\mathcal{O}$ under the smooth map $\LieK^\star \times \LieM \to K$ given by $(\omega^\star, \omega') \mapsto \exp(\omega^\star)\exp(\omega')$ gives the Haar measure $\mu_K$ restricted to the corresponding image open neighborhood of $e \in K$. We may assume that $D$ is sufficiently large so that $\tilde{B}_{\zeta}^{\LieK^\star}(0) \times B^{\LieM}_{\zeta^{\varsigma}}(0) \subset \mathcal{O}$. Using the definition of $\tilde{B}_{\zeta}^{K}$, the above formulation of the Haar measure $\mu_K$, and properties~(6) and (7) in \cref{lem:Lie_transverse}, we have
\begin{align*}
{}&\int_{\tilde{B}_{\zeta}^{K}}  \phi(x_0 k_0 k a_{tv}) \, d\mu_K(k) \\
={}& \int_{B^{\LieM}_{\zeta^{\varsigma}}(0)} \int_{\tilde{B}_{\zeta}^{\LieK^\star}(0)} \phi\bigl(x_0 k_0 u^{+}_{s^{+}(\omega^\star)} u^{-}_{s^{-}(\omega^\star)} a_{\tau(\omega^\star)} m_{\xi(\omega^\star)} \exp(\omega') a_{tv}\bigr) \cdot \varkappa(\omega^\star, \omega') \, d\omega^\star \, d\omega'\\
={}& \int_{B^{\LieM}_{\zeta^{\varsigma}}(0)} \int_{\tilde{B}_{\zeta}^{\LieK^\star}(0)} \phi\bigl(x_0 k_0 u^{+}_{s^{+}(\omega^\star)} u^{-}_{s^{-}(\omega^\star)} a_{tv} a_{\tau(\omega^\star)} m_{\xi(\omega^\star)} \exp(\omega')\bigr) \cdot \varkappa(\omega^\star, \omega') \, d\omega^\star \, d\omega'\\
={}& \int_{B^{\LieM}_{\zeta^{\varsigma}}(0)}  \int_{\tilde{B}_{\zeta}^{\LieK^\star}(0)} \phi\bigl(x_0 k_0 u^{+}_{s^{+}(\omega^\star)} u^{-}_{s^{-}(\omega^\star)} a_{tv}\bigr) \cdot \varkappa(\omega^\star, \omega') \, d\omega^\star \, d\omega' \\
&{}+ O\Bigl(\mu_{K}\bigl(\tilde{B}_{\zeta}^{K}\bigr)\mathcal{S}^\ell(\phi)\zeta^{\varsigma}\Bigr).
\end{align*}
Note that writing $\omega^\star = (s^{+}|_{\mathcal{O}^\star})^{-1}(w) \in \tilde{B}_{\zeta}^{\LieK^\star}(0)$, we have the following estimate by property~(5) in \cref{lem:Lie_transverse}: 
\begin{align*}
s^{-}(\omega^\star) = s^{-}((s^{+}|_{\mathcal{O}^\star})^{-1}(w)) = \theta(w) + O(\|w\|^2) = O(\zeta^{\varsigma}).
\end{align*}
We use similar techniques as in the proof of \cref{lem:LieTheoreticBoundsII} to get
\begin{align*}
\bigl\| \Ad_{a_{-tv}} s^{-}(\omega^\star) \bigr\| \leq e^{-\varsigma t} \| s^{-}(\omega^\star) \| = O(\zeta^{\varsigma})
\end{align*}
and continue the calculation of the above integral:
\begin{align}
\nonumber
{}&\int_{\tilde{B}_{\zeta}^{K}}  \phi(x_0 k_0 k a_{tv}) \, d\mu_K(k) \\
\nonumber
={}& \int_{B^{\LieM}_{\zeta^{\varsigma}}(0)}  \int_{\tilde{B}_{\zeta}^{\LieK^\star}(0)} \phi\Bigl(x_0 k_0 u^{+}_{s^{+}(\omega^\star)} a_{tv} u^{-}_{\Ad_{a_{-tv}} s^{-}(\omega^\star)}\Bigr) \cdot \varkappa(\omega^\star, \omega') \, d\omega^\star \, d\omega' \\
\nonumber
&{}+ O\Bigl(\mu_{K}\bigl(\tilde{B}_{\zeta}^{K}\bigr)\mathcal{S}^\ell(\phi)\zeta^{\varsigma}\Bigr) \\
={}& \int_{B^{\LieM}_{\zeta^{\varsigma}}(0)}  \int_{\tilde{B}_{\zeta}^{\LieK^\star}(0)} \phi\bigl(x_0 k_0 u^{+}_{s^{+}(\omega^\star)} a_{tv}\bigr) \cdot \varkappa(\omega^\star, \omega') \, d\omega^\star \, d\omega' + O\Bigl(\mu_{K}\bigl(\tilde{B}_{\zeta}^{K}\bigr)\mathcal{S}^\ell(\phi)\zeta^{\varsigma}\Bigr).\label{eqn:KtoU}
\end{align}

Later, we will focus on the main term in \cref{eqn:KtoU}. By abuse of notation, let us write $\omega^\star = (s^+|_{\mathcal{O}^\star})^{-1} \circ \Ad_{a_{Zv}}|_{B_1^{\LieU^+}(0)}$. Before we estimate the main term directly, we show the following estimates related to the Haar measure $\mu_K$:
\begin{enumerate}[label=(\roman*)]
\item $\displaystyle\bigl|\det \Ad_{a_{Zv}}|_{\LieU^+}\bigr| \mu_\LieM\bigl(B^{\LieM}_{\zeta^{\varsigma}}(0)\bigr) = O\bigl(\mu_{K}\bigl(\tilde{B}_{\zeta}^{K}\bigr)\bigr)$;
\item $\bigl|\det d((s^+|_{\mathcal{O}^\star})^{-1})_0\bigr| \cdot \varkappa(0, \omega') = \bigl|\det d((s^+|_{\mathcal{O}^\star})^{-1})_w\bigr| \cdot \varkappa(\omega^\star(w), \omega') \cdot (1 + O(\zeta^{\varsigma}))$ for all $(w, \omega') \in B^{\LieU^+}_{1}(0) \times B^{\LieM}_{\zeta^\varsigma}(0)$;
\item we have:
\begin{multline*}
\int_{B_{\zeta^{\varsigma}}^{\LieM}(0)} \bigl|\det d((s^{+}|_{\mathcal{O}^\star})^{-1})_0\bigr| \cdot \bigl|\det \Ad_{a_{Zv}}|_{\LieU^+}\bigr| \cdot \varkappa(0, \omega') \, d\omega' \\
= \frac{\mu_{K}\bigl(\tilde{B}_{\zeta}^{K}\bigr)}{\mu_{\LieU^+}\bigl(B^{\LieU^+}_{1}(0)\bigr)} \cdot (1 + O(\zeta^{\varsigma})).
\end{multline*}
\end{enumerate}
Firstly, by change of variables, we obtain
\begin{align}
\label{eqn:Haar on K}
\begin{aligned}
&\mu_K\bigl(\tilde{B}_{\zeta}^{K}\bigr) \\
={}&\int_{B_{\zeta^{\varsigma}}^{\LieM}(0)} \int_{\tilde{B}_{\zeta}^{\LieK^\star}(0)} \varkappa(\omega^\star, \omega') \, d\omega^\star \, d\omega'\\
={}&\int_{B_{\zeta^{\varsigma}}^{\LieM}(0)} \int_{B^{\LieU^+}_{1}(0)} \bigl|\det d((s^+|_{\mathcal{O}^\star})^{-1})_w \bigr| \cdot \bigl|\det \Ad_{a_{Zv}}|_{\LieU^+}\bigr| \cdot \varkappa(\omega^\star(w), \omega') \, dw \, d\omega'.
\end{aligned}
\end{align}
Property~(i) follows from the fact that
\begin{align*}
(w, \omega') \mapsto \bigl|\det d((s^+|_{\mathcal{O}^\star})^{-1})_w\bigr| \cdot \varkappa(\omega^\star(w), \omega')
\end{align*}
is a positive smooth function bounded away from $0$ and hence has a positive infimum on $B^{\LieU^+}_{1}(0) \times B^{\LieM}_{\zeta^\varsigma}(0)$, determined only by $G$. Similarly, property~(ii) follows from the fact that $\omega^\star(0) = 0$ and
\begin{align*}
(w, \omega') \mapsto \frac{\bigl|\det d((s^+|_{\mathcal{O}^\star})^{-1})_0\bigr| \cdot \varkappa(0, \omega')}{\bigl|\det d((s^+|_{\mathcal{O}^\star})^{-1})_w\bigr| \cdot \varkappa(\omega^\star(w), \omega')}
\end{align*}
is a positive smooth function bounded away from $0$ and $+\infty$ on $B^{\LieU^+}_{1}(0) \times B^{\LieM}_{\zeta^\varsigma}(0)$, determined only by $G$. Property~(iii) simply follows from property~(ii) and \cref{eqn:Haar on K}.

Now we estimate the main term in \cref{eqn:KtoU}. Using change of variables and properties~(i) and (ii), we have
\begin{align*}
{}& \int_{B^{\LieM}_{\zeta^{\varsigma}}(0)} \int_{\tilde{B}_{\zeta}^{\LieK^\star}(0)} \phi\bigl(x_0 k_0 u^{+}_{s^{+}(\omega^\star)} a_{tv}\bigr) \cdot \varkappa(\omega^\star, \omega') \, d\omega^\star \, d\omega' \\
={}& \int_{B^{\LieM}_{\zeta^{\varsigma}}(0)} \int_{B^{\LieU^+}_{1}(0)} \phi\Bigl(x_0 k_0 u^{+}_{\Ad_{a_{Zv}}(w)} a_{tv}\Bigr) \bigl|\det d((s^{+}|_{\mathcal{O}^\star})^{-1})_w\bigr| \cdot \bigl|\det \Ad_{a_{Zv}}|_{\LieU^+}\bigr| \\
&{}\cdot \varkappa(\omega^\star(w), \omega') \, dw \, d\omega' \\
={}& \int_{B^{\LieM}_{\zeta^{\varsigma}}(0)}  \int_{B^{\LieU^+}_{1}(0)} \phi\bigl(x_0 k_0 a_{Z v}u^{+}_{w} a_{(t - Z)v}\bigr) \bigl|\det d((s^{+}|_{\mathcal{O}^\star})^{-1})_0\bigr| \cdot \bigl|\det \Ad_{a_{Zv}}|_{\LieU^+}\bigr| \\
&{}\cdot \varkappa(0, \omega') \, dw \, d\omega' + O\Bigl(\mu_{K}\bigl(\tilde{B}_{\zeta}^{K}\bigr)\mathcal{S}^\ell(\phi)\zeta^{\varsigma}\Bigr)\\
={}&\int_{B^{\LieM}_{\zeta^{\varsigma}}(0)} \bigl|\det d((s^{+}|_{\mathcal{O}^\star})^{-1})_0\bigr| \cdot \bigl|\det \Ad_{a_{Zv}}|_{\LieU^+}\bigr| \cdot \varkappa(0, \omega') \, d\omega' \\
&{}\cdot \int_{B^{\LieU^+}_{1}(0)} \phi\bigl(x_0 k_0 a_{Z v}u^{+}_{w} a_{(t - Z)v}\bigr) \, dw + O\Bigl(\mu_{K}\bigl(\tilde{B}_{\zeta}^{K}\bigr)\mathcal{S}^\ell(\phi)\zeta^{\varsigma}\Bigr).
\end{align*}
Putting the above calculations together and using property~(iii), we have
\begin{multline*}
\frac{1}{\mu_K\bigl(\tilde{B}_{\zeta}^{K}\bigr)} \int_{B_{\zeta^\varsigma}^{\LieM}(0)}  \int_{\tilde{B}_{\zeta}^{\LieK^\star}(0)} \phi\bigl(x_0 k_0 u^{+}_{s^{+}(\omega^\star)} a_{tv}\bigr) \cdot \varkappa(\omega^\star, \omega') \, d\omega^\star \, d\omega'\\
= \frac{1}{\mu_{\LieU^+}\bigl(B^{\LieU^+}_{1}(0)\bigr)} \int_{B^{\LieU^+}_{1}(0)} \phi\bigl(x_0 k_0 a_{Z v}u^{+}_{w} a_{(t - Z)v}\bigr) \, dw + O\bigl(\mathcal{S}^\ell(\phi)\zeta^{\varsigma}\bigr).
\end{multline*}

Now, recall the constant $c_\varsigma$ from \ShahEE. Also recall $\zeta = R^{-\frac{1}{D}}$, and $C_\varsigma = c_\varsigma + c_{\Phi} + 1$ where the constant $c_\Phi$ is from \cref{eqn:Constant_cPhi}. Since $t \geq C_\varsigma \log(R)$, we first deduce
\begin{align*}
t' := t - Z = t - \frac{1}{D} \log(R) \geq{}& (C_\varsigma - 1) \log(R) \geq c_\varsigma \log(R).
\end{align*}
Now, note that the injectivity radius at any $x = \Gamma g \in \Gamma \backslash G$ can be written as
\begin{align*}
\inj_{\Gamma \backslash G}(x) &= \inf\{d(g, h): d(g, h) = d(\gamma g, h), \gamma \in \Gamma \setminus \{e\}, h \in G\} \\
&= \frac{1}{2}\inf_{\gamma \in \Gamma \setminus \{e\}} d(g, \gamma g).
\end{align*}
Hence, for $x_0' := x_0 k_0 a_{Zv}$, calculating as in \cref{lem:LieTheoreticBoundsII,lem:maxexpand} with the above formula, we derive
\begin{align*}
\inj_{\Gamma \backslash G}(x_0') = \inj_{\Gamma \backslash G}(x_0 k_0 a_{Zv}) \geq e^{-c_\Phi Z}\inj_{\Gamma \backslash G}(x_0) = \inj_{\Gamma \backslash G}(x_0)\zeta^{c_\Phi}.
\end{align*}
Since $R \gg \inj_{\Gamma \backslash G}(x_0)^{-\varrho}$, by our choice of $\varrho$ and the above inequality, we also deduce
\begin{align*}
R \gg \inj_{\Gamma \backslash G}(x_0)^{-\varrho_0}R^{\frac{c_\Phi \varrho_0}{D}} = \bigl(\inj_{\Gamma \backslash G}(x_0)\zeta^{c_\Phi}\bigr)^{-\varrho_0} \geq \inj_{\Gamma \backslash G}(x_0')^{-\varrho_0}.
\end{align*}
Now we apply \ShahEE to the point $x_0' = x_0 k_0 a_{Zv}$ with parameters $R \gg \inj_{\Gamma \backslash G}(x_0)^{-\varrho} \geq \inj_{\Gamma \backslash G}(x_0')^{-\varrho_0}$ and $t' = t - Z \geq c_\varsigma \log(R)$. We consider two separate cases.

\medskip
\noindent\textit{Case 1.} Suppose that for all $k_0 \in K$, property~(1) in \ShahEE holds, i.e.,
\begin{align*}
\frac{1}{\mu_{\LieU^+}\bigl(B^{\LieU^+}_{1}(0)\bigr)} \int_{B^{\LieU^+}_{1}(0)} \phi\bigl(x_0 k_0 a_{Z v}u^{+}_{w} a_{(t - Z)v}\bigr) \, dw = \int_{\Gamma \backslash G} \phi \, d\hat{\mu}_{\Gamma \backslash G} + O\bigl(\mathcal{S}^\ell(\phi)R^{-\varsigma \kappa_0}\bigr).
\end{align*}
Here, we have used $\mu_{\LieU^+}\bigl(B^{\LieU^+}_{1}(0)\bigr) = \mu_{U}\bigl(B^{U}_{1}(e)\bigr)$ where we recall $\mu_U = \exp_* \mu_{\LieU^+}$ from \cite[Chapter 1, \S 1.2, Theorem 1.2.10]{CG90}. Then, summarizing the above calculations, we have
\begin{align*}
{}&\int_{K} \phi(x_0 k a_{tv}) \varphi(k) \, d\mu_K(k)\\
={}&\int_K \Biggl(\frac{1}{\mu_K\bigl(\tilde{B}_{\zeta}^{K}\bigr)} \int_{\tilde{B}_{\zeta}^{K}} \phi(x_0 k_0 k a_{tv}) \, d\mu_K(k)\Biggr) \varphi(k_0) \, d\mu_K(k_0) + O\bigl(\mathcal{S}^\ell(\phi) \mathcal{S}^\ell(\varphi)\zeta^{\varsigma}\bigr)\\
={}&\int_K \Biggl(\frac{1}{\mu_{\LieU^+}\bigl(B^{\LieU^+}_{1}(0)\bigr)}\int_{B^{\LieU^+}_{1}(0)} \phi\bigl(x_0 k_0 a_{Z v} u_{w}^+ a_{(t - Z)v}\bigr) \, dw\Biggr) \varphi(k_0) \, d\mu_K(k_0) \\
&{}+ O\bigl(\mathcal{S}^\ell(\phi) \mathcal{S}^\ell(\varphi)\zeta^{\varsigma}\bigr)\\
={}&\int_{\Gamma \backslash G} \phi \, d\hat{\mu}_{\Gamma \backslash G} \cdot \int_K \varphi \, d\mu_K + O\bigl(\mathcal{S}^\ell(\phi) \mathcal{S}^\ell(\varphi)(\zeta^{\varsigma} + R^{-\varsigma \kappa_0})\bigr).
\end{align*}
By our choice of $\kappa$, we get $\zeta^{\varsigma} + R^{-\varsigma \kappa_0} \leq 2R^{-2\varsigma \kappa}$. We may now assume that the absolute implicit constant in the condition $R \gg \inj_{\Gamma \backslash G}(x_0)^{-\varrho}$ is sufficiently large so that $\frac{1}{2}R^{\varsigma \kappa}$ is larger than the absolute implicit constant in the above error term. Therefore, property~(1) of the theorem follows in this case.

\medskip
\noindent\textit{Case 2.} Now suppose that for some $k_0 \in K$, property~(2) in \ShahEE holds. Then, there exists $x \in \Gamma \backslash G$ with
\begin{align*}
d\bigl(x_0 k_0 a_{Zv}, x\bigr) \leq R^{c_\varsigma} (t')^{c_\varsigma} e^{-\varsigma t'}
\end{align*}
such that $xH$ is periodic with $\vol(xH) \leq R$. By \cref{lem:maxexpand}, we have
\begin{align*}
d\bigl(x_0, x(k_0 a_{Zv})^{-1}\bigr) \leq e^{c_{\Phi}Z} d\bigl(x_0 k_0 a_{Zv}, x\bigr) \leq e^{c_{\Phi}Z} R^{c_\varsigma} (t')^{c_\varsigma} e^{-\varsigma t'} \leq R^{C_\varsigma}t^{C_\varsigma}e^{-\varsigma t}
\end{align*}
where the last inequality is obtained using definitions, $\varsigma \leq c_{\Phi}$ and $D \geq 2c_{\Phi} + 1$. Therefore, property~(2) of the theorem holds in this case.
\end{proof}

\begin{remark}
\label{rem:DependenceOfInjectivityRadiusInProof}
Observe that when we invoke \ShahEE in Case 1 in the above proof, we need to verify the condition $R \gg_{\inj_{\Gamma \backslash G}(x_0k_0a_{Zv})} 1$ where the dependence on the injectivity radius must be exactly as in \ShahEE. This is delicate since $R$ then depends on $\zeta$ which in turn is taken to be $R^{-\frac{1}{D}}$. Nevertheless, it is possible to ensure the required condition precisely due to the \emph{explicit} dependence on the injectivity radius at hand in the condition $R \gg \inj_{\Gamma \backslash G}(x_0)^{-\varrho_0}$ in \ShahEE.
\end{remark}

\section{\texorpdfstring{Effective equidistribution of Riemannian skew balls from that of $K_H$-orbits}{Effective equidistribution of Riemannian skew balls from that of K\unichar{"005F}H-orbits}}
\label{sec:EquidistributionOfRiemannianSkewBalls}
In this section we prove \cref{pro:SkewballEffectiveEquidistribution} regarding effective equidistribution of Riemannian skew balls assuming effective equidistribution of $K_H$-orbits. Having established \cref{thm:K_EffectiveEquidistribution}, we can combine the two and derive \cref{thm:SkewballEffectiveEquidistributionDichotomy} which assumes \ShahEE and is stated in dichotomy form.

\begin{notation}
As in \cref{sec:VolumeCalculations}, we again \emph{drop the subscript $H$} for all the objects introduced in \cref{sec:Preliminaries} (except measures and ranks) for convenience.
\end{notation}

Fix $\varsigma_0 := \frac{1}{2}\min_{\alpha \in \Phi^+} \alpha(v_{2\rho}) > 0$ (see the beginning of \cref{sec:VolumeCalculations} for positivity) and recall $c_{\varsigma_0}$ from \ShahEE.

\begin{theorem}
\label{thm:SkewballEffectiveEquidistributionDichotomy}
Suppose \ShahEE holds. There exist $C_{\varsigma_0} \asymp c_{\varsigma_0}$, $\kappa \in (0, \kappa_0)$, and $\varrho \in (\varrho_0, 2\varrho_0)$ such that the following holds. For all $g_1, g_2 \in G$, there exists $M_{g_1, g_2} > 0$ such that for all $x_0 \in \Gamma \backslash G$, $R \gg \inj_{\Gamma \backslash G}(x_0)^{-\varrho}$, and $T \geq C_{\varsigma_0} \log(R) + M_{g_1, g_2}$, at least one of the following holds. 
\begin{enumerate}
\item For all $\phi \in C_{\mathrm{c}}^{\infty}(\Gamma \backslash G)$, we have
\begin{align*}
&\left| \frac{1}{\mu_{H}(H_{T}[g_1, g_2])} \int_{H_{T}[g_1, g_2]} \phi(x_0 h) \, d\mu_H(h) - \int_{\Gamma \backslash G} \phi \, d\hat{\mu}_{\Gamma \backslash G} \right| \\
\leq{}&
\begin{cases}
O\bigl(e^{O(d(o, g_1o))}\bigr) \mathcal{S}^{\ell}(\phi) R^{-\kappa}, & \rankG = 1 \\
O\bigl(e^{O(d(o, g_1o) + d(o, g_2o))}\bigr) \|\phi\|_{\infty} T^{-\frac{1}{2}} \log(T)^{\frac{1}{2}} + O\bigl(e^{O(d(o, g_1o))}\bigr) \mathcal{S}^{\ell}(\phi)R^{-\kappa}, & \rankG \geq 2.
\end{cases}
\end{align*}
\item There exists $x \in \Gamma \backslash G$ with
\begin{align*}
d(x_0, x) \leq R^{C_{\varsigma_0}} T^{C_{\varsigma_0}} e^{-\varsigma_0 T}
\end{align*}
such that $xH$ is periodic with $\vol(xH) \leq R$.
\end{enumerate}
Moreover, we can choose $M_{g_1, g_2} = C_1e^{C_2(d(o, g_1o) + d(o, g_2o))}$ for some absolute constants $C_1 > 0$ and $C_2 > 0$. 
\end{theorem}

\begin{remark}
In the $\rankG = 1$ case, if we replace the bound in property~(2) with $R^{C_{\varsigma_0}} T^{C_{\varsigma_0}} e^{-\frac{\varsigma_0}{2}T}$, then we can make the improvement $M_{g_1, g_2} = C(d(o, g_1o) + d(o, g_2o))$ for some $C > 0$. 
\end{remark}

\cref{thm:SkewballEffectiveEquidistributionDichotomy} follows from \cref{thm:K_EffectiveEquidistribution} and the following proposition. It says that once we have equidistribution of $K$-orbits in some Riemannian skew annulus, we have equidistribution of Riemannian skew balls.

\begin{proposition}
\label{pro:SkewballEffectiveEquidistribution}
For all $g_1, g_2 \in G$, there exists $M_{g_1, g_2} > 0$ such that the following holds:

Let $\phi \in C_{\mathrm{c}}^{\infty}(\Gamma \backslash G)$, $x_0 \in \Gamma \backslash G$, and $C > 0$. Suppose that there exist $\tilde{\kappa} > 0$, $R \gg 1$, $\tilde{C}_{\varsigma_0} > 0$, and $T \geq (C + \tilde{C}_{\varsigma_0}) \log(R) + M_{g_1, g_2}$ such that we have the following equidistribution of $K$-orbits:

For all $k_1, k_2 \in K$, and $v \in \interior(\LieA^+)$ satisfying:
\begin{enumerate}
\item $\min_{\alpha \in \Phi^+} \alpha(v/\|v\|) > \varsigma_0$,
\item $k_1\exp(v)k_2 \in H_{T + 1}[g_1, g_2] \backslash H_{T - C\log(R)}[g_1, g_2]$,
\end{enumerate}
and $\varphi \in C^{\infty}(K)$, we have
\begin{align*}
\left|\int_K \phi(x_0 k a_v k_2) \varphi(k) \, d\mu_K(k) - \int_{\Gamma \backslash G} \phi \, d\hat{\mu}_{\Gamma \backslash G} \cdot \int_K \varphi \, d\mu_K\right| \leq \mathcal{S}^\ell(\phi)\mathcal{S}^\ell(\varphi)R^{-\varsigma_0\tilde{\kappa}}.
\end{align*}

Then, there exist $\kappa_C \in (0, \tilde{\kappa})$ such that we also have
\begin{align*}
&\left|\frac{1}{\mu_H(H_T[g_1, g_2])}\int_{H_T[g_1, g_2]} \phi(x_0 h) \, d\mu_H(h) - \int_{\Gamma \backslash G} \phi \, d\hat{\mu}_{\Gamma \backslash G}\right| \\
\leq{}&
\begin{cases}
O\bigl(e^{O(d(o, g_1o))}\bigr) \mathcal{S}^{\ell}(\phi) R^{-\kappa_C}, & \rankG = 1 \\
O\bigl(e^{O(d(o, g_1o) + d(o, g_2o))}\bigr) \|\phi\|_{\infty} T^{-\frac{1}{2}} \log(T)^{\frac{1}{2}} + O\bigl(e^{O(d(o, g_1o))}\bigr) \mathcal{S}^{\ell}(\phi)R^{-\kappa_C}, & \rankG \geq 2.
\end{cases}
\end{align*}
Moreover, we can choose $M_{g_1, g_2} = C_1e^{C_2(d(o, g_1o) + d(o, g_2o))}$ for some absolute constants $C_1 > 0$ and $C_2 > 0$. 
\end{proposition}

\begin{remark}
The Riemannian skew annulus in \cref{pro:SkewballEffectiveEquidistribution} could have been chosen to be $H_{T+\delta}[g_1, g_2]\backslash H_{T - C\log(R) + \delta}[g_1, g_2]$ for some sufficiently small $\delta > 0$ depending on $R$. We need the radius to be slightly larger than $T$ for technical reasons. See \cref{eqn:decomposeskewball} in the proof. 
\end{remark}

\begin{proof}[Proof that \cref{pro:SkewballEffectiveEquidistribution} implies \cref{thm:SkewballEffectiveEquidistributionDichotomy}]
Suppose \ShahEE \linebreak holds. Let $g_1$, $g_2$, $x_0$, $T$, and $\phi$ be as in \cref{thm:SkewballEffectiveEquidistributionDichotomy}.

Fix $\tilde{\kappa} \in (0, \kappa_0)$, $\varrho \in (\varrho_0, 2\varrho_0)$, and $\tilde{C}_{\varsigma_0} := C := c_{\varsigma_0} + c_{\Phi} + 1$ to be the $\kappa$, $\varrho$, and $C_{\varsigma_0}$ provided by \cref{thm:K_EffectiveEquidistribution}. Let $R \gg \inj_{\Gamma \backslash G}(x_0)^{-\varrho}$ as in \cref{thm:K_EffectiveEquidistribution}. Fix $M_{g_1, g_2}$ to be the one from \cref{pro:SkewballEffectiveEquidistribution}. We will show that \cref{thm:SkewballEffectiveEquidistributionDichotomy} holds for $C_{\varsigma_0} := (2 + c_{\Phi})C + 1$. 

If the hypothesis in \cref{pro:SkewballEffectiveEquidistribution} holds, then the proposition gives $\kappa := \kappa_{C} \in (0, \tilde{\kappa})$ such that the following holds: 
\begin{multline*}
\left|\frac{1}{\mu_H(H_T[g_1, g_2])}\int_{H_T[g_1, g_2]} \phi(x_0 h) \, d\mu_H(h) - \int_{\Gamma \backslash G} \phi \, d\hat{\mu}_{\Gamma \backslash G}\right| \\
\leq 
\begin{cases}
O\bigl(e^{O(d(o, g_1o))}\bigr) \mathcal{S}^{\ell}(\phi) R^{-\kappa}, & \rankG = 1 \\
O\bigl(e^{O(d(o, g_1o) + d(o, g_2o))}\bigr) \|\phi\|_{\infty} T^{-\frac{1}{2}} \log(T)^{\frac{1}{2}} + O\bigl(e^{O(d(o, g_1o))}\bigr) \mathcal{S}^{\ell}(\phi)R^{-\kappa}, & \rankG \geq 2.
\end{cases}
\end{multline*} 
This is exactly property~(1) of the theorem.

Now, suppose the hypothesis in \cref{pro:SkewballEffectiveEquidistribution} does not hold. Then, there exists $k_1, k_2 \in K$ and $v \in \LieA_{T + 1}^+[g_1 k_1, k_2 g_2] \backslash \LieA_{T'}^+[g_1 k_1, k_2 g_2]$ with $\min_{\alpha \in \Phi^+} \alpha(v/\|v\|) > \varsigma_0$, where we write $T' := T - C\log(R)$, such that the equidistribution of $K$-orbits of the hypothesis fails. We will show that in this case property~(2) of the theorem holds. Let $t = \|v\| = d(o, a_v o)$. We have
\begin{align*}
T' - d(o, g_1o) - d(o, g_2o) \leq t \leq T + 1 + d(o, g_1o) + d(o, g_2o)
\end{align*}
due to the following calculations. Using the triangle inequality and left $G$-invariance of the metric, since $v \in \LieA_{T + 1}^+[g_1 k_1, k_2 g_2]$, we have
\begin{align*}
t ={}& d(o, \exp(v)o)\\
\leq{}& d(o, k_1^{-1}g_1^{-1}o) + d(k_1^{-1}g_1^{-1}o, \exp(v) k_2g_2o) + d(\exp(v) k_2 g_2o, \exp(v)o)\\
\leq{}& T + 1 + d(o, g_1o) + d(o, g_2o).
\end{align*}
Similarly, since $v \notin \LieA_{T'}^+[g_1 k_1, k_2 g_2]$, we have
\begin{align*}
t ={}& d(o, \exp(v)o)\\
\geq{}& d(o, g_1 k_1 \exp(v) k_2 g_2o) - d(\exp(v)k_2 g_2o, \exp(v)o)- d(o, k_1^{-1}g_1^{-1}o)\\
\geq{}& T' - d(o, g_1o) - d(o, g_2o).
\end{align*}
Now, since we may assume $M_{g_1, g_2} > d(o, g_1o) + d(o, g_2o)$, we have $t \geq C\log(R)$. By the decreasing property of the family $\{C_\varsigma\}_{\varsigma > 0}$ in \cref{thm:K_EffectiveEquidistribution}, we conclude that property~(1) in \cref{thm:K_EffectiveEquidistribution} does not hold for $\phi_{k_2} = \phi \circ m_{k_2}^{\mathrm{R}}$ where $m_{k_2}^{\mathrm{R}}$ is the right multiplication map by $k_2$. Here, we are also using the fact that $\mathcal{S}^{\ell}(\phi_{k_2}) = \mathcal{S}^{\ell}(\phi)$ by right $K$-invariance of the Riemannian metric on $G$. Thus, property~(2) in \cref{thm:K_EffectiveEquidistribution} holds, i.e., there exists $x \in \Gamma \backslash G$ with
\begin{align*}
d(x_0, x) \leq R^{C} t^{C} e^{-\varsigma_0t}
\end{align*}
such that $xH$ is periodic with $\vol(xH) \leq R$. We calculate that
\begin{align*}
R^C t^C e^{-\varsigma_0t} &\leq R^C t^C e^{-\varsigma_0T'} e^{d(o, g_1o) + d(o, g_2o)} = R^{(1 + c_{\Phi})C} t^{C} e^{-\varsigma_0 T} e^{d(o, g_1o) + d(o, g_2o)}\\
&\leq R^{(2 + c_{\Phi})C + 1} T^{(2 + c_{\Phi})C + 1} e^{-\varsigma_0 T} = R^{C_{\varsigma_0}} T^{C_{\varsigma_0}} e^{-\varsigma_0 T}
\end{align*}
where the last inequality follows from the fact that $T > e^{d(o, g_1o) + d(o, g_2o)}$. Therefore, property~(2) of the theorem holds. 
\end{proof}

\subsection{Outline of the proof and preparatory lemmas}
Due to the complicated calculations involving asymptotic formulas for the volume of Riemannian skew balls, we give a brief outline of the proof before starting it in earnest. We introduce the following notation corresponding to the $\LieA$-direction in the Riemannnian skew annulus $H_{T + \delta}[g_1, g_2] \backslash H_{T'}[g_1, g_2]$: 
For all $g_1, g_2 \in G$ and $T_2 > T_1 > 0$, we write
\begin{align*}
\LieA^+_{T_1, T_2}[g_1, g_2] := \LieA^+_{T_2}[g_1, g_2] \setminus \overline{\LieA^+_{T_1}[g_1, g_2]}
\end{align*}
throughout \cref{sec:EquidistributionOfRiemannianSkewBalls}. The proof will be divided into the following 4 steps:

\begin{enumerate}[label=Step \arabic*.]
\item We first apply the integral formula from \cref{eqn:IntegralFormulaForH}. In order to use the hypothesis in \cref{pro:SkewballEffectiveEquidistribution} we need to restrict the integral over the Weyl chamber $\LieA^+$ to an open convex cone $\mathcal{C}_{v_{2\rho}, \tau}^{\LieA}$ strictly contained in $\LieA^+$. \Cref{eqn:estimate_out_cone,eqn:estimate_otherroot} are heavily used in this step. The error term created in this step is denoted by $E_1$ in the proof. This step is not needed in the case $\rankH = 1$.
\item We decompose $K$ into small pieces so that one can approximate the Riemannian skew balls via Riemannian balls in a small sector. The main tool is \cref{lem:Sobolev}. The error term created in this step is denoted by $E_2$ in the proof.
\item In the process of proving \cref{thm:SkewBallVolumeAsymptotic}, we found that the volume of Riemannian skew balls are concentrated near its boundary. Therefore, it suffices to focus on the Riemannian skew annulus $H_{T + \delta}[g_1, g_2] \backslash H_{T'}[g_1, g_2]$ with a suitable choice of $T'$ and sufficiently small $\delta$. We remark here that $T'$ and $\delta$ will depends only on $R$. The error term created in this step is denoted by $E_3$ in the proof.
\item Focusing on the Riemannian skew annulus, we use the Cartan decomposition to write elements as $h = k a_v k_2 \in H_{T + \delta}[g_1, g_2] \backslash H_{T'}[g_1, g_2]$ and use the hypothesis in \cref{pro:SkewballEffectiveEquidistribution} regarding equidistribution of $K$-orbits. This produces our final error term $E_4$. Thus, combining this with the estimates from first three steps, we produce the integral of $\phi$ over $\Gamma \backslash G$ up to the total error term $\sum_{j = 1}^4 E_j$ where all the terms are shown to be of the desired form.
\end{enumerate}

\begin{figure}
\centering
\includegraphics{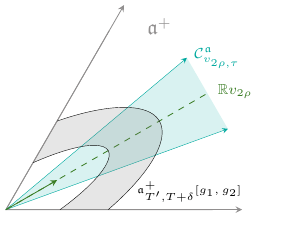}
\caption{This depicts the main part contributing to the Riemannian skew ball equidistribution in the $\LieA$-direction. The green part corresponds to the open cone in Step~1 and the gray part corresponds to the the Riemannian skew annulus in Step~3. The $T'$ and $\tau$ here will be explicated later in the proof. }
\label{fig:SkewBallEquidistribution}
\end{figure}

Now we prove a lemma regarding partition of unity on $K$ which is used in Step~2. Its proof uses similar techniques as in \cite[Lemma 2.3]{GS14}.

\begin{lemma}\label{lem:Sobolev}
Let $\delta > 0$, $C > 1$, and $\mathcal{O} \subset K$ be an open neighborhood of $e \in K$ such that
\begin{align*}
B_{\delta}^K(e) \subset \mathcal{O} \subset B_{C\delta}^K(e).
\end{align*}
Then, there exists a finite subset $\{k_j\}_{j = 1}^{N} \subset K$ for some $N \in \N$ such that: 
\begin{enumerate}
\item $\{\mathcal{O}k_j\}_{j = 1}^{N}$ is an open cover for $K$; 
\item there exists $m_K > 0$ depending only on $K$ such that the multiplicity of the open cover $\{\mathcal{O}k_j\}_{j = 1}^{N}$ is at most $C^{\dim(K)} m_K$;
\item there exists a partition of unity $\{\varphi_j\}_{j = 1}^{N} \subset C^\infty(K)$ subordinate to the open cover $\{\mathcal{O}k_j\}_{j = 1}^{N}$ such that $\mathcal{S}^\ell(\varphi_j) \ll \delta^{-(\ell + \frac{\dim(K)}{2})}$.
\end{enumerate}
\end{lemma}

\begin{proof}
Let $\delta$, $C$, and $\mathcal{O}$ be as in the lemma. Let $\{k_j\}_{j = 1}^{N} \subset K$ for some $N \in \N$ be a finite maximal $\frac{\delta}{2}$-separated set which exists since $K$ is compact. Then, by right $K$-invariance of the metric, $K$ $\bigl\{B_{\delta/2}^{K}(e)k_j\bigr\}_{j = 1}^{N}$ is an open cover for $K$. Take a partition of unity $\{\tilde{\varphi_j}\}_{j = 1}^{N} \subset C^\infty(K)$ subordinate to this open cover. There exists $\psi_{\delta} \in C^\infty(K)$ such that
\begin{align*}
\supp(\psi_{\delta}) &\subset B_{\delta/2}^K(e), & \int_K \psi_{\delta} \, d\mu_K &= 1, & \mathcal{S}^{\ell}(\psi_{\delta}) &\ll \delta^{-(\ell + \frac{\dim(K)}{2})}.
\end{align*}
Let $\varphi_j = \psi_{\delta} * \tilde{\varphi_j}$ for all $1 \leq j \leq N$. Then we have the properties:
\begin{enumerate}
\item $\sum_{j = 1}^N \varphi_j = \sum_{j = 1}^N \psi_{\delta} * \tilde{\varphi_j}  = \psi_{\delta} * \mathds{1} = \mathds{1}$;
\item $\supp(\varphi_j) \subset \supp(\psi_{\delta}) \cdot \supp(\tilde{\varphi_j}) \subset B_{\delta/2}^{K}(e) \cdot B_{\delta/2}^{K}(e) k_j \subset B_{\delta}^{K}(e) k_j \subset \mathcal{O} k_j$ for all $1 \leq j \leq N$.
\end{enumerate}
Moreover, using Young's inequality for convolutions, we have
\begin{align*}
\mathcal{S}^\ell(\varphi_j)^2 ={}& \sum_{\mathcal{D} \in \mathcal{D}(K), \deg \mathcal{D} \leq \ell} \|\mathcal{D}(\psi_{\delta} * \tilde{\varphi_j})\|_2^{2}
= \sum_{\mathcal{D} \in \mathcal{D}(K), \deg \mathcal{D} \leq \ell} \|\mathcal{D}(\psi_{\delta}) * \tilde{\varphi_j} \|_2^{2}\\
\leq{}& \sum_{\mathcal{D} \in \mathcal{D}(K), \deg \mathcal{D} \leq \ell} \|\tilde{\varphi_j}\|_1^2 \cdot \|\mathcal{D}\psi_{\delta}\|_2^{2} = \|\varphi_j\|_1^2 \mathcal{S}^\ell(\psi_{\delta})^2 \ll \delta^{-(2\ell + \dim(K))}
\end{align*}
where $\mathcal{D}(K) \subset U(\LieK)$ is the subset of monomials in a fixed orthonormal basis $\{\mathcal{D}_j\}_{j = 1}^{\dim(K)} \in \LieK$.

It remains to check the multiplicity condition for the open cover $\bigl\{B_{C\delta}^{K}(e) k_j\bigr\}_{j = 1}^{N}$. Note that for all $1 \leq j, j', j'' \leq N$, we have that $B_{C\delta}^{K}(e)k_j \cap B_{C\delta}^{K}(e)k_{j'} \neq \varnothing$ implies $k_{j'}k_j^{-1} \in \bigl(B_{C\delta}^{K}(e)\bigr)^{-1} B_{C\delta}^{K}(e) \subset B_{2C\delta}^{K}(e)$ and again by right $K$-invariance of the metric, we also have $d\bigl(k_{j'}k_j^{-1}, k_{j''}k_j^{-1}\bigr) = d(k_{j'}, k_{j''})$. Therefore, the multiplicity is the cardinality of a $\frac{\delta}{2}$-separated set in $B_{2C\delta}^{K}$ which is at most $C^{\dim(K)}m_K$ for some constant $m_K$ depending only on $K$. 
\end{proof}

\subsection{Proof of effective equidistribution of Riemannian skew balls}
Now, we prove \cref{pro:SkewballEffectiveEquidistribution} following the outline given in the previous subsection.

\begin{proof}[Proof of \cref{pro:SkewballEffectiveEquidistribution}]
We first introduce some notations and conventions. Let $g_1$, $g_2$, $\phi$, $x_0$, $C$, $R$, $\tilde{\kappa}$, $\tilde{C}_{\varsigma_0}$, and $T$ be as in the proposition and suppose that the hypothesis of the proposition holds. Set $T' := T - C\log(R)$. Let $\tilde{E}_{g_1, g_2} = \frac{E[g_1, g_2]}{C[g_1, g_2]}$ where $E[g_1, g_2]$ and $C[g_1, g_2]$ come from \cref{thm:SkewBallVolumeAsymptotic}. Let $D_{g_1, g_2} = d(o, g_1o) + d(o, g_2o)$. Recall $\varpi[g_1, g_2] = \beta_{e^+}(g_1^{-1}o, o) + \involution(\beta_{e^-}(g_2o, o))$. By \cref{thm:SkewBallVolumeAsymptotic}, we have the following estimates: 
\begin{enumerate}
\item $\frac{1}{C[g_1, g_2]} = O\bigl(e^{O(d(o, g_1o) + d(o, g_2o))}\bigr) = O\bigl(e^{O(D_{g_1, g_2})}\bigr)$;
\item $\tilde{E}_{g_1, g_2} = O\bigl(e^{O(d(o, g_1o) + d(o, g_2o))}\bigr) = O\bigl(e^{O(D_{g_1, g_2})}\bigr)$;
\item $\|\varpi[g_1, g_2]\| = O(D_{g_1, g_2})$.
\end{enumerate}
We will also use the same convention as in the proof of \cref{thm:SkewBallVolumeAsymptotic}. The implicit constant in $O_{g_1, g_2}$ \emph{is always} $O\bigl(e^{O(d(o, g_1o) + d(o, g_2o))}\bigr)$. 

Now we prove the proposition following the outline we gave above. By \cref{lem:NormOfAdjointEstimate}, it suffices to show that there exists $\kappa_C > 0$ and $\hat{\ell} > 0$ such that for $T$ and $R$ as in the statement of the proposition, the error term is of the following form:
\begin{align}
\label{eqn:FullErrorTermWithAdjoint}
\begin{cases}
O\bigl(\|\Ad_{g_1}\|_{\mathrm{op}}^{\hat{\ell}}\bigr) \mathcal{S}^{\ell}(\phi) R^{-\kappa_C}, & \rankG = 1 \\
O\bigl(e^{O(d(o, g_1o) + d(o, g_2o))}\bigr) \|\phi\|_{\infty} T^{-\frac{1}{2}} \log(T)^{\frac{1}{2}} + O\bigl(\|\Ad_{g_1}\|_{\mathrm{op}}^{\hat{\ell}}\bigr) \mathcal{S}^{\ell}(\phi)R^{-\kappa_C}, & \rankG \geq 2.
\end{cases}
\end{align}

\medskip
\noindent
\textit{Step 1. Restricting to the cone $\mathcal{C}_{v_{2\rho}, \tau}^{\LieA}$ in the case $\rankH \geq 2$.}

We use the estimates from \cref{sec:VolumeCalculations} to reduce the calculation into a cone $\mathcal{C}_{v_{2\rho}, \tau}^{\LieA}$ in the $\rankH \geq 2$ case. \Cref{eqn:estimate_out_cone,eqn:estimate_otherroot} are heavily used. For the $\rankH = 1$ case, the reader may skip this part and go directly to Step 2.

Let us fix the parameter $\tau$. Recall $\varsigma_0 = \frac{1}{2}\min_{\alpha \in \Phi^+} \alpha(v_{2\rho}) > 0$. Note that for all $v \in \mathcal{C}_{v_{2\rho}, \tau}^{\LieA}$ with $\|v\| = 1$, we have: 
\begin{align*}
\min_{\alpha \in \Phi^+} \alpha(v) \geq \min_{\alpha \in \Phi^+} \alpha(v - v_{2\rho}) + \min_{\alpha \in \Phi^+} \alpha(v_{2\rho}) \geq \min_{\alpha \in \Phi^+} \alpha(v_{2\rho}) - \max_{\alpha \in \Phi^+} \|\alpha\| \tau.
\end{align*}
We fix $\tau = \frac{\varsigma_0}{\max_{\alpha \in \Phi^+} \|\alpha\|}$ once and for all in this proof. Therefore, $\min_{\alpha \in \Phi^+} \alpha(v) \geq \varsigma_0$ for all $v \in \mathcal{C}_{v_{2\rho}, \tau}^{\LieA}$ with $\|v\| = 1$. 

Now we split the integral with respect to the cone $\mathcal{C}_{v_{2\rho}, \tau}^{\LieA}$: 
\begin{align*}
{}&\mu_M(M) \int_{H_{T}[g_1, g_2]} \phi(x_0 h) \, d\mu_H(h)\\
={}& \int_K \int_K \int_{\LieA_{T}[g_1 k_1, k_2 g_2]} \phi(x_0 k_1 a_v k_2) \xi(v) \, dv \, d\mu_K(k_1) \, d\mu_K(k_2)\\
={}& \int_K \int_K \int_{\LieA_{T}[g_1 k_1, k_2 g_2] \cap \mathcal{C}^{\LieA}_{v_{2\rho}, \tau}} \phi(x_0 k_1 a_v k_2) \xi(v) \, dv \, d\mu_K(k_1) \, d\mu_K(k_2)\\ 
{}&+ \int_K \int_K \int_{\LieA_{T}[g_1 k_1, k_2 g_2] \backslash \mathcal{C}^{\LieA}_{v_{2\rho}, \tau}} \phi(x_0 k_1 a_v k_2) \xi(v) \, dv \, d\mu_K(k_1) \, d\mu_K(k_2). 
\end{align*}
Using \cref{eqn:xiFormula}, we could decompose $\xi$ into two parts: 
\begin{align*}
\xi(v) = \frac{e^{2\rho(v)}}{2^{m_{\Phi^+}}} + \sum_{j = 1}^N c_je^{\lambda_j(v)}
\end{align*}
for some $N \in \N$, $\{c_j\}_{j = 1}^N \subset \R$, and $\{\lambda_j\}_{j = 1}^N \subset \LieA^*$. By \cref{eqn:lambdaLessThanTworho}, there exists $\eta \in (0, \delta_{2\rho})$ such that
\begin{align*}
\max_{j \in \{1, 2, \dotsc, N\}}\max_{v \in \overline{\LieA_1^+}} \lambda_j(v) = \delta_{2\rho} - \eta.
\end{align*}
Using \cref{eqn:estimate_out_cone} and \cref{eqn:estimate_otherroot}, we have estimate of this integral at outside of the cone: 
\begin{align*}
{}&\biggl| \int_K \int_K \int_{\LieA_{T}[g_1 k_1, k_2 g_2] \backslash \mathcal{C}^{\LieA}_{v_{2\rho}, \tau}} \phi(x_0 k_1 a_v k_2) \xi(v) \, dv \, d\mu_K(k_1) \, d\mu_K(k_2) \biggr|\\
\leq{}& \|\phi\|_{\infty} \bigl[ \beta_{\rankH}(T + O_{g_1, g_2}(1) + O_{g_1, g_2}(T^{-1}))^{\rankH}e^{\delta_{2\rho} (1 - \frac{\tau^2}{2}) (T + O_{g_1, g_2}(1) + O_{g_1, g_2}(T^{-1}))}\\ 
{}&+ \beta_\rankH e^{O_{g_1, g_2}(\delta_{2\rho} - \eta)} \bigl(T + O_{g_1, g_2}(1) + O_{g_1, g_2}(T^{-1})\bigr)^\rankH e^{(\delta_{2\rho} - \eta)T} e^{O_{g_1, g_2}(T^{-1})} \bigr].
\end{align*}
Denote by $E_1$ the term on the right hand side without the factor $\|\phi\|_{\infty}$. 

Dividing by $\mu_H(H_T[g_1, g_2])$ and using \cref{thm:SkewBallVolumeAsymptotic} (note that $\rankG \geq 2$), we have the following estimate: 
\begin{align*}
&\frac{E_1}{\mu_H(H_T[g_1, g_2])} \\
\leq{}& \frac{\beta_{\rankH}(T + O_{g_1, g_2}(1) + O_{g_1, g_2}(T^{-1}))^{\rankH}e^{\delta_{2\rho} (1 - \frac{\tau^2}{2}) (T + O_{g_1, g_2}(1) + O_{g_1, g_2}(T^{-1}))}}{C[g_1, g_2] T^{\frac{\rankH - 1}{2}} e^{\delta_{2\rho}T} -E[g_1, g_2]\bigl(\log(T)^{\frac{1}{2}}T^{\frac{\rankH - 2}{2}}e^{\delta_{2\rho}T}\bigr)}\\ 
{}&+ \frac{\beta_\rankH e^{O_{g_1, g_2}(\delta_{2\rho} - \eta)} \bigl(T + O_{g_1, g_2}(1) + O_{g_1, g_2}(T^{-1})\bigr)^\rankH e^{(\delta_{2\rho} - \eta)T} e^{O_{g_1, g_2}(T^{-1})}}{C[g_1, g_2] T^{\frac{\rankH - 1}{2}} e^{\delta_{2\rho}T} -E[g_1, g_2]\bigl(\log(T)^{\frac{1}{2}}T^{\frac{\rankH - 2}{2}}e^{\delta_{2\rho}T}\bigr)} \\
\leq{}& e^{- \eta'T} \leq R^{-\kappa_1}
\end{align*}
for all $T - C\log(R) \geq \Omega\bigl(e^{\Omega(D_{g_1, g_2})}\bigr)$, $\eta' = \min\bigl\{\frac{\delta_{2\rho}\tau^2}{4}, \frac{\eta}{2}\bigr\}$, and $\kappa_1 := C\eta'$. 

Therefore, for all $T - C\log(R) \geq \Omega\bigl(e^{\Omega(D_{g_1, g_2})}\bigr)$, we have: 
\begin{align}\label{eqn:equi_estimate_higherrankH}
\begin{aligned}
{}&\frac{\mu_M(M)}{\mu_H(H_T[g_1, g_2])}\int_{H_{T}[g_1, g_2]} \phi(x_0 h) \, d\mu_H(h)\\
={}& \frac{1}{\mu_H(H_T[g_1, g_2])}\int_K \int_K \int_{\LieA_{T}[g_1 k_1, k_2 g_2] \cap \mathcal{C}^{\LieA}_{v_{2\rho}, \tau}} \phi(x_0 k_1 a_v k_2) \xi(v) \, dv \, d\mu_K(k_1) \, d\mu_K(k_2)\\ 
{}&+ O(\mathcal{S}^\ell(\phi)R^{-\kappa_1}).
\end{aligned}
\end{align}

\medskip
\noindent
\textit{Step 2. Locally approximating the Riemannian skew ball via Riemannian balls.}

Now we focus on the integral inside the cone. If $\rankH = 1$, the cone is just $\interior(\LieA^+)$. Let $D > 0$ be a sufficiently large constant which will be fixed later. Let $R \gg_D 1$ so that $\delta := R^{-1/D} \in (0, \epsilon_G)$. Take the open symmetric neighborhood
\begin{align*}
\mathcal{O} = \{k \in K: d(e, g_1 k g_1^{-1}) < \delta\} \subset K
\end{align*}
of $e \in K$. Calculating as in \cref{eqn:ConjugationDifferential}, we have
\begin{align*}
B_{\frac{\delta}{\|\Ad_{g_1}\|_{\mathrm{op}}}}^{K}(e) \subset \mathcal{O} \subset B_{\|(\Ad_{g_1})^{-1}\|_{\mathrm{op}}\delta}^K(e).
\end{align*}
For all $\tilde{k} \in \mathcal{O}$, using the triangle inequality and left $G$-invariance of the metric, we have 
\begin{align*}
d(o, g_1 \tilde{k} k_1 \exp(v) k_2 g_2o ) &< d(o, g_1 k_1 \exp(v) k_2 g_2o) + d(o, g_1 \tilde{k} g_1^{-1}o) \\
&< d(o, g_1 k_1 \exp(v) k_2 g_2o) + \delta.
\end{align*}
Therefore, for all $\tilde{k} \in \mathcal{O}$, we have the following containments:
\begin{align}\label{eqn:SandwichInSkewBall}
\LieA_{T}^{+}[g_1 k_1, k_2 g_2] \subset \LieA_{T + \delta}^{+}[g_1 \tilde{k} k_1, k_2 g_2] \subset \LieA_{T + 2\delta}^{+}[g_1 k_1, k_2 g_2].
\end{align}
Using \cref{lem:Sobolev}, we have a finite open cover $\{\mathcal{O}\tilde{k}_j \}_{j = 1}^{N}$ for $K$, for some $\{\tilde{k}_j\}_{j = 1}^N \subset K$ and $N \in \N$, whose multiplicity is $O\bigl((\|\Ad_{g_1}\|_{\mathrm{op}} \cdot \|(\Ad_{g_1})^{-1}\|_{\mathrm{op}})^{\dim(K)}\bigr)$ and a partition of unity $\{\varphi_j\}_{j = 1}^{N} \subset C^\infty(K)$ subordinate to the open cover such that:
\begin{align*}
\mathcal{S}^\ell(\varphi_j) &\ll \left(\frac{\delta}{\|\Ad_{g_1}\|_{\mathrm{op}}}\right)^{-\ell'}, & \ell' := \ell + \frac{\dim(K)}{2}.
\end{align*}
We have $\|(\Ad_{g_1})^{-1}\|_{\mathrm{op}} \leq \|\Ad_{g_1}\|_{\mathrm{op}}^{\dim(G) - 1}$ using the relation to singular values in \cref{eqn:DistanceFormulaByOperatorNorm} and $\det \Ad_{g_1} = 1$ as $G$ is semisimple. Therefore, the multiplicity of the open cover is $O\bigl(\|\Ad_{g_1}\|_{\mathrm{op}}^{\dim(K)\dim(G)}\bigr)$.

Using \cref{eqn:SandwichInSkewBall}, we can locally approximate the Riemannian skew ball via Riemannian ball:
\begin{align}\label{eqn:equiskewH_in_cone}
\begin{aligned}
{}&\int_K \int_K \int_{\LieA_{T}[g_1 k_1, k_2 g_2] \cap \mathcal{C}^{\LieA}_{v_{2\rho}, \tau}} \phi(x_0 k_1 a_v k_2) \xi(v) \, dv \, d\mu_K(k_1) \, d\mu_K(k_2)\\
={}& \int_K \sum_{j = 1}^{N} \int_K \varphi_j(k_1) \int_{\LieA_{T}[g_1 k_1, k_2 g_2] \cap \mathcal{C}^{\LieA}_{v_{2\rho}, \tau}} \phi(x_0 k_1 a_v k_2) \xi(v) \, dv \, d\mu_K(k_1) \, d\mu_K(k_2)\\
={}& \int_K \sum_{j = 1}^{N} \int_K \varphi_j(k_1) \int_{\LieA_{T + \delta}[g_1 \tilde{k}_j, k_2 g_2] \cap \mathcal{C}^{\LieA}_{v_{2\rho}, \tau}} \phi(x_0 k_1 a_v k_2) \xi(v) \, dv \, d\mu_K(k_1) \, d\mu_K(k_2)\\
{}&+ \|\phi\|_{\infty} O(\mu_H(H_{T + 2\delta}[g_1, g_2]) - \mu_H(H_T[g_1, g_2])).
\end{aligned}
\end{align}
Denote by $E_2$ the error term without the factor $\|\phi\|_{\infty}$. Dividing by $\mu_H(H_{T}[g_1, g_2])$ and using \cref{thm:SkewBallVolumeAsymptotic}, we have the following two cases according to $\rankG$. 

If $\rankG = 1$, we have: 
\begin{align*}
\frac{E_2}{\mu_H(H_T[g_1, g_2])} \leq{}& \frac{C[g_1, g_2] e^{\delta_{2\rho} (T + 2\delta)} + E[g_1, g_2]\bigl(e^{(\delta_{2\rho} - \eta_1)(T + 2\delta)}\bigr)}{C[g_1, g_2] e^{\delta_{2\rho} T} - E[g_1, g_2]\bigl(e^{(\delta_{2\rho} - \eta_1)T}\bigr)} - 1\\
\leq{}&  e^{2\delta_{2\rho}\delta} \frac{1 + \tilde{E}_{g_1, g_2}e^{-\eta_1 (T + 2\delta)}}{1 - \tilde{E}_{g_1, g_2}e^{-\eta_1 T}} - 1\\
\leq{}& e^{2\delta_{2\rho}\delta} - 1 + e^{2\delta_{2\rho}\delta} \frac{2\tilde{E}_{g_1, g_2}e^{-\eta_1 T}}{1 - \tilde{E}_{g_1, g_2}e^{-\eta_1 T}}\\
\ll{}& R^{-\kappa_2}
\end{align*}
where $\kappa_2 = \min\bigl\{\frac{C\eta_1}{2}, 1/D\bigr\}$ and for all $T - C\log(R) \gg \frac{2}{\eta_1} \log(4 \tilde{E}_{g_1, g_2})$. 

If $\rankG \geq 2$, we have: 
\begin{align*}
\frac{E_2}{\mu_H(H_T[g_1, g_2])} \leq{}& \frac{C[g_1, g_2] (T + 2\delta)^{\frac{\rankH - 1}{2}} e^{\delta_{2\rho}(T + 2\delta)}}{C[g_1, g_2] T^{\frac{\rankH - 1}{2}} e^{\delta_{2\rho}T} - E[g_1, g_2]\bigl(\log(T)^{\frac{1}{2}}T^{\frac{\rankH - 2}{2}}e^{\delta_{2\rho}T}\bigr)}\\
{}&+ \frac{E[g_1, g_2]\bigl(\log(T + 2\delta)^{\frac{1}{2}}(T + 2\delta)^{\frac{\rankH - 2}{2}}e^{\delta_{2\rho}(T + 2\delta)}\bigr)}{C[g_1, g_2] T^{\frac{\rankH - 1}{2}} e^{\delta_{2\rho}T} - E[g_1, g_2]\bigl(\log(T)^{\frac{1}{2}}T^{\frac{\rankH - 2}{2}}e^{\delta_{2\rho}T}\bigr)} - 1 \\
\leq{}& (1 + \frac{2\delta}{T})^{\frac{\rankH - 1}{2}} e^{2\delta_{2\rho}\delta} \frac{1 + \tilde{E}_{g_1, g_2}\bigl( \log(T + 2\delta)^{\frac{1}{2}}T^{-\frac{1}{2}}\bigr)}{1 - \tilde{E}_{g_1, g_2}\bigl(\log(T)^{\frac{1}{2}}T^{-\frac{1}{2}}\bigr)} - 1\\
\leq{}& (1 + \frac{2\delta}{T})^{\frac{\rankH - 1}{2}} e^{2\delta_{2\rho}\delta} - 1\\
{}&+ (1 + \frac{2\delta}{T})^{\frac{\rankH - 1}{2}} e^{2\delta_{2\rho}\delta}\frac{2\tilde{E}_{g_1, g_2}\bigl(\log(T + 2\delta)^{\frac{1}{2}}T^{-\frac{1}{2}}\bigr)}{1 - \tilde{E}_{g_1, g_2}\bigl(\log(T)^{\frac{1}{2}}T^{-\frac{1}{2}}\bigr)}\\
\leq{}& O(R^{-\kappa_3}) + O\bigl(e^{O(d(o, g_1o) + d(o, g_2o))}\bigr) T^{-\frac{1}{2}} \log(T)^{\frac{1}{2}}
\end{align*}
where $\kappa_3 = 1/D$ and for all $T \gg  \tilde{E}_{g_1, g_2}^2$. 

\medskip
\noindent
\textit{Step 3. Restricting to the outermost annulus of the Riemannian skew ball.}

Now we study the outermost annulus $H_{T + \delta}[g_1, g_2] \backslash H_{T' + \delta}[g_1, g_2]$ of the Riemannian skew ball and show that it occupies most of the mass.  

We decompose the integral in the main term in \cref{eqn:equiskewH_in_cone} into two pieces according to $H_{T' + \delta}[g_1, g_2] \subset H_{T + \delta}[g_1, g_2]$. We have: 
\begin{align}
\begin{aligned}\label{eqn:decomposeskewball}
{}&\int_K \sum_{j = 1}^{N} \int_K \varphi_j(k_1) \int_{\LieA_{T + \delta}[g_1 \tilde{k}_j, k_2 g_2] \cap \mathcal{C}^{\LieA}_{v_{2\rho}, \tau}} \phi(x_0 k_1 a_v k_2) \xi(v) \, dv \, d\mu_K(k_1) \, d\mu_K(k_2)\\
={}& \int_K \sum_{j = 1}^{N} \int_{\LieA_{T + \delta}[g_1 \tilde{k}_j, k_2 g_2] \cap \mathcal{C}^{\LieA}_{v_{2\rho}, \tau}} \int_K \phi(x_0 k_1 a_v k_2) \varphi_j(k_1) \, d\mu_K(k_1) \, \xi(v) \, dv \, d\mu_K(k_2)\\
={}& \int_K \sum_{j = 1}^{N} \int_{\LieA_{T', T + \delta}[g_1\tilde{k}_j, k_2g_2] \cap \mathcal{C}^{\LieA}_{v_{2\rho}, \tau}} \int_K \phi(x_0 k_1 a_v k_2) \varphi_j(k_1) \, d\mu_K(k_1) \, \xi(v) \, dv \, d\mu_K(k_2)\\
{}&+ \int_K \sum_{j = 1}^{N} \int_{\LieA_{T'}[g_1\tilde{k}_j, k_2 g_2] \cap \mathcal{C}^{\LieA}_{v_{2\rho}, \tau}} \int_K \phi(x_0 k_1 a_v k_2) \varphi_j(k_1) \, d\mu_K(k_1) \, \xi(v) \, dv \, d\mu_K(k_2)\\
={}& \int_K \sum_{j = 1}^{N} \int_{\LieA_{T', T + \delta}[g_1\tilde{k}_j, k_2g_2] \cap \mathcal{C}^{\LieA}_{v_{2\rho}, \tau}} \int_K \phi(x_0 k_1 a_v k_2) \varphi_j(k_1) \, d\mu_K(k_1) \, \xi(v) \, dv \, d\mu_K(k_2)\\
{}&+ O(\|\phi\|_{\infty}\mu_H(H_{T' + \delta}[g_1,g_2])).
\end{aligned}
\end{align}
The last inequality follows from \cref{eqn:SandwichInSkewBall}. Let $E_3 = \mu_H(H_{T' + \delta}[g_1,g_2]))$. We will show that $H_{T' + \delta}[g_1, g_2]$ does not occupy much measure. Dividing by $\mu_H(H_T[g_1, g_2])$ and using \cref{thm:SkewBallVolumeAsymptotic}, we have the following two cases according to $\rankG$.  

If $\rankG = 1$, we have: 
\begin{align*}
\frac{E_3}{\mu_H(H_T[g_1, g_2])} \leq \frac{C[g_1, g_2] e^{\delta_{2\rho} (T' + \delta)} + E[g_1, g_2]\bigl(e^{(\delta_{2\rho} - \eta_1)(T' + \delta)}\bigr)}{C[g_1, g_2] e^{\delta_{2\rho} T} - E[g_1, g_2]\bigl(e^{(\delta_{2\rho} - \eta_1)T}\bigr)} \leq R^{-\delta_{2\rho}C}
\end{align*}
for $T - C\log(R) \gg \frac{1}{\eta_1}\log(2\tilde{E}_{g_1, g_2})$.

If $\rankG \geq 2$, we have: 
\begin{align*}
\frac{E_3}{\mu_H(H_T[g_1, g_2])} \leq{}& \frac{C[g_1, g_2] (T' + \delta)^{\frac{\rankH - 1}{2}} e^{\delta_{2\rho}(T' + \delta)}}{C[g_1, g_2] T^{\frac{\rankH - 1}{2}} e^{\delta_{2\rho}T} - E[g_1, g_2]\bigl(\log(T)^{\frac{1}{2}}T^{\frac{\rankH - 2}{2}}e^{\delta_{2\rho}T}\bigr)} \\
{}&+ \frac{E[g_1, g_2]\bigl(\log(T' + \delta)^{\frac{1}{2}}(T' + \delta)^{\frac{\rankH - 2}{2}}e^{\delta_{2\rho}(T' + \delta)}\bigr)}{C[g_1, g_2] T^{\frac{\rankH - 1}{2}} e^{\delta_{2\rho}T} - E[g_1, g_2]\bigl(\log(T)^{\frac{1}{2}}T^{\frac{\rankH - 2}{2}}e^{\delta_{2\rho}T}\bigr)}\\
\leq{}& R^{-\delta_{2\rho}C}
\end{align*}
for $T - C\log(R) \gg  \tilde{E}_{g_1, g_2}^4$.

\medskip
\noindent
\textit{Step 4. Using the hypothesis, i.e., equidistribution of $K$-orbits. }

Fix $\kappa' := \varsigma_0 \tilde{\kappa}$. Applying the hypothesis in \cref{pro:SkewballEffectiveEquidistribution} for $R$ and $v$, we get the following: 

For all $1 \leq j \leq N$, $k_2 \in K$, and $v \in \LieA^+_{T + \delta}[g_1\tilde{k}_j, k_2g_2] \cap \mathcal{C}_{v_{2\rho}, \tau} \backslash \LieA_{T'}[g_1\tilde{k}_j, k_2g_2]$, we have 
\begin{align*}
\biggl|\int_K \phi(x_0k_1 a_{v}k_2)\varphi_j(k_1) \, d\mu_K(k_1) - \int_{\Gamma \backslash G} \phi \, d\hat{\mu}_{\Gamma \backslash G} \cdot \int_K \varphi_j \, d\mu_K\biggr| \leq \mathcal{S}^\ell(\phi)\mathcal{S}^{\ell}(\varphi_j)R^{-\kappa'}.
\end{align*}

Therefore, we have the following estimate: 
\begin{align*}
{}&\int_K \sum_{j = 1}^{N} \int_{\LieA_{T', T + \delta}[g_1\tilde{k}_j, k_2g_2] \cap \mathcal{C}^{\LieA}_{v_{2\rho}, \tau}} \int_K \phi(x_0 k_1 a_v k_2) \varphi_j(k_1) \, d\mu_K(k_1) \, \xi(v) \, dv \, d\mu_K(k_2)\\
={}& \int_K \sum_{j = 1}^{N} \int_{\LieA_{T', T + \delta}[g_1\tilde{k}_j, k_2g_2] \cap \mathcal{C}^{\LieA}_{v_{2\rho}, \tau}} \int_{\Gamma \backslash G} \phi \, d\hat{\mu}_{\Gamma \backslash G} \cdot \int_K \varphi_j \, d\mu_K \, \xi(v) \, dv \, d\mu_K(k_2) \\
{}&+ \int_K \sum_{j = 1}^{N} \int_{\LieA_{T', T + \delta}[g_1\tilde{k}_j, k_2g_2] \cap \mathcal{C}^{\LieA}_{v_{2\rho}, \tau}} \mathcal{S}^\ell(\phi)\mathcal{S}^\ell(\varphi_j) R^{-\kappa'} \, \xi(v) \, dv \, d\mu_K(k_2).
\end{align*}
Call the second term $E_4$ without the factor $\mathcal{S}^\ell(\phi)$. It can be bounded in the following way: 
\begin{align*}
E_4 &\ll R^{-\kappa'} \biggl(\frac{\delta}{\|\Ad_{g_1}\|_{\mathrm{op}}}\biggr)^{-\ell'} \int_K \sum_{j = 1}^{N} \int_{\LieA_{T', T + \delta}[g_1\tilde{k}_j, k_2g_2]} \, \xi(v) \, dv \, d\mu_K(k_2)\\
&\leq \frac{R^{-\kappa'}\delta^{-\ell'}}{\|\Ad_{g_1}\|_{\mathrm{op}}^{-\ell'}} \int_K \sum_{j = 1}^N \frac{1}{\mu_K(\tilde{k}_j \mathcal{O})} \int_{\tilde{k}_j \mathcal{O}} \int_{\LieA_{T + 2 \delta}[g_1 k_1, k_2 g_2]} \xi(v) \, dv \, d\mu_K(k_1) d\mu_K(k_2)\\
&\ll \|\Ad_{g_1}\|_{\mathrm{op}}^{\ell' + \dim(K)(1 + \dim(G))} R^{-\kappa'} \delta^{-(\ell' + \dim(K))} \mu_H(H_{T + 2\delta}[g_1, g_2]).
\end{align*}
The last inequality is due to the bound on the multiplicity of the open cover. We finally choose $D = \frac{2(\ell' + \dim(K))}{\kappa'}$ so that $R^{-\kappa'} \delta^{-(\ell' + \dim(K))} \leq R^{-\kappa'/2}$, and so
\begin{align*}
E_4 \ll \|\Ad_{g_1}\|_{\mathrm{op}}^{\ell' + \dim(K)(1 + \dim(G))} R^{-\kappa'/2}\mu_H(H_{T}[g_1, g_2]).
\end{align*}

Now we deal with the first term. We factor out $\int_{\Gamma \backslash G} \phi \, d\hat{\mu}_{\Gamma \backslash G}$ and estimate the remaining expression. Using \cref{eqn:SandwichInSkewBall}, we have the following estimate on the upper bound: 
\begin{align*}
{}&\int_K \sum_{j = 1}^{N} \int_{\LieA_{T', T + \delta}[g_1 \tilde{k}_j, k_2 g_2] \cap \mathcal{C}^{\LieA}_{v_{2\rho}, \tau}}  \int_K \varphi_j \, d\mu_K \, \xi(v) \, dv \, d\mu_K(k_2)\\
\leq{}& \int_K \sum_{j = 1}^{N} \int_K \int_{\LieA_{T + 2\delta}[g_1 k_1, k_2 g_2]}  \, \xi(v) \, dv \, \varphi_j(k_1) \, d\mu_K(k_1) \, d\mu_K(k_2)\\
\leq{}& \mu_H(H_{T + 2\delta}[g_1, g_2]).
\end{align*}
Similarly, using \cref{eqn:SandwichInSkewBall}, we have the following estimate on lower bound: 
\begin{align*}
{}&\int_K \sum_{j = 1}^{N} \int_{\LieA_{T', T + \delta}[g_1 \tilde{k}_j, k_2 g_2] \cap \mathcal{C}^{\LieA}_{v_{2\rho}, \tau}}  \int_K \varphi_j \, d\mu_K \, \xi(v) \, dv \, d\mu_K(k_2)\\
\geq{}& \int_K \sum_{j = 1}^{N} \int_K \varphi_j(k_1) \int_{\LieA_{T}[g_1 k_1, k_2 g_2] \cap \mathcal{C}^{\LieA}_{v_{2\rho}, \tau} \backslash \LieA_{T'}[g_1 \tilde{k}_j, k_2g_2]}  \, \xi(v) \, dv \, d\mu_K(k_1) \, d\mu_K(k_2)\\
={}& \int_K \sum_{j = 1}^{N} \int_K \varphi_j(k_1) \int_{\LieA_{T}[g_1 k_1, k_2 g_2] \cap \mathcal{C}^{\LieA}_{v_{2\rho}, \tau}}  \, \xi(v) \, dv \, d\mu_K(k_1) \, d\mu_K(k_2) \\
{}&- \mu_H(H_{T' + \delta}[g_1, g_2])\\
\geq{}& \mu_H(H_{T}[g_1, g_2]) - \mu_H(H_{T' + \delta}[g_1, g_2]) - E_1.
\end{align*}
We find that the error terms are combination of $E_1$, $E_2$, and $E_3$. 

Finally, collecting all the error terms, we get \cref{eqn:FullErrorTermWithAdjoint} where we take
\begin{align*}
\hat{\ell} &:= \ell' + \dim(K)(1 + \dim(G)) = \ell + \dim(K)\biggl(\frac{3}{2} + \dim(G)\biggr), \\
\kappa_C &:= \min\left\{\kappa_1, \kappa_2, \kappa_3, \delta_{2\rho} C, \frac{\kappa'}{2}\right\}\\
&= \min\left\{\frac{\delta_{2\rho}\tau^2}{4}C, \frac{\eta C}{2}, \frac{\eta_1 C}{2}, \frac{\varsigma_0\tilde{\kappa}}{2(\ell' + \dim(K))}, \delta_{2\rho}C, \frac{\varsigma_0\tilde{\kappa}}{2}\right\}.
\end{align*} 
For $M_{g_1, g_2} > 0$, using the estimate $\tilde{E}_{g_1, g_2} = O\bigl(e^{O(d(o, g_1o) + d(o, g_2o))}\bigr)$, we can choose
\begin{align*}
M_{g_1, g_2} = \begin{cases}
C_1(d(o, g_1o) + d(o, g_2o)), & \rankG = 1 \\
C_1e^{C_2(d(o, g_1o) + d(o, g_2o))}, & \rankG \geq 2
\end{cases}
\end{align*}
where $C_1 > 0$ and $C_2 > 0$ are two constants depending only on $(G, H)$.  
\end{proof}

\section{Effective duality}
\label{sec:EffectiveDuality}
In this section we will prove \cref{thm:EffectiveDuality}. This uses the well-known duality between $H$-orbits in $\Gamma \backslash G$ and $\Gamma$-orbits in $G/H$. The key point here is that the relation between the two is \emph{effectivized}.

We fix the following for this section. We use the notation $G_T := \{g \in G: d(o, go) < T\}$ and $\Gamma_T := \{\gamma \in \Gamma: d(o, \gamma o) < T\}$ for any $T > 0$. For all $y \in G/H$, we fix an open neighborhood $\mathcal{U}_y \subset G/H$ of $y$ and a \emph{smooth} section $\sigma_y: \mathcal{U}_y \to G$ of the principal bundle $G \to G/H$ in the following way. For $H \in G/H$, we fix any smooth section $\sigma_H: \mathcal{U}_H \to G$ whose image is a connected smooth submanifold $\sigma_H(\mathcal{U}_H) \subset B_1^G(e)$ containing $e \in G$ with smooth topological boundary so that $\delta_{\mathcal{U}} := \inf\{d(o, go): g \in \partial\sigma_H(\mathcal{U}_H)\} > 0$. For all other $gH \in G/H$, we take any optimal lift $g \in G$ such that $d(o, go) = \min\{d(o, g'o): g'H = gH\}$ and we fix $\mathcal{U}_{gH} := g\mathcal{U}_H$ and $\sigma_{gH} := m_g^{\mathrm{L}} \circ \sigma_H \circ m_{g^{-1}}^{\mathrm{L}}: \mathcal{U}_{gH} \to G$ where $m_g^{\mathrm{L}}$ and $m_{g^{-1}}^{\mathrm{L}}$ are the left multiplication maps by $g$ and $g^{-1}$.

Recall the measure $\mu_{G/H}$ on $G/H$ satisfying $d\mu_G = d\mu_H \, d\mu_{G/H}$ from \cref{subsec:TheSubgroupH}. Since $\mu_G$ is the measure induced by the volume form on $G$, which can be disintegrated over the image of any section $\sigma_y(\mathcal{U}_y)$ along the fibers of the principal bundle $G \to G/H$, we conclude that $\mu_{G/H}$ is a measure induced by a top-dimensional differential form with a positive smooth density function. Moreover, we can fix an appropriate Riemannian metric on $G/H$ which is compatible with $\mu_{G/H}$ by taking any top-dimensional differential form and normalizing it by an appropriate positive smooth density function.

\begin{remark}
Although these sections are not canonical, some choice of a family of sections on an open cover which trivializes the principal bundle $G \to G/H$ over $\supp(\psi)$ is required for the statement and proof of \cref{thm:EffectiveDuality}. Similarly, there is no canonical Riemannian metric on $G/H$ but we need one to make sense of the space of H\"{o}lder functions $C^{0, \chi}(G/H)$.
\end{remark}

For all $\psi \in C_{\mathrm{c}}(G/H)$, we define the constant
\begin{align}
\label{eqn:Dpsi}
D_\psi := \inf\bigl\{r > 0: \supp(\psi) \subset B_r^G(e) \cdot H \subset G/H\bigr\} > 0.
\end{align}
Recall $c_{\varsigma_H}$ from \ShahEE where $\varsigma_H = \frac{1}{2}\min_{\alpha \in \Phi_H^+} \alpha(v_{2\rho_H})$.

\begin{theorem}
\label{thm:EffectiveDuality}
Suppose \ShahEE holds. There exist $C_{\varsigma_H} \asymp c_{\varsigma_H}$, $\kappa \in (0, \kappa_0)$, and $\varrho \in (\varrho_0, 2\varrho_0)$ such that the following holds. Let $\chi \in (0, 1]$ and $p := \dim(G) + 1 + \chi$. Let $\psi \in C_{\mathrm{c}}^{0, \chi}(G/H)$, $g_0 \in G$, $x_0 = \Gamma g_0 \in \Gamma \backslash G$, and $y_0 = g_0 H \in G/H$. There exists $M_{\psi, g_0} > 0$ such that for all $R \gg \inj_{\Gamma \backslash G}(x_0)^{-\varrho}$ and $T \geq C_{\varsigma_H} \log(R) + M_{\psi, g_0}$, at least one of the following holds. 
\begin{enumerate}
\item We have
\begin{align*}
&\left|\frac{1}{\mu_H(H_T)}\sum_{\gamma \in \Gamma_T} \psi(\gamma y_0) - \frac{1}{\mu_H(H_T)\mu_{\Gamma \backslash G}(\Gamma \backslash G)}\int_{G_T} \psi(g y_0) \, d\mu_G(g)\right| \\
\leq{}&
\begin{cases}
O\bigl(e^{O(D_\psi + d(o, g_0o))}\bigr) \|\psi\|_{C^{0, \chi}}R^{-\chi\kappa}, & \rankG = 1 \\
O\bigl(e^{O(D_\psi + d(o, g_0o))}\bigr) \|\psi\|_{C^{0, \chi}} \bigl(T^{-\frac{1}{2p}}\log(T)^{\frac{1}{2p}} + R^{-\chi\kappa}\bigr), & \rankG \geq 2.
\end{cases}
\end{align*}
\item There exists $x \in \Gamma \backslash G$ with
\begin{align*}
d(x_0, x) \leq R^{C_{\varsigma_H}} T^{C_{\varsigma_H}} e^{-\varsigma_HT}
\end{align*}
such that $xH$ is periodic with $\vol(xH) \leq R$.
\end{enumerate}
Moreover, we can choose $M_{\psi, g_0} = C_1e^{C_2(D_\psi + d(o, g_0o))}$ for some absolute constants $C_1 > 0$ and $C_2 > 0$.
\end{theorem}

\Cref{thm:EffectiveDuality} follows from \cref{thm:SkewballEffectiveEquidistributionDichotomy} and the following proposition. The proof of \cref{thm:EffectiveDuality} assuming \cref{pro:EffectiveDuality'} is very similar to the proof of \cref{thm:SkewballEffectiveEquidistributionDichotomy} assuming \cref{pro:SkewballEffectiveEquidistribution}. Hence, we omit the derivation.

\begin{proposition}
\label{pro:EffectiveDuality'}
Let $\chi \in (0, 1]$ and $\psi \in C_{\mathrm{c}}^{0, \chi}(G/H)$. Let $g_0 \in G$, $x_0 = \Gamma g_0 \in \Gamma \backslash G$, and $y_0 = g_0H \in G/H$. There exists $M_{\psi, g_0} > 0$ such that the following holds:

Suppose that there exists $\tilde{\kappa} > 0$, $p := \dim(G) + 1 + \chi$, $\kappa := \tilde{\kappa} (\dim(G) + 2)^{-1} \cdot \Bigl(\ell + \frac{\dim(G)}{2} + 1\Bigr)^{-1}$, $R > 0$, $\tilde{C}_{\varsigma_H} > 0$, and $T \geq \tilde{C}_{\varsigma_H} \log(R) + M_{\psi, g_0}$ such that for all $b \in G$, $\phi \in C_{\mathrm{c}}^\infty(\Gamma \backslash G)$, and $T' \in [T - 1, T + 1]$, we have
\begin{align*}
&\left|\frac{1}{\mu_H(H_{T'}[g_0, b])} \int_{H_{T'}[g_0, b]} \phi(x_0h) \, d\mu_H(h) - \int_{\Gamma \backslash G} \phi(x) \, d\hat{\mu}_{\Gamma \backslash G}(x)\right| \\
\leq{}&
\begin{cases}
O\bigl(e^{O(d(o, g_0o))}\bigr) \mathcal{S}^{\ell}(\phi) R^{-\tilde{\kappa}}, & \rankG = 1 \\
O\bigl(e^{O(d(o, g_0o) + d(o, bo))}\bigr) \|\phi\|_\infty T^{-\frac{1}{2}} \log(T)^{\frac{1}{2}} + O\bigl(e^{O(d(o, g_0o))}\bigr) \mathcal{S}^{\ell}(\phi)R^{-\tilde{\kappa}}, & \rankG \geq 2.
\end{cases}
\end{align*}
Then, we have
\begin{align*}
&\left|\frac{1}{\mu_H(H_T)}\sum_{\gamma \in \Gamma_T} \psi(\gamma y_0) - \frac{1}{\mu_H(H_T)\mu_{\Gamma \backslash G}(\Gamma \backslash G)}\int_{G_T} \psi(g y_0) \, d\mu_G(g)\right| \\
\leq{}&
\begin{cases}
O\bigl(e^{O(D_\psi + d(o, g_0o))}\bigr) \|\psi\|_{C^{0, \chi}} R^{-\chi\kappa}, & \rankG = 1 \\
O\bigl(e^{O(D_\psi + d(o, g_0o))}\bigr) \|\psi\|_{C^{0, \chi}} \bigl(T^{-\frac{1}{2p}}\log(T)^{\frac{1}{2p}} + R^{-\chi\kappa}\bigr), & \rankG \geq 2.
\end{cases}
\end{align*}
Moreover, we can choose $M_{\psi, g_0} = C_1e^{C_2(D_\psi + d(o, g_0o))}$ for some absolute constants $C_1 > 0$ and $C_2 > 0$.
\end{proposition}

\begin{proof}
Let $\psi$, $g_0$, $x_0$, $y_0$, $\tilde{\kappa}$, $R$, $\tilde{C}_{\varsigma_H}$, and $T$ be as in the proposition and suppose that the hypothesis of the proposition holds.

By a standard convolution trick as in \cite[Appendix]{KM96} and \cite[Corollary 5.2]{MW12} and a Sobolev norm estimate as in property~(3) in \cref{lem:Sobolev}, the hypothesis also holds for all $\chi$-H\"{o}lder continuous functions $\phi \in C_{\mathrm{c}}^{0, \chi}(\Gamma \backslash G)$ with the error term
\begin{align*}
\begin{cases}
O\bigl(e^{O(d(o, g_0o))}\bigr) \mathcal{S}^{\ell}(\phi) R^{-\tilde{\kappa}}, & \rankG = 1 \\
O\bigl(e^{O(d(o, g_0o) + d(o, bo))}\bigr) \|\phi\|_{\infty} T^{-\frac{1}{2}} \log(T)^{\frac{1}{2}} + O\bigl(e^{O(d(o, g_0o))}\bigr) \mathcal{S}^{\ell}(\phi)R^{-\tilde{\kappa}}, & \rankG \geq 2
\end{cases}
\end{align*}
replaced with
\begin{align*}
\begin{cases}
O\bigl(e^{O(d(o, g_0o))}\bigr) \|\phi\|_{C^{0, \chi}} R^{-\chi\kappa'}, & \rankG = 1 \\
O\bigl(e^{O(d(o, g_0o) + d(o, bo))}\bigr) \|\phi\|_\infty T^{-\frac{1}{2}} \log(T)^{\frac{1}{2}} + O\bigl(e^{O(d(o, g_0o))}\bigr) \|\phi\|_{C^{0, \chi}} R^{-\chi\kappa'}, & \rankG \geq 2
\end{cases}
\end{align*}
where
\begin{align*}
\kappa' := \tilde{\kappa}\left(\ell + \frac{\dim(G)}{2} + 1\right)^{-1} \leq \tilde{\kappa}\left(\ell + \frac{\dim(G)}{2} + \chi\right)^{-1}.
\end{align*}
We outline the proof here for completeness. Convolving $\phi \in C_{\mathrm{c}}^{0, \chi}(\Gamma \backslash G)$ with an appropriate nonnegative smooth bump function $\varphi \in C^\infty(G)$ supported in $B_\delta^G(e)$ for some $\delta > 0$ with $\int_G \varphi \, d\mu_G = 1$ and $\mathcal{S}^\ell(\varphi) \ll \delta^{-(\ell + \frac{\dim G}{2})}$, one can approximate $\phi$ by a smooth compactly supported function $\phi_{\delta}$ such that
\begin{align*}
\|\phi - \phi_{\delta}\|_{\infty} &\leq \|\phi\|_{C^{0, \chi}} \delta^{\chi}, & \mathcal{S}^\ell(\phi_{\delta}) &\ll \|\phi\|_{C^{0, \chi}} \delta^{-(\ell + \frac{\dim G}{2})}.
\end{align*}
Now we simply take $\delta = R^{-\tilde{\kappa}(\ell + \frac{\dim(G)}{2} + \chi)^{-1}} \leq R^{-\kappa'}$.

\begin{figure}
\centering
\includegraphics{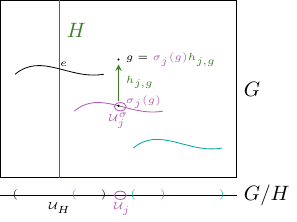}
\caption{Cover for $\supp(\psi)$}
\label{fig:PrincipalBundleTrivialization}
\end{figure}

We fix a finite open cover $\{\mathcal{U}_{y_k}\}_{k = 1}^n$ for the compact set $\supp(\psi)$ for some $\{y_k\}_{k = 1}^n \subset \supp(\psi)$ and $n \in \N$ so that $\bigl\{\sigma_{y_k}^{-1}\bigl(\sigma_{y_k}(y_k)B_{\delta_\mathcal{U}}^G(e)\bigr)\bigr\}_{k = 1}^n$ also covers $\supp(\psi)$ and $\bigcup_{k = 1}^n \overline{\mathcal{U}_{y_k}} \subset \{gH \in G/H: d(o, go) \leq D_\psi + 1\}$. We introduce a parameter $\delta \in \bigl(0, \min\bigl\{\frac{1}{2}, \frac{\delta_\mathcal{U}}{2}\bigr\}\bigr)$ which depends on $R$ and $T$ which we explicate later. Now, depending on $\supp(\psi)$ and $\delta$, take a finite open cover $\{\mathcal{U}_j\}_{j = 1}^N$ for $\supp(\psi)$ which is subordinate to $\{\mathcal{U}_{y_k}\}_{k = 1}^n$. Let $1 \leq j \leq N$ and fix a corresponding $1 \leq k \leq n$ such that $\mathcal{U}_j \subset \mathcal{U}_{y_k}$. For convenience, we set $\sigma_j := \sigma_{y_k}$ and extend it to a Borel section $\sigma_j: G/H \to G$. We may now assume that $\mathcal{U}_j^\sigma := \sigma_j(\mathcal{U}_j) = b_j B_\delta^G(e) \cap \sigma_{y_k}(\mathcal{U}_{y_k})$ for some $b_j \in \sigma_j(\mathcal{U}_j)$, which is compactly contained in $\sigma_{y_k}(\mathcal{U}_{y_k})$. Write $\check{\mathcal{U}}_j := \sigma_j^{-1}\bigl(b_j B_{\delta/2}^G(e) \cap \sigma_{y_k}(\mathcal{U}_{y_k})\bigr)$. We may also assume that $\{\check{\mathcal{U}}_j\}_{j = 1}^N$ is an open cover for $\supp(\psi)$. Fix the constant $E_\psi = e^{c_\Phi D_\psi} > 1$ and impose the condition
\begin{align}
\label{eqn:deltaCondition}
\delta < E_\psi^{-1}
\end{align}
so that using \cref{lem:NormOfAdjointEstimate} as in \cref{eqn:ConjugationDifferential}, we get the important containments
\begin{align}
\label{eqn:ContainmentOfSkewBallsInG}
b B_\delta^G(e) b^{-1} \subset B_{\|\Ad_b\|_{\mathrm{op}}\delta}^G(e) \subset B_{E_\psi\delta}^G(e) \subset B_1^G(e)
\end{align}
for all $b \in \bigcup_{j = 1}^N \overline{\mathcal{U}_j^\sigma}$. Moreover, using the left $G$-invariance of the metric on $G$ and a general Besicovitch covering theorem from \cite[Chapter 2, \S 2.8, Theorem 2.8.14]{Fed69} on $\sigma_H(\mathcal{U}_H)$ and induction on the finite open cover for $\supp(\psi)$, we can also ensure that the multiplicity of the open cover $\{\mathcal{U}_j\}_{j = 1}^N$ is bounded by a constant $m_\mathcal{U} \in \N$ depending only on $\mathcal{U}_H$ and in particular, uniform in $\psi$ and $\delta$. We deduce that the open sets in the cover have measures of the order $O\bigl(\delta^{\dim(G/H)}\bigr)$ again using left $G$-invariance of $\mu_G$ and the fixed bounded set $\sigma_H(\mathcal{U}_H)$.

We denote the positive and negative parts of $\psi$ by
\begin{align*}
\psi^+ &= \max(\psi, 0) \geq 0, & \psi^- &= \min(\psi, 0) \leq 0
\end{align*}
which are still $\chi$-H\"{o}lder continuous. We will use a set of nonnegative smooth bump functions $\bigl\{\varphi_{\delta, j}^{G/H}\bigr\}_{j = 1}^N$ on $\bigcup_{j = 1}^N \mathcal{U}_j$ subordinate to $\{\mathcal{U}_j\}_{j = 1}^N$, which is a partition of unity for $\bigcup_{j = 1}^N \check{\mathcal{U}}_j$. Clearly, we can ensure
\begin{align*}
\bigl|\varphi_{\delta, j}^{G/H}\bigr|_{C^1} &\ll \delta^{-1}, & \bigl|\varphi_{\delta, j}^{G/H}\bigr|_{C^{0, \chi}} &\ll \bigl|\varphi_{\delta, j}^{G/H}\bigr|_{C^1}\delta^{1 - \chi} \ll \delta^{-\chi}
\end{align*}
for all $1 \leq j \leq N$. For brevity, we write $\psi_{\delta, j} = \psi \cdot \varphi_{\delta, j}^{G/H}$ and $\psi_{\delta, j}^\pm = \psi^\pm \cdot \varphi_{\delta, j}^{G/H}$ for all $1 \leq j \leq N$.

Similar to above, take the open symmetric neighborhood $\mathcal{O} = B_\delta^H(e) \subset H$ so that by triangle inequality and left $G$-invariance of the metric and \cref{eqn:ConjugationDifferential}, we have
\begin{align}
\label{eqn:TriangleInequalityLemma}
d(o, ghb^{-1}o) < d(o, gb^{-1}o) + d(o, bhb^{-1}o) < d(o, gb^{-1}o) + E_\psi\delta
\end{align}
for all $g \in G$, $h \in \mathcal{O}$, and $b \in \bigcup_{j = 1}^N \overline{\mathcal{U}_j^\sigma}$. Take a smooth nonnegative bump function $\varphi_\delta^H \in C_{\mathrm{c}}^\infty(H)$ such that
\begin{align*}
\supp\bigl(\varphi_\delta^H\bigr) &\subset \mathcal{O}, & \int_H \varphi_\delta^H \, d\mu_H &= 1.
\end{align*}
This will be used to ``thicken'' the function $\psi$ akin to Margulis' thickening argument. We can also ensure that
\begin{align*}
\bigl\|\varphi_\delta^H\bigr\|_\infty &\ll \delta^{-\dim(H)}, \\
\bigl|\varphi_\delta^H\bigr|_{C^1} &\ll \delta^{-(\dim(H) + 1)}, \\
\bigl|\varphi_\delta^H\bigr|_{C^{0, \chi}} &\ll \bigl|\varphi_\delta^H\bigr|_{C^1} \delta^{1 - \chi} \ll \delta^{-(\dim(H) + \chi)}
\end{align*}
again using \cref{lem:NormOfAdjointEstimate}.

Let $1 \leq j \leq N$. Abusing notation, we define $\sigma_j(g) := \sigma_j(gH)$ and $h_{j, g} := \sigma_j(g)^{-1}g \in H$ so that $g = \sigma_j(g)h_{j, g}$, for all $g \in G$ (See \cref{fig:PrincipalBundleTrivialization}). We define the functions $\Phi_{\delta, j} \in C_{\mathrm{c}}^{0, \chi}(G)$ and $\phi_{\delta, j} \in C_{\mathrm{c}}^{0, \chi}(\Gamma \backslash G)$ by
\begin{align*}
\Phi_{\delta, j}(g) &= \psi_{\delta, j}(gH)\varphi_\delta^H(h_{j, g}), & \phi_{\delta, j}(\Gamma g) &= \sum_{\gamma \in \Gamma} \Phi_{\delta, j}(\gamma g), \qquad \text{for all $g \in G$}
\end{align*}
and define $\Phi_{\delta, j}^\pm \in C_{\mathrm{c}}^{0, \chi}(G)$ and $\phi_{\delta, j}^\pm \in C_{\mathrm{c}}^{0, \chi}(\Gamma \backslash G)$ analogously. The above sum is actually a finite sum for all $\Gamma g \in \Gamma \backslash G$ and in fact, the number of nonzero summands is bounded above by an absolute constant depending only on $\Gamma$. Using the above bounds on various seminorms and norms, we have
\begin{align*}
\|\Phi_{\delta, j}\|_{C^{0, \chi}} \ll{}& \|\psi\|_\infty\bigl\|\varphi_{\delta, j}^{G/H}\bigr\|_\infty\bigl\|\varphi_\delta^H\bigr\|_\infty + |\psi|_{C^{0, \chi}}\bigl\|\varphi_{\delta, j}^{G/H}\bigr\|_\infty\bigl\|\varphi_\delta^H\bigr\|_\infty \\
{}&+ \|\psi\|_\infty\bigl|\varphi_{\delta, j}^{G/H}\bigr|_{C^{0, \chi}}\bigl\|\varphi_\delta^H\bigr\|_\infty + \|\psi\|_\infty\bigl\|\varphi_{\delta, j}^{G/H}\bigr\|_\infty\bigl|\varphi_\delta^H\bigr|_{C^{0, \chi}} \\
\leq{}& \|\psi\|_\infty\delta^{-\dim(H)} + |\psi|_{C^{0, \chi}}\delta^{-\dim(H)} \\
{}&+ \|\psi\|_\infty\delta^{-\chi}\delta^{-\dim(H)} + \|\psi\|_\infty\delta^{-(\dim(H) + \chi)} \\
\leq{}&\|\psi\|_{C^{0, \chi}} \delta^{-(\dim(H) + \chi)}
\end{align*}
and similarly with superscripts $\pm$. Consequently, we have
\begin{align*}
\|\phi_{\delta, j}\|_\infty &\ll \|\psi\|_\infty\delta^{-\dim(H)}, & \|\phi_{\delta, j}\|_{C^{0, \chi}} &\ll \|\psi\|_{C^{0, \chi}}\delta^{-(\dim(H) + \chi)}
\end{align*}
and similarly with superscripts $\pm$.

Let $1 \leq j \leq N$ and $g \in G_T$ such that $gg_0 \in \mathcal{U}_j^\sigma H$. Following definitions, we get
\begin{align*}
d\bigl(o, g_0 h_{j, gg_0}^{-1} \sigma_j(gg_0)^{-1}o\bigr) = d(o, g_0 (gg_0)^{-1}o) = d(o, go) < T.
\end{align*}
Hence by \cref{eqn:TriangleInequalityLemma}, $h \in \mathcal{O}$ implies
\begin{align*}
d\bigl(o, g_0 h_{j, gg_0}^{-1}h \sigma_j(gg_0)^{-1}o\bigr) < T + E_\psi\delta.
\end{align*}
Consequently, $\supp\bigl(\varphi_\delta^H\bigr) \subset \mathcal{O} \subset h_{j, gg_0}H_{T + E_\psi\delta}[g_0, \sigma_j(gg_0)^{-1}]$. Calculating as in \cite[Section 4]{GW07} using the above, we get
\begin{align}
\label{eqn:psiAsIntegral}
\begin{aligned}
\psi_{\delta, j}(gg_0H) &= \int_{H_{T + E_\psi\delta}[g_0, \sigma_j(gg_0)^{-1}]} \psi_{\delta, j}(gg_0hH)\varphi_\delta^H(h_{j, gg_0h}) \, d\mu_H(h) \\
&= \int_{H_{T + E_\psi\delta}[g_0, \sigma_j(gg_0)^{-1}]} \Phi_{\delta, j}(gg_0h) \, d\mu_H(h)
\end{aligned}
\end{align}
and similarly with superscripts $\pm$. Using these formulas and \cref{eqn:TriangleInequalityLemma}, we obtain
\begin{align}
\label{eqn:psi+<=}
\begin{aligned}
\sum_{\gamma \in \Gamma_T} \psi_{\delta, j}^+(\gamma y_0) &= \sum_{\gamma \in \Gamma_T} \int_{H_{T + E_\psi\delta}[g_0, \sigma_j(\gamma g_0)^{-1}]} \Phi_{\delta, j}^+(\gamma g_0 h) \, d\mu_H(h) \\
&= \sum_{\gamma \in \Gamma_T} \int_{H_{T + 2E_\psi\delta}[g_0, b_j^{-1}]} \Phi_{\delta, j}^+(\gamma g_0 h) \, d\mu_H(h) \\
&= \int_{H_{T + 2E_\psi\delta}[g_0, b_j^{-1}]} \sum_{\gamma \in \Gamma_T} \Phi_{\delta, j}^+(\gamma g_0 h) \, d\mu_H(h) \\
&\leq \int_{H_{T + 2E_\psi\delta}[g_0, b_j^{-1}]} \phi_{\delta, j}^+(x_0 h) \, d\mu_H(h).
\end{aligned}
\end{align}
Similarly, we have
\begin{align}
\label{eqn:psi-<=}
\begin{aligned}
\sum_{\gamma \in \Gamma_T} \psi_{\delta, j}^-(\gamma y_0) &= \sum_{\gamma \in \Gamma_T} \int_{H_{T + E_\psi\delta}[g_0, \sigma_j(\gamma g_0)^{-1}]} \Phi_{\delta, j}^-(\gamma g_0 h) \, d\mu_H(h) \\
&\leq \sum_{\gamma \in \Gamma_T} \int_{H_{T - 2E_\psi\delta}[g_0, b_j^{-1}]} \Phi_{\delta, j}^-(\gamma g_0 h) \, d\mu_H(h) \\
&= \int_{H_{T - 2E_\psi\delta}[g_0, b_j^{-1}]} \sum_{\gamma \in \Gamma_T} \Phi_{\delta, j}^-(\gamma g_0 h) \, d\mu_H(h) \\
&= \int_{H_{T - 2E_\psi\delta}[g_0, b_j^{-1}]} \phi_{\delta, j}^-(x_0 h) \, d\mu_H(h).
\end{aligned}
\end{align}
The last equality holds for the following reason. First, $h \in H_{T - 2E_\psi\delta}[g_0, b_j^{-1}]$ means $d(o, g_0hb_j^{-1}o) < T - 2E_\psi\delta$ which, by a similar inequality as \cref{eqn:TriangleInequalityLemma} using \cref{eqn:ContainmentOfSkewBallsInG}, implies $d(o, g_0hb^{-1}o) < d(o, g_0hb_j^{-1}o) + E_\psi\delta < T - E_\psi\delta$ for all $b \in \mathcal{U}_j^\sigma$. Writing $\gamma g_0 h = bh' \in \mathcal{U}_j^\sigma \mathcal{O}$, we have $\gamma^{-1} = g_0h (h')^{-1}b^{-1}$ and again \cref{eqn:TriangleInequalityLemma} gives $d(o, \gamma^{-1} o) < d(o, g_0hb^{-1}o) + E_\psi\delta < T$, meaning that $\gamma \in \Gamma_T$. Combining the above two inequalities, we obtain
\begin{align*}
\sum_{\gamma \in \Gamma_T} \psi_{\delta, j}(\gamma y_0) \leq{}&\int_{H_{T - 2E_\psi\delta}[g_0, b_j^{-1}]} \phi_{\delta, j}(x_0 h) \, d\mu_H(h) \\
{}&+ \int_{H_{T + 2E_\psi\delta}[g_0, b_j^{-1}] \setminus H_{T - 2E_\psi\delta}[g_0, b_j^{-1}]} \phi_{\delta, j}^+(x_0 h) \, d\mu_H(h).
\end{align*}
Similar calculations for the reverse inequality yields
\begin{align*}
\sum_{\gamma \in \Gamma_T} \psi_{\delta, j}(\gamma y_0) \geq{}&\int_{H_{T - 2E_\psi\delta}[g_0, b_j^{-1}]} \phi_{\delta, j}(x_0 h) \, d\mu_H(h) \\
{}&+ \int_{H_{T + 2E_\psi\delta}[g_0, b_j^{-1}] \setminus H_{T - 2E_\psi\delta}[g_0, b_j^{-1}]} \phi_{\delta, j}^-(x_0 h) \, d\mu_H(h).
\end{align*}
We can combine and simplify the previous two inequalities to get
\begin{multline}
\label{eqn:psiSumApproximation}
\Biggl|\sum_{\gamma \in \Gamma_T} \psi_{\delta, j}(\gamma y_0) - \int_{H_{T - 2E_\psi\delta}[g_0, b_j^{-1}]} \phi_{\delta, j}(x_0 h) \, d\mu_H(h)\Biggr| \\
\leq \int_{H_{T + 2E_\psi\delta}[g_0, b_j^{-1}] \setminus H_{T - 2E_\psi\delta}[g_0, b_j^{-1}]} |\phi_{\delta, j}(x_0 h)| \, d\mu_H(h).
\end{multline}
Now, we deal with the error term in \cref{eqn:psiSumApproximation}. First, we rewrite the error term as
\begin{align*}
\int_{H_{T + 2E_\psi\delta}[g_0, b_j^{-1}]} |\phi_{\delta, j}(x_0 h)| \, d\mu_H(h) - \int_{H_{T - 2E_\psi\delta}[g_0, b_j^{-1}]} |\phi_{\delta, j}(x_0 h)| \, d\mu_H(h).
\end{align*}
Recalling that $|\phi_{\delta, j}|$ is $\chi$-H\"{o}lder continuous, we can apply the hypothesis regarding equidistribution of Riemannian skew balls to both integrals. The error term in \cref{eqn:psiSumApproximation} can then be bounded above by
\begin{align*}
&\bigl(\mu_H(H_{T + 2E_\psi\delta}[g_0, b_j^{-1}]) - \mu_H(H_{T - 2E_\psi\delta}[g_0, b_j^{-1}])\bigr)\int_{\Gamma \backslash G} |\phi_{\delta, j}| \, d\hat{\mu}_{\Gamma \backslash G} \\
+{}&\bigl(\mu_H(H_{T + 2E_\psi\delta}[g_0, b_j^{-1}]) + \mu_H(H_{T - 2E_\psi\delta}[g_0, b_j^{-1}])\bigr) \\
\cdot{}&
\begin{cases}
O\bigl(e^{O(d(o, g_0o) + d(o, b_jo))}\bigr) \|\phi_{\delta, j}\|_{C^{0, \chi}} R^{-\chi\kappa'}, & \rankG = 1 \\
O\bigl(e^{O(d(o, g_0o) + d(o, b_jo))}\bigr) \bigl(\|\phi_{\delta, j}\|_\infty T^{-\frac{1}{2}} \log(T)^{\frac{1}{2}} + \|\phi_{\delta, j}\|_{C^{0, \chi}} R^{-\chi\kappa'}\bigr), & \rankG \geq 2.
\end{cases}
\end{align*}

We can carry out analogous calculations using formulas indicated in \cref{eqn:psiAsIntegral} where sums over $\Gamma_T$ are replaced with integrals over $G_T$. Similar to \cref{eqn:psi+<=}, we have
\begin{align*}
\int_{G_T} \psi_{\delta, j}^+(gy_0) \, d\mu_G(g) &= \int_{G_T} \int_{H_{T + E_\psi\delta}[g_0, \sigma_j(\gamma g_0)^{-1}]} \Phi_{\delta, j}^+(g g_0 h) \, d\mu_H(h) \, d\mu_G(g) \\
&= \int_{H_{T + 2E_\psi\delta}[g_0, b_j^{-1}]} \int_{G_T} \Phi_{\delta, j}^+(g g_0 h) \, d\mu_G(g) \, d\mu_H(h) \\
&\leq \int_{H_{T + 2E_\psi\delta}[g_0, b_j^{-1}]} \int_G \Phi_{\delta, j}^+ \, d\mu_G \, d\mu_H(h) \\
&= \mu_H(H_{T + 2E_\psi\delta}[g_0, b_j^{-1}]) \int_{\Gamma \backslash G} \phi_{\delta, j}^+ \, d\mu_{\Gamma \backslash G}.
\end{align*}
Similar to \cref{eqn:psi-<=} and the above, we also have
\begin{align*}
\int_{G_T} \psi_{\delta, j}^-(gy_0) \, d\mu_G(g) \leq \mu_H(H_{T - 2E_\psi\delta}[g_0, b_j^{-1}]) \int_{\Gamma \backslash G} \phi_{\delta, j}^- \, d\mu_{\Gamma \backslash G}.
\end{align*}
Combining the above two inequalities gives
\begin{align*}
&\int_{G_T} \psi_{\delta, j}(gy_0) \, d\mu_G(g) \\
\leq{}&\mu_H(H_{T - 2E_\psi\delta}[g_0, b_j^{-1}]) \int_{\Gamma \backslash G} \phi_{\delta, j} \, d\mu_{\Gamma \backslash G} \\
&{}+\bigl(\mu_H(H_{T + 2E_\psi\delta}[g_0, b_j^{-1}]) - \mu_H(H_{T - 2E_\psi\delta}[g_0, b_j^{-1}])\bigr) \int_{\Gamma \backslash G} \phi_{\delta, j}^+ \, d\mu_{\Gamma \backslash G}.
\end{align*}
Similar calculations for the reverse inequality yields
\begin{align*}
&\int_{G_T} \psi_{\delta, j}(gy_0) \, d\mu_G(g) \\
\geq{}&\mu_H(H_{T - 2E_\psi\delta}[g_0, b_j^{-1}]) \int_{\Gamma \backslash G} \phi_{\delta, j} \, d\mu_{\Gamma \backslash G} \\
&{}+\bigl(\mu_H(H_{T + 2E_\psi\delta}[g_0, b_j^{-1}]) - \mu_H(H_{T - 2E_\psi\delta}[g_0, b_j^{-1}])\bigr) \int_{\Gamma \backslash G} \phi_{\delta, j}^- \, d\mu_{\Gamma \backslash G}.
\end{align*}
As before, we can combine the previous two inequalities to get
\begin{multline}
\label{eqn:psiIntegralApproximation}
\Biggl|\int_{G_T} \psi_{\delta, j}(gy_0) \, d\mu_G(g) - \mu_H(H_{T - 2E_\psi\delta}[g_0, b_j^{-1}]) \int_{\Gamma \backslash G} \phi_{\delta, j} \, d\mu_{\Gamma \backslash G}\Biggr| \\
\leq \bigl(\mu_H(H_{T + 2E_\psi\delta}[g_0, b_j^{-1}]) - \mu_H(H_{T - 2E_\psi\delta}[g_0, b_j^{-1}])\bigr) \int_{\Gamma \backslash G} |\phi_{\delta, j}| \, d\mu_{\Gamma \backslash G}.
\end{multline}

Now, we treat the integral $\int_{H_{T - 2E_\psi\delta}[g_0, b_j^{-1}]} \phi_{\delta, j}(x_0 h) \, d\mu_H(h)$. As before, using the hypothesis regarding equidistribution of Riemannian skew balls, we obtain
\begin{align}
\label{eqn:EquidistributionHypothesis}
\begin{aligned}
&\Biggl|\int_{H_{T - 2E_\psi\delta}[g_0, b_j^{-1}]} \phi_{\delta, j}(x_0 h) \, d\mu_H(h) - \mu_H(H_{T - 2E_\psi\delta}[g_0, b_j^{-1}]) \int_{\Gamma \backslash G} \phi_{\delta, j} \, d\hat{\mu}_{\Gamma \backslash G}\Biggr| \\
\leq{}&\mu_H(H_{T - 2E_\psi\delta}[g_0, b_j^{-1}]) \\
&{}\cdot
\begin{cases}
O\bigl(e^{O(d(o, g_0o) + d(o, b_jo))}\bigr) \|\phi_{\delta, j}\|_{C^{0, \chi}} R^{-\chi\kappa'}, & \rankG = 1 \\
O\bigl(e^{O(d(o, g_0o) + d(o, b_jo))}\bigr) \bigl(\|\phi_{\delta, j}\|_\infty T^{-\frac{1}{2}} \log(T)^{\frac{1}{2}} + \|\phi_{\delta, j}\|_{C^{0, \chi}} R^{-\chi\kappa'}\bigr), & \rankG \geq 2.
\end{cases}
\end{aligned}
\end{align}

Thus, summing over $1 \leq j \leq N$ and using the triangle inequality and \cref{eqn:psiSumApproximation,eqn:psiIntegralApproximation,eqn:EquidistributionHypothesis} gives
\begin{align*}
&\Biggl|\frac{1}{\mu_H(H_T)}\sum_{\gamma \in \Gamma_T} \psi(\gamma y_0) - \frac{1}{\mu_H(H_T)\mu_{\Gamma \backslash G}(\Gamma \backslash G)}\int_{G_T} \psi(g y_0) \, d\mu_G(g)\Biggr| \\
\leq{}&2\sum_{j = 1}^N \frac{\mu_H(H_{T + 2E_\psi\delta}[g_0, b_j^{-1}]) - \mu_H(H_{T - 2E_\psi\delta}[g_0, b_j^{-1}])}{\mu_H(H_T)} \int_{\Gamma \backslash G} |\phi_{\delta, j}| \, d\hat{\mu}_{\Gamma \backslash G} \\
&{}+ 3\sum_{j = 1}^N \frac{\mu_H(H_{T + 2E_\psi\delta}[g_0, b_j^{-1}])}{\mu_H(H_T)} \\
&{}\cdot
\begin{cases}
O\bigl(e^{O(D_\psi + d(o, g_0o))}\bigr) \|\phi_{\delta, j}\|_{C^{0, \chi}} R^{-\chi\kappa'}, & \rankG = 1 \\
O\bigl(e^{O(D_\psi + d(o, g_0o))}\bigr) \bigl(\|\phi_{\delta, j}\|_\infty T^{-\frac{1}{2}} \log(T)^{\frac{1}{2}} + \|\phi_{\delta, j}\|_{C^{0, \chi}} R^{-\chi\kappa'}\bigr), & \rankG \geq 2.
\end{cases}
\end{align*}
We call these two sums, $E_1$ and $E_2$ respectively, which we bound above using similar techniques as in the proof of \cref{pro:SkewballEffectiveEquidistribution}.

We first bound $E_2$ above. In the $\rankG = 1$ case, for all $1 \leq j \leq N$, we use \cref{thm:SkewBallVolumeAsymptotic} with the same notations to calculate that
\begin{align*}
&\frac{\mu_H(H_{T + 2E_\psi\delta}[g_0, b_j^{-1}])}{\mu_H(H_T)} = \frac{\mu_H(H_{T + 2E_\psi\delta}[g_0, b_j^{-1}])e^{-\delta_{2\rho} T}}{\mu_H(H_T)e^{-\delta_{2\rho} T}} \\
={}&\frac{C[g_0, b_j^{-1}] e^{2E_\psi\delta\delta_{2\rho}} + E[g_0, b_j^{-1}] e^{-\eta_1 T + 2E_\psi\delta(\delta_{2\rho} - \eta_1)}}{C[e, e] + E[e, e] e^{-\eta_1 T}} \\
={}&O\bigl(e^{O(D_\psi + d(o, g_0o))}\bigr).
\end{align*}
Similarly, in the $\rankG \geq 2$ case, for all $1 \leq j \leq N$, we get
\begin{align*}
&\frac{\mu_H(H_{T + 2E_\psi\delta}[g_0, b_j^{-1}])}{\mu_H(H_T)} = \frac{\mu_H(H_{T + 2E_\psi\delta}[g_0, b_j^{-1}])T^{-\frac{\rankH - 1}{2}} e^{-\delta_{2\rho}T}}{\mu_H(H_T)T^{-\frac{\rankH - 1}{2}} e^{-\delta_{2\rho}T}} \\
={}&\frac{\splitfrac{\scriptstyle C[g_0, b_j^{-1}] \bigl(1 + \frac{2E_\psi\delta}{T}\bigr)^{\frac{\rankH - 1}{2}} e^{2E_\psi\delta\delta_{2\rho}}}{\scriptstyle + E[g_0, b_j^{-1}] (T + 2E_\psi\delta)^{-\frac{1}{2}} \log(T + 2E_\psi\delta)^{\frac{1}{2}}\bigl(1 + \frac{2E_\psi\delta}{T}\bigr)^{\frac{\rankH - 1}{2}}e^{2E_\psi\delta\delta_{2\rho}}}}{C[e, e] + E[e, e] T^{-\frac{1}{2}} \log(T)^{\frac{1}{2}}} \\
={}&O\bigl(e^{O(D_\psi + d(o, g_0o))}\bigr).
\end{align*}
Next, recalling the bound on the nonzero summands, we have $\|\phi_{\delta, j}\|_\infty \ll \|\psi\|_\infty$ for all $1 \leq j \leq N$. Since the multiplicity of the open cover $\{\mathcal{U}_j\}_{j = 1}^N$ for $\supp(\psi)$, which we recall consists of open sets with measures of the order $O\bigl(\delta^{\dim(G/H)}\bigr)$, is bounded above by $m_\mathcal{U}$, the number of open sets is
\begin{align*}
N \ll \mu_{G/H}\Biggl(\bigcup_{j = 1}^N \mathcal{U}_j\Biggr) \delta^{-(\dim(G) - \dim(H))}.
\end{align*}
Let $\mathcal{P}$ denote the coarsest measurable partition of $\bigcup_{j = 1}^N \mathcal{U}_j$ formed by the open cover $\{\mathcal{U}_j\}_{j = 1}^N$. For each $P \in \mathcal{P}$, we can assign a choice of section $\sigma_P := \sigma_j$ for some $1 \leq j \leq N$ such that $P \subset \mathcal{U}_j$. We can then bound
\begin{align}
\label{eqn:VolumeBoundFor_supp_psi}
\begin{aligned}
\mu_{G/H}(\supp(\psi)) &\leq \mu_{G/H}\Biggl(\bigcup_{j = 1}^N \mathcal{U}_j\Biggr) = \sum_{P \in \mathcal{P}} \int_P \mathds{1} \, d(\sigma_P)_*(\mu_{G/H}) \\
&\ll \mu_G(G_{D_\psi + 2}) \leq O\bigl(e^{O(D_\psi + 2)}\bigr) \leq O\bigl(e^{O(D_\psi)}\bigr)
\end{aligned}
\end{align}
by applying \cref{thm:SkewBallVolumeAsymptotic} where we take $G$ itself for its subgroup. Thus, we have $N \leq O\bigl(e^{O(D_\psi)}\bigr)\delta^{-(\dim(G) - \dim(H))}$. Set $q := \dim(G)$. Recalling various norms and putting everything together, we get
\begin{align*}
E_2 =
\begin{cases}
O\bigl(e^{O(D_\psi + d(o, g_0o))}\bigr) \|\psi\|_{C^{0, \chi}} \delta^{-(q + \chi)} R^{-\chi\kappa'}, & \rankG = 1 \\
O\bigl(e^{O(D_\psi + d(o, g_0o))}\bigr) \bigl(\|\psi\|_\infty \delta^{-q} T^{-\frac{1}{2}} \log(T)^{\frac{1}{2}} + \|\psi\|_{C^{0, \chi}} \delta^{-(q + \chi)} R^{-\chi\kappa'}\bigr), & \rankG \geq 2.
\end{cases}
\end{align*}

Now, we bound $E_1$ above. In the $\rankG = 1$ case, for all $1 \leq j \leq N$, we use \cref{thm:SkewBallVolumeAsymptotic} with the same notations to calculate that
\begin{align*}
&\frac{\bigl(\mu_H(H_{T + 2E_\psi\delta}[g_0, b_j^{-1}]) - \mu_H(H_{T - 2E_\psi\delta}[g_0, b_j^{-1}])\bigr)e^{-\delta_{2\rho} T}}{\mu_H(H_T)e^{-\delta_{2\rho} T}} \\
={}&\frac{\substack{C[g_0, b_j^{-1}] \bigl(e^{2E_\psi\delta\delta_{2\rho}} - e^{-2E_\psi\delta\delta_{2\rho}}\bigr) + E[g_0, b_j^{-1}] \bigl(e^{-\eta_1 T + 2E_\psi\delta(\delta_{2\rho} - \eta_1)} - e^{-\eta_1 T - 2E_\psi\delta(\delta_{2\rho} - \eta_1)}\bigr)}}{C[e, e] + E[e, e] e^{-\eta_1 T}} \\
\leq{}&O\bigl(e^{O(D_\psi + d(o, g_0o))}\bigr) \cdot \bigl(\delta + e^{-\eta_1 T}\bigr)
\end{align*}
for all $T \gg 0$, where we used $E_\psi = O\bigl(e^{O(D_\psi)}\bigr)$. Similarly, in the $\rankG \geq 2$ case, for all $1 \leq j \leq N$, we get
\begin{align*}
&\frac{\bigl(\mu_H(H_{T + 2E_\psi\delta}[g_0, b_j^{-1}]) - \mu_H(H_{T - 2E_\psi\delta}[g_0, b_j^{-1}])\bigr)T^{-\frac{\rankH - 1}{2}} e^{-\delta_{2\rho}T}}{\mu_H(H_T)T^{-\frac{\rankH - 1}{2}} e^{-\delta_{2\rho}T}} \\
={}&\frac{C[g_0, b_j^{-1}] \Bigl(\bigl(1 + \frac{2E_\psi\delta}{T}\bigr)^{\frac{\rankH - 1}{2}} e^{2E_\psi\delta\delta_{2\rho}} - \bigl(1 - \frac{2E_\psi\delta}{T}\bigr)^{\frac{\rankH - 1}{2}} e^{-2E_\psi\delta\delta_{2\rho}}\Bigr)}{C[e, e] + E[e, e] T^{-\frac{1}{2}} \log(T)^{\frac{1}{2}}} \\
&{}+\frac{\splitfrac{\scriptstyle E[g_0, b_j^{-1}] (T + 2E_\psi\delta)^{-\frac{1}{2}}\log(T + 2E_\psi\delta)^{\frac{1}{2}}\bigl(1 + \frac{2E_\psi\delta}{T}\bigr)^{\frac{\rankH - 1}{2}} e^{2E_\psi\delta\delta_{2\rho}}}{\scriptstyle - E[g_0, b_j^{-1}] (T - 2E_\psi\delta)^{-\frac{1}{2}}\log(T - 2E_\psi\delta)^{\frac{1}{2}}\bigl(1 - \frac{2E_\psi\delta}{T}\bigr)^{\frac{\rankH - 1}{2}} e^{-2E_\psi\delta\delta_{2\rho}}}}{C[e, e] + E[e, e] T^{-\frac{1}{2}} \log(T)^{\frac{1}{2}}} \\
\leq{}&O\bigl(e^{O(D_\psi + d(o, g_0o))}\bigr) \cdot \bigl(\delta + T^{-\frac{1}{2}}\log(T)^{\frac{1}{2}}\bigr)
\end{align*}
for all $T \gg 0$, where we again used $E_\psi = O\bigl(e^{O(D_\psi)}\bigr)$. We also calculate that
\begin{align*}
\sum_{j = 1}^N \int_{\Gamma \backslash G} |\phi_{\delta, j}| \, d\mu_{\Gamma \backslash G} &= \int_G \sum_{j = 1}^N |\Phi_{\delta, j}| \, d\mu_G \\
&= \int_G \sum_{j = 1}^N |\psi(gH)| \varphi_{\delta, j}^{G/H}(gH) \varphi_\delta^H(h_{j, g}) \, d\mu_G(g) \\
&\leq \int_{G/H} |\psi| \, d\mu_{G/H} \cdot \int_H \mathds{1}_{B_1^H(e)} \, d\mu_H \\
&\ll \mu_{G/H}(\supp(\psi))\|\psi\|_\infty \\
&\leq O\bigl(e^{O(D_\psi)}\bigr)\|\psi\|_\infty
\end{align*}
using \cref{eqn:VolumeBoundFor_supp_psi} and in particular, there is no dependence on $N$ and hence on $\delta$. Thus, combining the above bounds, we get
\begin{align*}
E_1 =
\begin{cases}
O\bigl(e^{O(D_\psi + d(o, g_0o))}\bigr) \|\psi\|_\infty \bigl(\delta + e^{-\eta_1 T}\bigr), & \rankG = 1 \\
O\bigl(e^{O(D_\psi + d(o, g_0o))}\bigr) \|\psi\|_\infty \bigl(\delta + T^{-\frac{1}{2}} \log(T)^{\frac{1}{2}}\bigr), & \rankG \geq 2.
\end{cases}
\end{align*}

Worsening $\kappa'$ if necessary so that $\kappa' < \tilde{C}_{\varsigma_H}\eta_1$, we combine the two error terms to get
\begin{align*}
&E_1 + E_2 \\
\leq{}&
\begin{cases}
O\bigl(e^{O(D_\psi + d(o, g_0o))}\bigr) \|\psi\|_{C^{0, \chi}} \bigl(\delta + \delta^{-(q + \chi)} R^{-\chi\kappa'}\bigr), & \rankG = 1 \\
O\bigl(e^{O(D_\psi + d(o, g_0o))}\bigr) \|\psi\|_{C^{0, \chi}} \bigl(\delta + \delta^{-q} T^{-\frac{1}{2}} \log(T)^{\frac{1}{2}} + \delta^{-(q + \chi)} R^{-\chi\kappa'}\bigr), & \rankG \geq 2.
\end{cases}
\end{align*}
Fix $p := q + 1 + \chi$ and $\kappa := \frac{\kappa'}{q + 2}$. In the $\rankG = 1$ case, we choose $\delta$ in the optimal way up to a constant factor by setting $\delta = C^p\delta^{-(q + \chi)} R^{-\chi\kappa'}$ for some $C > 0$. This gives $\delta = CR^{-\chi\frac{\kappa'}{p}}$ and we choose $C = E_\psi^{-1} = e^{-c_\Phi D_\psi} = O\bigl(e^{O(D_\psi)}\bigr)$ so that it satisfies \cref{eqn:deltaCondition}. Then, $\delta^{-(q + \chi)} R^{-\chi\kappa'} = E_\psi^{q + \chi}R^{-\chi\frac{\kappa'}{p}} \leq O\bigl(e^{O(D_\psi)}\bigr)R^{-\chi\frac{\kappa'}{p}}$. Thus, we obtain the desired error term of the proposition after worsening the second factor of the exponent from $\frac{\kappa'}{p}$ to $\kappa$ to remove the dependence on $\chi$. In the $\rankG \geq 2$ case, we slightly worsen the second term by replacing the factor $\delta^{-q}$ with $\delta^{-(q + \chi)}$ and then we choose $\delta$ in the optimal way up to a constant factor by setting $\delta = C^p\delta^{-(q + \chi)} \bigl(T^{-\frac{1}{2}} \log(T)^{\frac{1}{2}} + R^{-\chi\kappa'}\bigr)$ for some $C > 0$. This gives
\begin{align*}
\delta = C\bigl(T^{-\frac{1}{2}} \log(T)^{\frac{1}{2}} + R^{-\chi\kappa'}\bigr)^{\frac{1}{p}} \leq C\max\Bigl\{\bigl(2T^{-\frac{1}{2}} \log(T)^{\frac{1}{2}}\bigr)^{\frac{1}{p}}, \bigl(2R^{-\chi\kappa'}\bigr)^{\frac{1}{p}}\Bigr\}.
\end{align*}
Similar to above, we choose $C = (2E_\psi)^{-1} = 2^{-1}e^{-c_\Phi D_\psi} = O\bigl(e^{O(D_\psi)}\bigr)$ so that it satisfies \cref{eqn:deltaCondition}.
We again worsen $\frac{\kappa'}{p}$ to $\kappa$ for the last term of the full resulting error term. Hence, we obtain the desired error term of the proposition.
\end{proof}

\section{Limiting density}
\label{sec:LimitingDensity}
We first recall from \cite[Subsection 2.5]{GW07} the limiting density $\nu_{y_0}$ associated to each $y_0 \in G/H$. We define them in a slightly different fashion described in loc. cit. which is possible since $H$ is semisimple and hence unimodular. Define the map $\tilde{\alpha}: G \times G \to \R_{\geq 0}$ by
\begin{align*}
\tilde{\alpha}(g_1, g_2) := \lim_{T \to +\infty} \frac{\mu_H(H_T[g_1, g_2^{-1}])}{\mu_H(H_T)} \qquad \text{for all $g_1, g_2 \in G$}.
\end{align*}
Observe that $\tilde{\alpha}$ is right $H$-invariant in both arguments since $\mu_H$ is bi-$H$-invariant. Thus, $\tilde{\alpha}$ descends to a map $\alpha: G/H \times G/H \to \R_{\geq 0}$. For all $y_0 \in G/H$, we define a measure $\nu_{y_0}$ on $G/H$ by
\begin{align*}
d\nu_{y_0} = \alpha(\cdot, y_0) \, d\mu_{G/H}.
\end{align*}

In this section we prove the following theorem, effectivizing \cite[Theorem 2.3]{GW07}. Its proof is also similar to that of \cite[Theorem 2.3]{GW07} in \cite[Section 5]{GW07} but we use \cref{cor:RiemannianVolumeRatioAsymptotic}. Recall the constant $D_\psi > 0$ associated to each $\psi \in C_{\mathrm{c}}(G/H)$ from \cref{eqn:Dpsi}, and the constant $\eta_1 > 0$ from \cref{eqn:Eta1}.

\begin{theorem}
\label{thm:LimitingDensity}
Let $\psi \in C_{\mathrm{c}}(G/H)$. Let $g_0 \in G$ and $y_0 = g_0H \in G/H$. Then, we have
\begin{multline*}
\left|\frac{1}{\mu_H(H_T)} \int_{G_T} \psi(g y_0) \, d\mu_G(g) - \int_{G/H} \psi \, d\nu_{y_0}\right| \\
\leq
\begin{cases}
O\bigl(e^{D_\psi + d(o, g_0o)}\bigr) \|\psi\|_\infty e^{-\eta_1 T}, & \rankG = 1 \\
O\bigl(e^{D_\psi + d(o, g_0o)}\bigr) \|\psi\|_\infty T^{-\frac{1}{2}}\log(T)^{\frac{1}{2}}, & \rankG \geq 2
\end{cases}
\end{multline*}
for all $T \geq \Omega(D_\psi + d(o, g_0o))$ in the $\rankG = 1$ case and for all $T \geq \Omega\bigl(e^{\Omega(D_\psi + d(o, g_0o))}\bigr)$ in the $\rankG \geq 2$ case.
\end{theorem}

\begin{proof}
Let $\psi$, $y_0 = g_0H$, and $T$ be as in the theorem. We have
\begin{align*}
\int_{G_T} \psi(gy_0) \, \mu_G(g) &= \int_{\{g \in G: d(o, gg_0^{-1}o) < T\}} \psi(g H) \, d\mu_G(g) \\
&= \int_{G/H} \int_{\{h \in H: d(o, ghg_0^{-1}o) < T\}} \psi(gH) \, d\mu_H(h) \, d\mu_{G/H}(gH) \\
&= \int_{G/H} \psi(gH) \cdot \mu_H(H_T[g, g_0^{-1}]) \, d\mu_{G/H}(gH).
\end{align*}
Now, we use the precise asymptotic formulas for the volume of Riemannian skew balls together with the above calculation. Let $E[g, g_0^{-1}] = O\bigl(e^{O(d(o, go) + d(o, g_0o))}\bigr)$ be the constant provided by \cref{cor:RiemannianVolumeRatioAsymptotic} which is continuous and hence measurable in $g \in G$. Then using \cref{cor:RiemannianVolumeRatioAsymptotic}, we get
\begin{align*}
&\left|\frac{1}{\mu_H(H_T)}\int_{G_T} \psi(gy_0) \, \mu_G(g) - \int_{G/H} \psi \, d\nu_{y_0}\right| \\
={}&\left|\int_{G/H} \psi(gH) \cdot \left(\frac{\mu_H(H_T[g, g_0^{-1}])}{\mu_H(H_T)} - \tilde{\alpha}(g, g_0)\right) \, d\mu_{G/H}(gH)\right| \\
\leq{}&\int_{G/H} |\psi(gH)| \cdot \left|\frac{\mu_H(H_T[g, g_0^{-1}])}{\mu_H(H_T)} - \tilde{\alpha}(g, g_0)\right| \, d\mu_{G/H}(gH) \\
\leq{}&
\begin{cases}
\int_{G/H} |\psi(gH)| \cdot E[g, g_0^{-1}] e^{-\eta_1 T} \, d\mu_{G/H}(gH), & \rankG = 1 \\
\int_{G/H} |\psi(gH)| \cdot E[g, g_0^{-1}] T^{-\frac{1}{2}}\log(T)^{\frac{1}{2}} \, d\mu_{G/H}(gH), & \rankG \geq 2
\end{cases}
\end{align*}
where we use the optimal lifts $g \in G$ for each $gH$ in the sense that $d(o, go) = \min\{d(o, g'o): g'H = gH\}$. Since $\psi$ is compactly supported, the theorem follows by using \cref{eqn:VolumeBoundFor_supp_psi}.
\end{proof}

\section{General theorems: full versions}
\label{sec:GeneralTheorems}
We finish the paper with this section where we state our general theorems with full details. Recall $c_{\varsigma_H}$ from \ShahEE where $\varsigma_H = \frac{1}{2}\min_{\alpha \in \Phi_H^+} \alpha(v_{2\rho_H})$.

\begin{theorem}
\label{thm:MainTheorem}
Suppose \ShahEE holds. There exist $C_{\varsigma_H} \asymp_{G, H} c_{\varsigma_H}$, $\kappa \in (0, \kappa_0)$, and $\varrho \in (\varrho_0, 2\varrho_0)$ such that the following holds. Let $\chi \in (0, 1]$. Fix $p := \dim(G) + 1 + \chi$. Let $\psi \in C_{\mathrm{c}}^{0, \chi}(G/H)$, $g_0 \in G$, $x_0 = \Gamma g_0 \in \Gamma \backslash G$, and $y_0 = g_0 H \in G/H$. There exists $M_{\psi, g_0} > 0$ such that for all $R \gg_{G, \Gamma} \inj_{\Gamma \backslash G}(x_0)^{-\varrho}$ and $T \geq C_{\varsigma_H} \log(R) + M_{\psi, g_0}$, at least one of the following holds.
\begin{enumerate}
\item We have
\begin{align*}
&\displaystyle\left|\frac{1}{\mu_H(H_T)} \sum_{\gamma \in \Gamma_T} \psi(\gamma y_0) - \int_{G/H} \psi \, d\hat{\nu}_{y_0} \right| \\
\leq{}&
\begin{cases}
O\bigl(e^{O(D_\psi + d(o, g_0o))}\bigr) \|\psi\|_{C^{0, \chi}} R^{-\chi\kappa}, & \rankG = 1 \\
O\bigl(e^{O(D_\psi + d(o, g_0o))}\bigr) \|\psi\|_{C^{0, \chi}} \bigl(T^{-\frac{1}{2p}}\log(T)^{\frac{1}{2p}} + R^{-\chi\kappa}\bigr), & \rankG \geq 2.
\end{cases}
\end{align*}
\item There exists $x \in \Gamma \backslash G$ with
\begin{align*}
d(x_0, x) \leq R^{C_{\varsigma_H}} T^{C_{\varsigma_H}} e^{-\varsigma_HT}
\end{align*}
such that $xH$ is periodic with $\vol(xH) \leq R$.
\end{enumerate}
Moreover, we can choose $M_{\psi, g_0} = C_2e^{C_1(D_\psi + d(o, g_0o))}$ for some constants $C_1, C_2 > 0$. These and the implicit constants depend only on $(G, H)$.
\end{theorem}

\begin{proof}
The theorem follows precisely by combining \cref{thm:EffectiveDuality,thm:LimitingDensity}. This is really a combination of \cref{thm:K_EffectiveEquidistribution,pro:SkewballEffectiveEquidistribution,pro:EffectiveDuality',thm:LimitingDensity} but the work to put the propositions in dichotomy form was already done in their respective sections.
\end{proof}

\begin{remark}
Observe that the dichotomy in the error term is according to the rank of $G$ and not $H$. In the full error term above for the $\rankG \geq 2$ case, the source of the first term is the error term in the formula for the ratio of the volume of Riemannian skew balls in \cref{cor:RiemannianVolumeRatioAsymptotic} and is the reason that the full error term is sensitive to the rank of $G$. The source of the second term is the error term in \ShahEE. In the $\rankG = 1$ case, writing both terms is redundant.
\end{remark}

\begin{remark}
For the optimal error term above (which is only a slight improvement), the factor $\log(T)^{\frac{1}{2p}}$ is to be replaced with $W(\delta_{2\rho}T^{\rankH + 2})^{\frac{1}{2p}}$ using the Lambert $W$ function. See \cref{rem:OptimalErrorTerm}.
\end{remark}

\begin{remark}
The theorem can also be formulated using well-approximable vectors associated to $G/H$. We will give a detailed discussion on this in a sequel paper. 
\end{remark}

The theorem below is a corollary of the theorem above. The derivation is simply by taking $R$ and $T$ sufficiently large and, in the $\rankG \geq 2$ case, also worsening the exponent $-\frac{1}{2p}$ in the error term in the theorem below to $-\frac{1 - \varepsilon}{2p}$ to eliminate the factor $\log(T)^{\frac{1}{2p}}$ and more importantly, the constant coefficient associated to the base point $y_0 = g_0H$.

\begin{theorem}
\label{thm:MainTheoremEpsilon}
Suppose \ShahEE holds. There exist $C_{\varsigma_H} \asymp_{G, H} c_{\varsigma_H}$, $\kappa \in (0, \kappa_0)$, and $\varrho \in (\varrho_0, 2\varrho_0)$ such that the following holds. Let $\varepsilon \in (0, 1)$, $\chi \in (0, 1]$, and $p := \dim(G) + 1 + \chi$. Let $\psi \in C_{\mathrm{c}}^{0, \chi}(G/H)$, $g_0 \in G$, $x_0 = \Gamma g_0 \in \Gamma \backslash G$, and $y_0 = g_0 H \in G/H$. There exist $L_{g_0, \chi} > 0$ and $M_{\psi, g_0, \varepsilon} > 0$ such that for all $R \gg_{G, \Gamma} \inj_{\Gamma \backslash G}(x_0)^{-\varrho} + L_{g_0, \chi}$ and $T \geq C_{\varsigma_H} \log(R) + M_{\psi, g_0, \varepsilon}$, at least one of the following holds.
\begin{enumerate}
\item We have:
\begin{enumerate}
\item if $\rankG = 1$, then
\begin{align*}
\left|\frac{1}{\mu_H(H_T)} \sum_{\gamma \in \Gamma_T} \psi(\gamma y_0) - \int_{G/H} \psi \, d\hat{\nu}_{y_0} \right| \leq O\bigl(e^{O(D_\psi)}\bigr) \|\psi\|_{C^{0, \chi}} R^{-\chi\kappa};
\end{align*}
\item if $\rankG \geq 2$, then
\begin{align*}
\left|\frac{1}{\mu_H(H_T)} \sum_{\gamma \in \Gamma_T} \psi(\gamma y_0) - \int_{G/H} \psi \, d\hat{\nu}_{y_0} \right| \leq O\bigl(e^{O(D_\psi)}\bigr) \|\psi\|_{C^{0, \chi}} \bigl(T^{-\frac{1 - \varepsilon}{2p}} + R^{-\chi\kappa}\bigr).
\end{align*}
\end{enumerate}
\item There exists $x \in \Gamma \backslash G$ with
\begin{align*}
d(x_0, x) \leq R^{C_{\varsigma_H}} T^{C_{\varsigma_H}} e^{-\varsigma_HT}
\end{align*}
such that $xH$ is periodic with $\vol(xH) \leq R$.
\end{enumerate}
Moreover, we can choose $L_{g_0, \chi} = e^{\frac{C_1}{\chi\kappa}d(o, g_0o)}$, and $M_{\psi, g_0, \varepsilon} = C_3e^{C_2(D_\psi + d(o, g_0o))}$ if $\rankG = 1$ and $M_{\psi, g_0, \varepsilon} = \frac{C_3}{\varepsilon^2}e^{\frac{C_2}{\varepsilon}(D_\psi + d(o, g_0o))}$ if $\rankG \geq 2$, for some constants $C_1, C_2, C_3 > 0$. These and the implicit constants depend only on $(G, H)$.
\end{theorem}

\begin{remark}
Observe that the constant $L_{g_0, \chi}$ associated to the Diophantine condition depends on the basepoint $y_0 = g_0H$ but not on $\psi$.
\end{remark}

\nocite{*}
\bibliographystyle{alpha_name-year-title}
\bibliography{References}
\end{document}